\documentclass
{amsart}

\usepackage{mathtools}
\usepackage{amsmath}
\usepackage{amsthm}
\usepackage{amssymb}
\usepackage{mathabx}
\usepackage{dsfont}
\usepackage[english]{babel}
\usepackage[utf8x]{inputenc}
\usepackage{overpic}
\usepackage{enumerate}
\usepackage{pdflscape}
\usepackage{float}
\usepackage{hyperref}
\usepackage[all,cmtip]{xy}
\usepackage{todonotes}

\newtheorem{thm}{Theorem}[section]
\newtheorem{lem}[thm]{Lemma}			

\newtheorem{proposition}[thm]{Proposition}
\newtheorem{claim}[thm]{Claim}

\newtheorem*{qn*}{Question}
\newtheorem*{lem*}{Lemma}
\usepackage{thmtools}
\usepackage{thm-restate}

 \newtheorem{thm*}{Theorem}

\newtheorem{dfn}[thm]{Definition}

\newtheorem{rmk}[thm]{Remark}

\newcommand{\ep}{
\epsilon
}

\newcommand{\mc}[1]{
\mathcal{#1}
}
\newcommand{\mb}[1]{
\mathbb{#1}
}

\newcommand{\calD} {\ensuremath {\mathcal{D}}}

\def\epsilon{\varepsilon}

\newcommand{\nest}{\sqsubseteq}
\newcommand{\propnest}{\sqsubsetneq}
\newcommand{\orth}{\bot}
\newcommand{\transverse}{\pitchfork}

\begin{document}

\title[Hyperbolization, length bounds, Heegaard splittings]{Effective hyperbolization and length bounds for Heegaard splittings}

\author[Feller]{Peter Feller}
\address{Department of Mathematics, ETH Z\"{u}rich, Switzerland}
\email{peter.feller@math.ch}

\author[Sisto]{Alessandro Sisto}
	\address{Department of Mathematics, Heriot-Watt University, Edinburgh, UK}
	\email{a.sisto@hw.ac.uk}

\author[Viaggi]{Gabriele Viaggi}
\address{Department of Mathematics, Sapienza University, Rome, Italy}
\email{gabriele.viaggi@uniroma1.it}

\thanks{}
\keywords{}

\begin{abstract}
We consider 3-manifolds given as Heegaard splittings $M=H^-\cup_\Sigma H^+$ with the aim to describe the hyperbolic metric of $M$ under topological conditions on the splitting guaranteeing that the manifold is hyperbolic. In particular, given a suitable `sufficiently incompressible'' curve $\gamma\subset\Sigma$, we show (without appealing to Geometrization) that $M$ is hyperbolic and we compute the length of $\gamma$ in terms of the projection coefficients of the disk sets, up to a uniform multiplicative error.
\end{abstract}

\maketitle
\tableofcontents

\section{Introduction}

The geometry and topology of a closed orientable 3-manifold $M$ is better understood by considering the surfaces it contains. Classical work of Kneser and Milnor \cite{Kneser,Milnor} and Jaco, Shalen, and Johansson \cite{JSJ78,Joh79} shows how to canonically decompose every 3-manifold by cutting it along spheres and tori into elementary building blocks. The breakthroughs of Thurston \cite{Thu82}, and the solution of the Geometrization Conjecture by Perelman, further allow us to equip each such piece with one out of eight locally homogeneous Riemannian structures. Among those, hyperbolic geometry is the most elusive one.

In many setups the topology of $M$ is completely determined by a single surface $\Sigma\subset M$ of genus at least 2 together with some gluing data and it is natural to ask how to read off the geometric decomposition of $M$ from this data.   

A celebrated example is the case of 3-manifolds fibering over the circle $M\to S^1$, also called {\em mapping tori}. In this case the fiber $\Sigma\subset M$ together with the monodromy map $f:\Sigma\to\Sigma$ entirely determine the manifold. Thurston \cite{ThurstonII} shows that $M$ admits a hyperbolic metric if and only if $f$ is a so-called {\em pseudo-Anosov} homeomorphism. Such a map has two naturally associated transverse laminations  $\lambda^-,\lambda^+$ on $\Sigma$ and groundbreaking work of Minsky \cite{M10} and Brock, Canary, and Minsky \cite{BrockCanaryMinsky:ELC2} shows how to convert the purely topological data of the two laminations into a geometric control on the geometry of $M$. In particular, we have the following.
\begin{enumerate}
    \item The set of geodesics of $M$ which are shorter than a Margulis constant $\ep_\Sigma>0$ can be read off in terms of the {subsurface projections} of $\lambda^-,\lambda^+$.
    \item{The lengths of those curves can be computed explicitly up to multiplicative error in terms of the subsurface projections of $\lambda^-,\lambda^+$.}
\end{enumerate}

Motivated by these results, in this paper we aim to provide (partial) answers to the same problems on {hyperbolicity} and {length bounds} for a different family of manifolds, namely {\em Heegaard splittings}, as also suggested in \cite[Sec.~6]{Minsky_ICM_06}. Recall that {\em every} closed oriented connected 3-manifold $M$ contains a closed orientable {\em Heegaard surface} $\Sigma\subset M$ dividing it into two handlebodies $M-\Sigma=H^-\sqcup H^+$. Within this setup it is natural to consider the so-called {\em disk sets} $\mc{D}^-,\mc{D}^+$ associated with the splitting. These are the sets of isotopy classes of essential simple closed curves on $\Sigma$ that bound properly embedded disks in $H^-$ and $H^+$, respectively. Informally our main result is the following.

\par\medskip

\begin{thm*}[Hyperbolization and Length Bound, Informal Version]
Consider a Heegaard splitting $M=H^-\cup_\Sigma H^+$ of genus at least 2 of a closed orientable 3-manifold $M$. Let $\gamma\subset\Sigma$ be an essential simple closed curve that is sufficiently incompressible. Suppose that there is a large subsurface projection between $\mc{D}^-,\mc{D}^+$ on a subsurface of $\Sigma$ with $\gamma$ in its boundary. Then: 
\begin{enumerate}
\item{$M$ has a hyperbolic metric for which $\gamma\subset\Sigma$ is a geodesic.}
\item{There is an explicit formula in the subsurface projections of $\mc{D}^-,\mc{D}^+$ to the subsurfaces of $\Sigma$ that contain $\gamma$ in their boundary that coarsely computes the length of $\gamma\subset M$.}
\end{enumerate}
\end{thm*}

\par\medskip

In order to give a more precise statement we need to briefly discuss the {\em curve graph} of a finite type surface and the tool of {\em subsurface projections} (as introduced in \cite{MasurMinsky:II}). First, we recall that if $W$ is a connected orientable surface of finite type with negative Euler characteristic (called \emph{non-annular surface} below),
then, apart from some exceptional cases that for the sake of exposition we ignore,
the curve graph $\mc{C}(W)$ is the graph whose vertices are the isotopy classes of essential non-peripheral simple closed curves of $W$  and properly embedded open intervals of $W$ and where $\alpha,\beta$ are connected by an edge of length 1 if they can be realized disjointly. By seminal work of Masur and Minsky \cite{MasurMinsky:I,MasurMinsky:II}, this is an infinite diameter Gromov hyperbolic space with a rich combinatorial and geometric structure reflecting the complex nesting and disjointness relations between subsurfaces of $\Sigma$. Its relevance to the geometry of hyperbolic 3-manifolds has been uncovered by Minsky in a sequence of groundbreaking works \cite{M00,M01} leading to the solution of the Ending Lamination Conjecture by Minsky \cite{M10} and Brock, Canary, and Minsky \cite{BrockCanaryMinsky:ELC2}.  

Secondly, we recall subsurface projections. If $W\subset\Sigma$ is a proper essential non-annular subsurface and $\alpha,\beta\subset\Sigma$ intersect essentially $\partial W$, then the components of $\alpha\cap W$ and $\beta\cap W$ define two subsets $\pi_W(\alpha),\pi_W(\beta)$ of $\mc{C}(W)$ and we set $d_W(\alpha,\beta):={\rm diam}_{\mc{C}(W)}(\pi_W(\alpha)\cup\pi_W(\beta))$. For annular subsurfaces $W\subset\Sigma$ we need a slightly more involved definition for the subsurface projection. We denote by $d_\gamma(\alpha,\beta)$ the coefficient associated with the projections of $\alpha,\beta$ to an annulus whose core curve is isotopic to $\gamma$; see Definition \ref{def:annular coeff}.

We are now able to state the precise version of the main result. For an essential simple closed curve $\gamma\subset\Sigma$ define
\[
S_\gamma(\mc{D}^-,\mc{D}^+):=1+\sum_{W\in\mc{Y}_\gamma}{\{\{d_W(\mc{D}^-,\mc{D}^+)\}\}_K},
\] 
where $\mc{Y}_\gamma$ is the set of all essential non-annular subsurfaces $W\subset \Sigma$ with $\gamma\subset\partial W$, and $\{\{\bullet\}\}_K$ is the function $\{\{x\}\}=x$ if $x>K$ and $0$ otherwise. 

\begin{thm*}[Hyperbolization and Length Bound]
\label{thm:main}
There exist $R,c>0$ depending on $\Sigma$ such that the following holds. Let $M=H^-\cup_\Sigma H^+$ be a Heegaard splitting where $\Sigma$ is a closed orientable surface of genus at least 2. Let $\gamma$ be an essential curve on $\Sigma$ such that $(H^-,\gamma)$ and $(H^+,\gamma)$ are {\rm pared acylindrical} and either
\begin{enumerate}
    \item{$S_\gamma=S_\gamma(\mc{D}^+,\mc{D}^-)\ge R$ or}
    \item{$d_\gamma=d_\gamma(\mc{D}^+,\mc{D}^-)\geq R$.}
\end{enumerate}  
Then $M$ carries a hyperbolic metric for which $\gamma$ is a geodesic
and $\gamma$'s length satisfies
\[
\frac{1}{c}\frac{S_\gamma}{S_\gamma^2+d_\gamma^2}\leq \ell_M(\gamma)\leq c \frac{S_\gamma}{S_\gamma^2+d_\gamma^2}.
\]
\end{thm*}

Of the various features of Theorem \ref{thm:main}, a central novelty is the length bound on $\gamma$, which converts the purely topological information given by $S_\gamma,d_\gamma$ into a concrete formula for the length of $\gamma$ in the (unique) hyperbolic metric on $M$. Observe that, coarsely speaking, it is the best possible bound.

The requirement that $(H^-,\gamma),(H^+,\gamma)$ are both pared acylindrical essentially means that $H^\pm-\gamma$ does not contain any non-trivial essential disks and annuli. Such a condition appears naturally in the study of deformations of hyperbolic metrics on $H^\pm$ with cusps. Let us mention that if $d_{\mc{C}(\Sigma)}(\gamma,\mc{D}^-\cup\mc{D}^+)\ge 3$, then the pared acylindrical assumption is automatically satisfied.

Conjecturally, an essential simple closed curve $\gamma\subset\Sigma$ (without additional assumptions) is isotopic to a short geodesic in $M$ if and only if there is a large projection of $\mc{D}^-,\mc{D}^+$ to an essential subsurface $W\subset\Sigma$ containing $\gamma$ and the length can be coarsely computed from the projections; compare with~\cite[Sec.~6]{Minsky_ICM_06}.

Our construction of the hyperbolic metric only employs Thurston's Hyperbolization for Haken manifolds and the Universal Dehn Filling Theorem. Its explicit nature is the reason why we are able to get such an effective control on the length of~$\gamma$. 

We provide some more context for our result around the problem of hyperbolicity of Heegaard splittings. 
Hempel established that, if $d_{\mc{C}}(\mc{D}^-,\mc{D}^+)\ge 3$, then $M$ is irreducible, atoroidal, and non-Seifert fibered~ \cite{He01}, which by the Geometrization Theorem implies that $M$ admits a hyperbolic metric. The above criterion was used by Maher~\cite{Maher:Heegaard} to show that a random Heegaard splitting of fixed genus is hyperbolic answering a conjecture of Dunfield and Thurston \cite{DT06} (see also Problem 24 of \cite{Thu82}). We developed a different more explicit approach to the hyperbolization of random 3-manifolds in~\cite{HV} and~\cite{FSVa} obtaining as a byproduct a control on the geometry of those objects. More recently, Hamenstädt and Jäckel~\cite{HJ} proved that every splitting with large $d_{\mc{C}}(\mc{D}^-,\mc{D}^+)$ admits a hyperbolic metric similar to the one constructed in~\cite{HV}. Their result also bypasses the Geometrization Theorem, but they do not obtain length bounds.

Lastly, Theorem \ref{thm:main} allows us to understand the quantitative effect on the geometry of the manifold $M$ by changing the gluing map $\partial H^-\to\partial H^+$ by a (large) power of a of partial pseudo-Anosovs supported on a proper non-annular subsurface $W\subset\Sigma$ or a Dehn twist around a simple closed curve $\gamma\subset\Sigma$. In fact the first operation increases $d_W(\mc{D}^-,\mc{D}^+)$ (linearly in the power of the partial pseudo-Anosov) while the second one increases $d_\gamma(\mc{D}^-,\mc{D}^+)$ (linearly in the power of the twist).

\subsection*{Acknowledgements}
PF acknowledges financial support of the SNSF Grant 181199. GV acknowledges financial support of the DFG 427903332 and of the DFG 390900948.

\section{About strategy and proofs}
\label{sec:strategy}

We first describe in this section the global strategy with its milestones and the ingredients of the proofs. Along the way, we point to the sections where the various problems we have to address are solved.

\subsection*{Structure}
The proof of Theorem \ref{thm:main} is carried out in this section modulo three other results, namely Theorems \ref{thm:non-annular}, \ref{thm:annular}, and \ref{thm:annular2}. 

The proof of Theorem \ref{thm:non-annular} will take place in Section \ref{sec:upper_bound} building upon the work done in Section \ref{sec:disk_proj}.

The proofs of Theorems \ref{thm:annular} and \ref{thm:annular2} will be discussed in Section \ref{sec:annuli}. They rely on the material of Sections \ref{sec:move_seq} and \ref{sec:good_homot}.

\subsection*{Convention}
Throughout this section, we write $x\lesssim y$ (resp. $x\gtrsim y$) to indicate that there exist constants $a,b>0$, only depending on the genus of $\Sigma$ such that $x\le ay+b$ (resp. $x\ge y/a-b$). We write $x\simeq y$ if both $x\lesssim y$ and $x\gtrsim y$ hold.

\subsection*{On the strategy}
The proof strategy of the main theorem is to first consider the drilled manifold $M-\gamma$, which is hyperbolic as we discuss below, and then to understand its geometry sufficiently well to control the filling from $M-\gamma$ to $M$. The technical tool to perform such an operation is the Universal Dehn Filling Theorem of Hodgson and Kerckhoff \cite{HK:universal}, which we recall. We need two preliminary definitions.

\begin{dfn}[Standard Margulis Neighborhood]
Let $\mb{M}$ be a complete finite volume orientable hyperbolic 3-manifold with cusps. Given any choice of Margulis constant $\ep$, each cusp $\gamma$ of $\mb{M}$ has a {\rm standard $\ep$-Margulis neighborhood} $\mb{T}_{\ep}(\gamma)\subset\mb{M}$. The (open) subset $\mb{T}_\ep(\gamma)$ is a connected component of the $\ep$-thin part 
\[
\mb{M}_\ep^{{\rm thin}}:=\{x\in\mb{M}\left|\,{\rm inj}_{\mb{M}}(x)<\ep\right.\}
\]
and is isometric to a quotient
\[
\mb{T}_\ep(\gamma)=\{(z,t)\in\mb{C}\times(0,\infty)\left|\,t\ge E\right.\}/\tau_\alpha\mb{Z}\oplus\tau_\beta\mb{Z}
\]
where we are working in the upper half space model of the hyperbolic 3-space $\mb{H}^3=\mb{C}\times(0,\infty)$ and for every $c\in \mb{C}$ the transformation $\tau_c\in{\rm Isom}^+(\mb{H}^3)$ is the parabolic isometry $\tau_c(z,t)=(z+c,t)$ fixing $\infty$.

The boundary $T_\gamma:=\partial\mb{T}_{\ep}(\gamma)\subset\mb{M}$ is a 2-dimensional flat torus isometric to
\[
T_\ep(\gamma)=\{(z,t)\in\mb{C}\times(0,\infty)\left|\,t= E\right.\}/\tau_\alpha\mb{Z}\oplus\tau_\beta\mb{Z}.
\]    
\end{dfn}

\begin{dfn}[Normalized Length]
Let $\mb{M}$ be a complete finite volume hyperbolic 3-manifold with a cusp with standard $\ep$-Margulis neighborhood $\mb{T}$ and boundary torus $T$. Let $\mu\subset T$ be a simple closed flat geodesic. The {\em normalized length} of $\mu$, is the quantity defined by 
\[
L(\mu):=\frac{{\rm Length}(\mu)}{\sqrt{{\rm Area}(T)}}.
\]
\end{dfn}

It is easy to check that $L(\mu)$ does not depend on the choice of the Margulis constant $\ep$.

We have the following (see \cite[Theorem 1.1]{HK:universal} and \cite[Corollary 6.13]{FPS22}).

\begin{thm}[Universal Dehn Filling]\label{thm:filling}
Let $\gamma\subset M$ be a knot in a 3-manifold $M$. Suppose that $M-\gamma$ has a complete finite volume hyperbolic metric for which the normalized length $L$ of the meridian of $\gamma$ is at least $L\ge 7.823$. Then $M$ has a complete hyperbolic metric for which $\gamma$ is a geodesic and 
\[
\frac{2\pi}{L^2+16.17}<\ell_M(\gamma)<\frac{2\pi}{L^2-28.78}.
\]
\end{thm}

For us, hyperbolicity of $M-\gamma$ is guaranteed by the topological assumptions on $\gamma$ and is obtained via Thurston's Hyperbolization.

\begin{dfn}[Pared Acylindrical]
Let $H$ be a handlebody with boundary $\Sigma=\partial H$. Let $\gamma\subset\Sigma$ be an essential multicurve. Denote by $A$ a tubular neighborhood of $\gamma$ in $\Sigma$. We say that $(H,\gamma)$ is {\em pared acylindrical} if the following holds:
\begin{enumerate}
\item{$\gamma$ is $\pi_1$-injective in $H$.}
\item{Every $\pi_1$-injective map $(A,\partial A)\rightarrow (H,A)$ is homotopic as a map of pairs into $A$.}
\item{Every $\pi_1$-injective map $(A,\partial A)\rightarrow (H,\Sigma-A)$ is homotopic as a map of pairs into $\Sigma-A$.}
\end{enumerate}
\end{dfn}

The above is a strong topological assumption. We state two consequences, one algebraic and one geometric. 

\begin{lem}\label{lem:graph_of_groups}
Let $M=H^-\cup_\Sigma H^+$ be a Heegaard splitting where $\Sigma$ is a closed orientable surface of genus at least 2. Let $\gamma$ be an essential simple closed curve on $\Sigma$ such that $(H^-,\gamma)$ and $(H^+,\gamma)$ are pared acylindrical. Then $\pi_1(M)$ split as a graph of groups with two vertices connected by: 
\begin{itemize}
    \item{Two edges if $\Sigma-\gamma=\Sigma_1\sqcup\Sigma_2$ has two components.}
    \item{One edge if $\Sigma-\gamma$ is connected.}
\end{itemize}, 
The vertex groups are $\pi_1(H^-),\pi_1(H^+)$, while the edge groups are:
\begin{itemize}
    \item{$\pi_1(\Sigma_1),\pi_1(\Sigma_2)$ in the first case.}
    \item{$\pi_1(\Sigma-\gamma)$ in the second case.}
\end{itemize}
In both cases, the inclusions in the vertex groups are induced by the inclusions of $\Sigma-\gamma$ in $H^-,H^+$ respectively.

In particular, the inclusion $\Sigma-\gamma\to M-\gamma$ is $\pi_1$-injective.
\end{lem} 

\begin{proof}
This follows from Seifert–van Kampen’s Theorem.
\end{proof}

Lemma \ref{lem:graph_of_groups} will be needed to study some natural coverings of $M-\gamma$.

\begin{thm}
\label{thm:hyperbolicity}
Let $M=H^-\cup_\Sigma H^+$ be a Heegaard splitting where $\Sigma$ is a closed orientable surface of genus at least 2. Let $\gamma$ be an essential simple closed curve on $\Sigma$ such that $(H^-,\gamma)$ and $(H^+,\gamma)$ are pared acylindrical. Then $M-\gamma$ has a complete finite volume hyperbolic metric.
\end{thm}

\begin{proof}
To check the hypotheses of Thurston's Hyperbolization see the proofs of \cite[Lemmas 3.3-3.4]{FSVa}, which handle the case where $\gamma$ is non-separating curve, but only use incompressibility of $\Sigma-\gamma$, given in that context by \cite[Lemma 3.2]{FSVa}, and which in our context follows from Lemma \ref{lem:graph_of_groups}.
\end{proof}

Now that we addressed hyperbolicity of $M-\gamma$, we go back to our strategy overview.

We recall that there are essentially two quantities that play a role in the game. First, there is the area of the boundary $T$ of a Margulis neighborhood of the cusp of $M-\gamma$, which is a 2-dimensional flat torus. Second, we have the length of the flat geodesic $\mu\subset T$ on this torus that represents the homotopy class of the canonical meridian of $M$. By the Universal Dehn Filling Theorem \cite{HK:universal}, the length of $\gamma$ in $M$ will be comparable to the ratio ${\rm Area}(T)/{\rm Length}(\mu)^2$.

All our work is devoted to getting some control over these quantities. 

\subsubsection*{Area bound}
We prove that the torus $T$ has a particular shape.
\begin{figure}[h]
\begin{overpic}[scale=1.8]{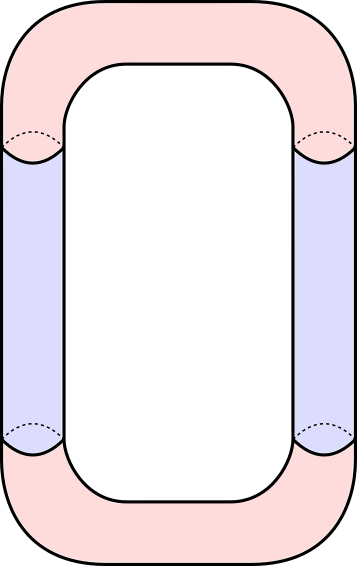}
\end{overpic}
\caption{Decomposition of the flat torus $T$ into a vertical and horizontal part. 
}
\label{fig:torus1}
\end{figure}
More precisely and quantitatively, for some universal constant $B$ the torus $T$ can be decomposed in two vertical and two horizontal annuli, illustrated in Figure~\ref{fig:torus1}, as follows.
\begin{itemize}
    \item{The two vertical annuli are both isometric to a cylinder with base circle of uniformly bounded length contained in the interval $[1/B,B]$ and height in the interval $[S_\gamma/B,BS_\gamma]$.}
    \item{The two horizontal annuli are both isometric to a cylinder with uniformly bounded base and height contained in the interval $[1/B,B]$.}
\end{itemize} 

In particular 
\[
{\rm Area}(T)\simeq S_\gamma
\]
and the length of any $\mu\subset T$ that crosses essentially the four pieces is bounded by 
\[
{\rm Length}(\mu)\gtrsim S_\gamma.
\]

This already produces the (weaker) upper bound
\[
\frac{{\rm Area}(T)}{{\rm Length}(\mu)^2}\lesssim\frac{1}{S_\gamma}
\]
and, hence, the following.

\begin{restatable}{thm}{nonannular}
\label{thm:non-annular}
There exists $c>0$ depending on $\Sigma$ such that the following holds. Let $M=H^-\cup_\Sigma H^+$ be a Heegaard splitting where $\Sigma$ is a closed orientable surface of genus at least 2. Let $\gamma$ be an essential curve on $\Sigma$ such that $(H^-,\gamma)$ and $(H^+,\gamma)$ are pared acylindrical. Then in the hyperbolic metric on $M-\gamma$ the normalized length of the standard meridian $\mu$ determined by $M$ is bounded from below as follows:
\[
L(\mu)^2\ge cS_\gamma.
\]
\end{restatable}

Note that in the area estimate $d_\gamma$ does not appear. It only enters the picture in the length upper bound where it plays the role of a spiraling parameter for $\mu$ controlling how many times $\mu$ winds around the base of the annular pieces. 

\subsubsection*{Length bound}
We prove that, in relation to the decomposition of $T$ mentioned above, the canonical meridian $\mu$ of $M$ always admits a representative that looks like the curve in Figure \ref{fig:torus2}. 
\begin{figure}[h]
\begin{overpic}[scale=1.8]{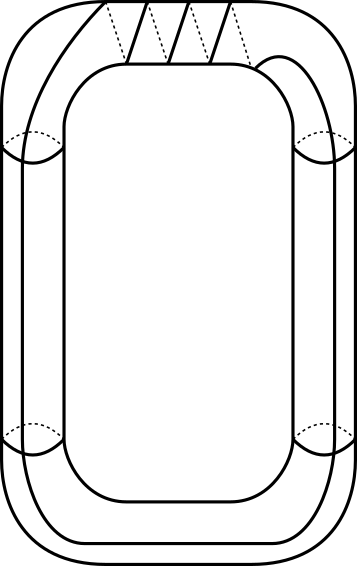}
\end{overpic}
\caption{The model for the meridian $\mu$.}
\label{fig:torus2}
\end{figure}
More precisely, the meridian consists of the following segments.
\begin{itemize}
    \item{It has two (flat) geodesic vertical segments going straight up without spiraling around the base of the vertical annuli of length in the interval $[S_\gamma/B,BS_\gamma]$.}
    \item{It has two (flat) geodesic horizontal segments one of which is spiraling around the base of the top horizontal annulus and has length in the interval $[d_\gamma/B,Bd_\gamma]$ while the other one is going straight without spiraling through the bottom horizontal annulus and has length uniformly bounded in the interval $[1/B,B]$.}
\end{itemize}

In particular
\[
{\rm Length}(\mu)\lesssim S_\gamma+d_\gamma.
\]

Hence, we obtain
\[
\frac{{\rm Area}(T)}{{\rm Length}(\mu)^2}\gtrsim\frac{S_\gamma}{(S_\gamma+d_\gamma)^2}.
\]

\begin{restatable}{thm}{annular}
\label{thm:annular}
There exists $c>0$ depending on $\Sigma$ such that the following holds. Let $M=H^-\cup_\Sigma H^+$ be a Heegaard splitting where $\Sigma$ is a closed orientable surface of genus at least 2. Let $\gamma$ be an essential curve on $\Sigma$ such that $(H^-,\gamma)$ and $(H^+,\gamma)$ are pared acylindrical. Then in the hyperbolic metric on $M-\gamma$ the normalized length of the standard meridian $\mu$ determined by $M$ is bounded from above as follows:
\[
L(\mu)^2\le c\frac{S_\gamma^2+d_\gamma^2}{S_\gamma}.
\]
\end{restatable}

\subsubsection*{A closed formula}
So far we have
\[
\frac{S_\gamma}{(S_\gamma+d_\gamma)^2}\lesssim\frac{{\rm Area}(T)}{{\rm Length}(\mu)^2}\lesssim\frac{1}{S_\gamma}.
\]

When $S_\gamma$ is the dominant term, that is $S_\gamma\ge\ep\cdot d_\gamma$ for some (small) fixed $\ep>0$, both sides are coarsely the same and coarsely equal to $S_\gamma/(S_\gamma^2+d_\gamma^2)$.

When $d_\gamma$ is instead the dominant term, that is $S_\gamma\le\ep\cdot d_\gamma$, then we can get an improved upper bound on the length of the canonical meridian via
\[
{\rm Length}(\mu)\gtrsim d_\gamma-S_\gamma\ge(1-\ep)d_\gamma
\]
and, hence, an improved upper bound
\[
\frac{{\rm Area}(T)}{{\rm Length}(\mu)^2}\lesssim\frac{S_\gamma}{d_\gamma^2}.
\]

\begin{restatable}{thm}{annulartwo}
\label{thm:annular2}
There exists $c>0$ depending on $\Sigma$ such that the following holds. Let $M=H^-\cup_\Sigma H^+$ be a Heegaard splitting where $\Sigma$ is a closed orientable surface of genus at least 2. Let $\gamma$ be an essential curve on $\Sigma$ such that $(H^-,\gamma)$ and $(H^+,\gamma)$ are pared acylindrical and $S_\gamma\le d_\gamma/c$. Then in the hyperbolic metric on $M-\gamma$ the normalized length of the standard meridian $\mu$ determined by $M$ is bounded from below as follows:
\[
L(\mu)^2\ge c\frac{d_\gamma^2}{S_\gamma}.
\]
\end{restatable}

Again, in the same scenario, $(S_\gamma^2+d_\gamma^2)/S_\gamma$ is coarsely the same as $d_\gamma^2/S_\gamma$ when $d_\gamma$ is large. This produces the general closed formula
\[
\frac{{\rm Area}(T)}{{\rm Length}(\mu)^2}\simeq\frac{S_\gamma}{S_\gamma^2+d_\gamma^2}
\]
under the assumptions of Theorem \ref{thm:main}. Combining Theorems \ref{thm:non-annular}, \ref{thm:annular}, \ref{thm:annular2} and a simple study of the asymptotic behavior of the function $r(x,y)=x/(x^2+y^2)$ (Lemma \ref{lem:calculus} below) one gets Theorem \ref{thm:main}.

\begin{lem}\label{lem:calculus}
 For every $c, C>0$ there exists $K>0$ such that the following holds. Suppose that the function $r(x,y)$, defined on some subset of $\mathbb (0,\infty)\times(0,\infty)$ satisfies
 \[
 r(x,y)\leq Cx/y^2 \text{if $x\leq y/c$} \text{ and }r(x,y)\leq C/x.
 \]
 
 Then $r(x,y)\leq K \frac{x}{x^2+y^2}.\qedhere$\qed
\end{lem}

The proof of Lemma \ref{lem:calculus} is elementary and we omit it.

All the milestones of the strategy described so far in this section readily combine to a proof of Theorem \ref{thm:main}. 

\begin{proof}[Proof of Theorem \ref{thm:main}]
As both $(H^+,\gamma)$ and $(H^-,\gamma)$ are pared acylindrical, by Theorem \ref{thm:hyperbolicity} $M-\gamma$ has a unique complete finite volume hyperbolic metric. Denote by $T$ the boundary torus of the standard $\ep$-Margulis neighborhood of the cusp and by $\mu\subset T$ the standard meridian determined by $M$.  

Consider $S_\gamma$ and $d_\gamma$. There are four cases depending on whether $S_\gamma\ge R$ or $d_\gamma\ge R$ and $S_\gamma\ge d_\gamma/c$ or $S_\gamma\le d_\gamma/c$ (where $c$ is as in Theorem \ref{thm:annular2}).

{\em Case $S_\gamma\ge R$ and $S_\gamma\ge d_\gamma/c$}. By Theorem \ref{thm:non-annular}, we have $L(\mu)\gtrsim S_\gamma$. Thus the normalized length is as large as we want provided that $S_\gamma$ is large enough. Hence, by Theorem \ref{thm:filling}, $M$ has a hyperbolic metric for which $\gamma$ is a geodesic of length $\ell_M(\gamma)\simeq 1/L(\mu)^2$. By Theorems \ref{thm:non-annular} and \ref{thm:annular} we have $(S_\gamma^2+d_\gamma^2)/S_\gamma\lesssim L(\mu)^2\lesssim S_\gamma$. By Lemma \ref{lem:calculus} we also have $S_\gamma\lesssim (S_\gamma^2+d_\gamma^2)/S_\gamma$. Thus $L(\mu)\simeq (S_\gamma^2+d_\gamma^2)/S_\gamma$.

{\em Case $S_\gamma\ge R$ and $S_\gamma\le d_\gamma/c$}. By Theorem \ref{thm:non-annular}, we have $L(\mu)\gtrsim S_\gamma$. Thus the normalized length is as large as we want provided that $S_\gamma$ is large enough. Hence, by Theorem \ref{thm:filling}, $M$ has a hyperbolic metric for which $\gamma$ is a geodesic of length $\ell_M(\gamma)\simeq 1/L(\mu)^2$. By Theorems \ref{thm:non-annular} and \ref{thm:annular2} we have $(S_\gamma^2+d_\gamma^2)/S_\gamma\lesssim L(\mu)^2\lesssim S_\gamma/d_\gamma^2$. By Lemma \ref{lem:calculus} we also have $S_\gamma/d_\gamma^2\lesssim (S_\gamma^2+d_\gamma^2)/S_\gamma$. Thus $L(\mu)\simeq (S_\gamma^2+d_\gamma^2)/S_\gamma$.

{\em Case $d_\gamma\ge R$ and $S_\gamma\ge d_\gamma/c$}.
Observe that $S_\gamma\ge d_\gamma/c\ge R/c$. By Theorem \ref{thm:non-annular}, we have $L(\mu)\gtrsim S_\gamma$. Thus the normalized length is as large as we want provided that $d_\gamma$ is large enough. Hence, by Theorem \ref{thm:filling}, $M$ has a hyperbolic metric for which $\gamma$ is a geodesic of length $\ell_M(\gamma)\simeq 1/L(\mu)^2$. By Theorems \ref{thm:non-annular} and \ref{thm:annular} we have $(S_\gamma^2+d_\gamma^2)/S_\gamma\lesssim L(\mu)^2\lesssim S_\gamma$. By Lemma \ref{lem:calculus} we also have $S_\gamma\lesssim (S_\gamma^2+d_\gamma^2)/S_\gamma$. Thus $L(\mu)\simeq (S_\gamma^2+d_\gamma^2)/S_\gamma$.

{\em Case $d_\gamma\ge R$ and $S_\gamma\le d_\gamma/c$}.
By Theorem \ref{thm:annular2}, we have $L(\mu)\gtrsim d_\gamma^2/S_\gamma\ge cd_\gamma$. Hence the normalized length is as large as we want provided that $d_\gamma$ is large enough. Hence, by Theorem \ref{thm:filling}, $M$ has a hyperbolic metric for which $\gamma$ is a geodesic of length $\ell_M(\gamma)\simeq 1/L(\mu)^2$. By Theorems \ref{thm:non-annular} and \ref{thm:annular2} we have $(S_\gamma^2+d_\gamma^2)/S_\gamma\lesssim L(\mu)^2\lesssim S_\gamma/d_\gamma^2$. By Lemma \ref{lem:calculus} we also have $S_\gamma/d_\gamma^2\lesssim (S_\gamma^2+d_\gamma^2)/S_\gamma$. Thus $L(\mu)\simeq (S_\gamma^2+d_\gamma^2)/S_\gamma$.
\end{proof}

\subsection*{On the proofs}
We now spend some words on the key ingredients of the proofs. 

In particular, we explain what goes into the description of the torus $T$ and of the meridian $\mu\subset T$ given above.

As is often the case in the study of the relations between curve graphs and hyperbolic 3-manifolds, the problem that we want to solve splits into two parts, namely a non-annular part and an annular one. Both give important contributions to the situation which we briefly describe with the help of some pictures. The annular part is the one that poses the most challenging problems and where many new ideas are needed. For the sake of simplicity, in order to avoid technicalities, what follows is a simplified version of the arguments but it still highlights the main ideas. 

\subsubsection*{Non-annular contribution}
The non-annular subsurfaces of $\Sigma$ having $\gamma$ in their boundary are responsible for the vertical annuli in the above description of the shape of the boundary torus $T$ of the standard Margulis neighborhood of the cusp in $M-\gamma$.

The key input here is to look at the covering $Q\to M-\gamma$ corresponding to $\pi_1(W)$ where $W\subset\Sigma-\gamma$ is the component where most of $S_\gamma$ lies (say at least $1/2$ of $S_\gamma$ comes from subsurfaces contained in this component). It is a geometrically finite hyperbolic 3-manifold diffeomorphic to $W\times\mb{R}$ with rank one cusps corresponding to $\partial W$. The boundaries $T_Q$ of the Margulis neighborhoods of such rank one cusps are annuli of infinite height that cover $T$ under the covering projection $Q\to M-\gamma$ see Figure \ref{fig:cover1}.

\begin{figure}[h]
\begin{overpic}[scale=1.8]{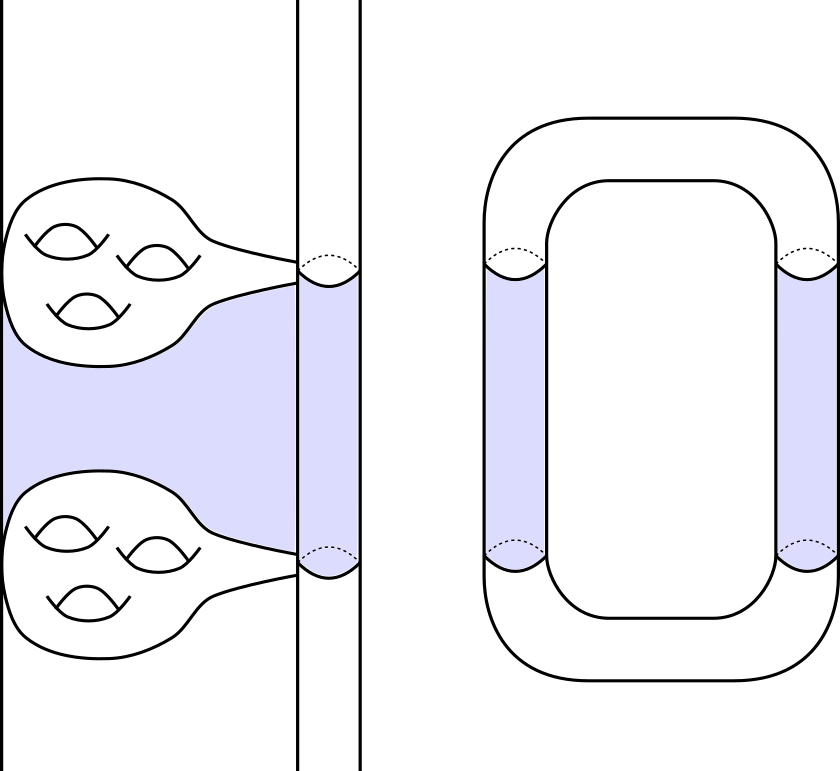}
\end{overpic}
\caption{The covering $Q\to M-\gamma$.}
\label{fig:cover1}
\end{figure}

The convex core $\mc{CC}(Q)$, a canonical convex submanifold of $Q$ isotopic to $W\times[-1,1]\subset Q$ under the diffeomorphism $Q\simeq W\times\mb{R}$, intersects $T_Q$ in a couple of finite annuli $A_Q\subset T_Q$. By the structure theory developed by Minsky \cite{M10} and Brock, Canary, and Minsky \cite{BrockCanaryMinsky:ELC2}, the height of the annuli $A_Q$ is roughly given by the following formula
\[
S_\gamma(\nu^-,\nu^+)=1+\sum_{Y\in\mc{Y}_\gamma\,,\,Y\subset W}{\{\{d_Y(\nu^-,\nu^+)\}\}_K}
\]
where $\nu^-,\nu^+\in\mc{C}(W)$ are the so-called end invariants of $Q$.

We prove two things.
\begin{itemize}
    \item{The restriction of the covering projection $Q\to M-\gamma$ to $T_Q\to T$ is essentially an embedding on $A_Q$, so that $A_Q$ appears isometrically embedded in $T$ as a vertical annulus.}
    \item{For every $Y\in\mc{Y}_\gamma$ and $Y\subset W$ we have $d_Y(\nu^-,\nu^+)\gtrsim d_Y(\mc{D}^-,\mc{D}^+)$.}
\end{itemize}

The second point is one of our new contributions.

In order to prove it we need to look at other coverings of $M-\gamma$, namely $N^-,N^+\to M-\gamma$ corresponding to $\pi_1(H^--\gamma),\pi_1(H^+-\gamma)$. These are geometrically finite handlebodies diffeomorphic to ${\rm int}(H^-),{\rm int}(H^+)$ respectively. Extending ideas from \cite{FSVa} combined with Efficiency of Pleated Surfaces \cite{M00}, we show that for every $Y\in\mc{Y}_\gamma$ with $Y\subset W$ there are disks $\delta^+\in\mc{D}^+,\delta^-\in\mc{D}^-$ whose subsurface projections to $W$ intersect essentially $Y$ and such that the geodesic representatives of such projections in $N^-,N^+$ have uniformly bounded length respectively. The claim will then follow more or less directly from the work \cite{BBCM}. This is discussed in detail in Section \ref{sec:disk_proj}.

The first point and the proof of Theorem \ref{thm:non-annular} is carried out in Section \ref{sec:upper_bound}.

\subsubsection*{Annular contribution}
The annular projections are responsible for the spiraling of the meridian $\mu$ around the waist of the torus $T$. 

The key here is to construct a nice model for the meridian $\mu\subset T$. Schematically, it will consist of four parts, two vertical (blue) and two horizontal (red) as in Figure \ref{fig:sweepout1}. 

\begin{figure}[h]
\begin{overpic}[scale=1.8]{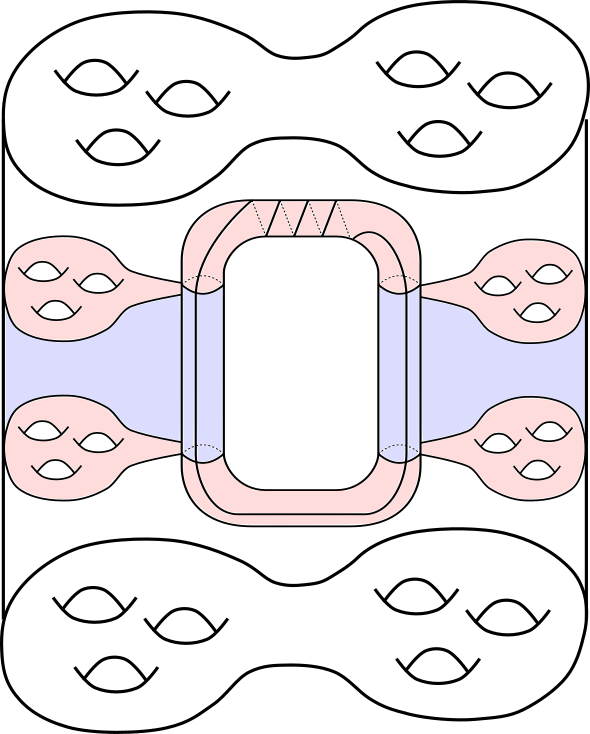}

\end{overpic}
\caption{Construction for the meridian.}
\label{fig:sweepout1}
\end{figure}

In order to construct the model for $\mu$ we start with the red surfaces in the picture. They represent two pleated surfaces properly homotopic to $\Sigma-\gamma\subset M-\gamma$ mapping geodesically the finite laminations of $\Sigma-\gamma$ obtained by intersecting with this surface two disks $\delta^+\in\mc{D}^+$ and $\delta^-\in\mc{D}^-$. 

Using in an essential way the work of Minsky \cite{M01}, but with a key input given by special resolution sequences that we construct, we insert a chain of geometrically controlled homotopies (the blue shaded region) joining the two red surfaces. The tracks of this homotopy on the torus $T$ are the vertical segments of $\mu$. They have length roughly $S_\gamma$.  

The special resolution sequence is constructed in Section \ref{sec:move_seq} and is the crucial combinatorial input to start the good homotopy machine set up in Section \ref{sec:good_homot}.

We now come to the horizontal portions of the meridian.

Recall that the red surfaces are pleated surfaces mapping geodesically the finite leaved laminations coming from the intersections $\delta^-\cap(\Sigma-\gamma)$ and $\delta^+\cap(\Sigma-\gamma)$ for some (arbitrarily) chosen disks $\delta^-\in\mc{D}^-$ and $\delta^+\in\mc{D}^+$. Building on ideas from \cite{FSVDehn}, we use the geometry of the ideal pleated disks filling $\delta^+-\gamma$ and $\delta^--\gamma$ on the two pleated surface to produce two arcs on $T$ (crossing side to side the red horizontal annuli in Figure \ref{fig:sweepout1}) of universally bounded length. Concatenating these arcs with the vertical ones coming from the homotopy described above we obtain a meridian $\mu'$ (different from $\mu$) of length bounded from above roughly by $S_\gamma$ (the length of the vertical pieces). 

We can then complete $\mu'$ to a homology basis for $T$ by adding as a longitude the homotopy class of the intersection of the red pleated surfaces with $T$. It will not be difficult to show that this longitude has always length in some uniform interval $[a,b]$. 

Having such a basis already tames the geometry of $T$ quite a bit. For example, it says that its area is bounded from above by roughly $S_\gamma$. Recall that above we also argued that the presence of the vertical annuli of height $S_\gamma$ and bounded base forces the area of $T$ to be at least $S_\gamma$. So we have that ${\rm Area}(T)\simeq S_\gamma$.

We then consider $\mu$. A model for $\mu$ is constructed by taking $\mu'$ but replacing the top horizontal segment with a parallel copy of the bottom horizontal segment. These two different segments will spiral one with respect to the other as illustrated in Figure \ref{fig:sweepout1}. Morally, as the top horizontal segment of $\mu'$ has universally bounded length, the length of the top horizontal segment of $\mu$ will be computed roughly by how many times the two segments intersect. 

Therefore, the crux of the problem is to show that such intersection is essentially measured by $d_\gamma$. This relies on some 2-dimensional hyperbolic geometry and the fact that the hyperbolic metrics on the red pleated surfaces are related by $k$ consecutive uniformly bilipschitz changes of metrics with $k$ roughly given by $S_\gamma$ (another output of the good homotopy machine proved in Section \ref{sec:good_homot}).

The discussion about the model of the meridian is given in Section \ref{sec:annuli}. There we also prove Theorems \ref{thm:annular} and \ref{thm:annular2}.

\section{Disk projections and moderate length curves}
\label{sec:disk_proj}

The main goal of this section is to prove the following.

\begin{proposition}\label{prop:short_projection}
 Let $H$ be a handlebody of genus at least 2. There exists $L>0$ such that the following holds: Let $\gamma\subset\Sigma=\partial H$ be an essential multicurve and let $Y\subset W:=\Sigma-\gamma$ be a connected component of the complement. Suppose that $(H,\gamma)$ is pared acylindrical. Fix any hyperbolic structure $N$ on $H$ which has rank-1 cusps precisely at $\gamma$. Then there exists $\alpha\in\pi_Y(\calD)$ such that $\ell(\alpha)\leq L$.
\end{proposition}

This result might have independent interest. 

As discussed in Section \ref{sec:strategy}, we will employ it to obtain control on the shape of the flat torus $T$ bounding a standard Margulis neighborhood of the cusp of $M-\gamma$. We discuss all the details in the next section. Roughly speaking, there we first apply Proposition \ref{prop:short_projection} to the handlebody coverings $N^-,N^+$ of $M-\gamma$ corresponding to $\pi_1(H^--\gamma),\pi_1(H^+-\gamma)$ to get some moderate length curves on every component of $\Sigma-\gamma$ in such a way that each of them is a surgery of some disks in $\mc{D}^-,\mc{D}^+$. In turn, such moderate length curves are key to controlling the geometry of the $\pi_1(\Sigma-\gamma)$-cover $Q$ of $M-\gamma$ and, hence, the shape of $T$.

As a preparation for the proof, we briefly recall some facts about pleated surfaces. These are a very useful tool introduced by Thurston \cite{Thu86} to study the geometry of hyperbolic 3-manifolds.

\subsection{Laminations and pleated surfaces}

\begin{dfn}[Lamination]
A {\em lamination} of a hyperbolic surface $(W,\sigma)$ of finite area is a closed subset $\lambda\subset W$ which can be written as a union of pairwise disjoint complete simple geodesics, the {\em leaves} of the lamination. A lamination is called {\em maximal} if each component of the complement $W-\lambda$ is an ideal hyperbolic triangle.
\end{dfn}

As the only laminations that we will ever use are finite unions of proper essential complete geodesics we don't need to discuss much of the structural results for laminations. We address the interested reader to \cite[Chaper I.4]{CEG:notes_on_notes} for a comprehensive treatment. We will use the following:

\begin{lem}[{see \cite[Proposition 8.8.4]{ThuNotes}}]
\label{lem:straight in cusp}
    For all sufficiently small $\epsilon_0>0$ the following holds. Let $(W,\sigma)$ be a finite-volume hyperbolic surface and $\lambda$ a lamination on $W$. Then any component of the intersection of $\lambda$ with the $\epsilon_0$-cuspidal part of $W$ is an infinite ray orthogonal to the boundary of the $\epsilon_0$-cuspidal part.
\end{lem}

\begin{dfn}[Pleated Surface]
A {\em pleated surface} consists of the following data: A (possibly disconnected) surface of finite type $W$, containing a lamination $\lambda$, and a proper map $f:W\to N$ in a hyperbolic 3-manifold $N$ such that:
\begin{itemize}
  \item{The pull-back metric $(W,\sigma)$ is a complete finite area hyperbolic metric.}
  \item{$f$ maps each leaf $\ell$ of $\lambda$ to a complete geodesic in $N$ (possibly closed if the leaf $\ell$ is closed).}
\end{itemize}  

We will also say that $f$ {\rm maps geodesically the lamination $\lambda$}.

The restriction of a pleated surface $g:W\to N$ to the geodesic lamination $\lambda\subset W$ that it maps geodesically in $N$ has a {\rm canonical lift} ${\bf p}_g:\lambda\to{\bf P}(N)$ to the projectivization of the unit tangent bundle of $N$. The image ${\bf p}_g(x)$ of a point $x\in\lambda$ lying on a leaf $\ell$ is the tangent line tangent to the geodesic $g(\ell)$ at $g(x)$.
\end{dfn}

Again, we refer to \cite[Chapter I.5]{CEG:notes_on_notes} for an in-depth discussion. 

In this paper we will need a few fundamental structural properties of pleated surface. Discussing them is our next step. Roughly speaking, it is possible to turn topological properties of the maps $f:W\to N$ into quantitative geometric control.

\subsection{Geometry of pleated surfaces}

\begin{dfn}[Type Preserving]
A pleated surface $f:(W,\sigma)\to N$ is {\rm type preserving} if it is $\pi_1$-injective and $f_*:\pi_1(W)\to\pi_1(N)$ takes hyperbolic (resp. parabolic) elements to hyperbolic (resp. parabolic) elements. 
\end{dfn}

Type preserving pleated surfaces relate nicely with the thick-thin decomposition of the hyperbolic surface $(W,\sigma)$ and of the hyperbolic 3-manifold $N$. In fact, we have the following.

\begin{lem}[{see \cite[Lemma 3.1]{M00}}]
\label{lem:thick to thick}
For every $\ep_0>0$ there exists $\ep_1>0$ (only depending on $\ep_0$ and $W$) such that for every type preserving $\pi_1$-injective pleated surface $g:(W,\sigma)\to N$ we have that 
\[g^{-1}(N^{thin}_{\leq \epsilon_1})\subseteq W^{thin}_{\leq \epsilon_0},\]
where $M^{thin}_{\leq \epsilon}$ denotes the $\epsilon$-thin part of the hyperbolic manifold $M$.
\end{lem}

A rephrasing of the conclusion of the lemma is that only the $\ep_0$-thin part of $(W,\sigma)$ can be mapped to the $\ep_1$-thin part of $N$.

The next topological condition introduced by Thurston \cite{Thu86} will be converted into geometric control on the distortion of the image of a lamination mapped geodesically by a pleated surface in a hyperbolic 3-manifold.

\begin{dfn}[(Weakly) Doubly Incompressible {\cite{Thu86}}]
A map $f:W\to N$ of a connected orientable finite area hyperbolic surface $W$ in a hyperbolic 3-manifold $N$ that takes cusps to cusps is {\em doubly incompressible} if    
\begin{enumerate}[(a)]
    \item{$f$ is $\pi_1$-injective.}
    \item{Non-trivial homotopy classes of arcs $(I,\partial I)\to(W,{\rm cusps}(W))$ relative to cusps map injectively to homotopy classes of arcs $(I,\partial I)\to(N,{\rm cusps}(N))$.}
    \item{Every $\pi_1$-injective cylinder $c:S^1\times I\to N$ such that $c|_{\partial S^1\times I}=f\circ c_0$ with $c_0:\partial S^1\times I\to W$ satisfies either $c_*\pi_1(S^1\times I)<\pi_1({\rm cusps}(N))$ or $c_0$ extends to a map $c_0:S^1\times I\to W$.}
    \item{$f_*$ maps every maximal abelian subgroup of $\pi_1(W)$ to a maximal abelian subgroup of $\pi_1(N)$.}
\end{enumerate}
The map $f$ is {\em weakly doubly incompressible} if instead of (d) it satisfies the weaker property
\begin{enumerate}[(d')]
    \item{$f_*$ maps every maximal cyclic subgroup of $\pi_1(W)$ to a maximal cyclic subgroup of $\pi_1(N)$.} 
\end{enumerate}
\end{dfn}

Thurston proves the following two structural properties:

\begin{thm}[Lamination Injects, {\cite[Theorem 5.6]{Thu86}}]
\label{thm:weakly embeds}
If $g:(W,\sigma)\to N$ is a weakly doubly incompressible type preserving pleated surface in a hyperbolic 3-manifold $N$ mapping geodesically a lamination $\lambda$ then the canonical lift ${\bf p}_g:\lambda\to{\bf P}(N)$ is an embedding.
\end{thm}

\begin{thm}[Uniform Injectivity, {\cite[Theorem 5.7]{Thu86}}]
\label{uniforminj}
Fix $\ep_0>0$, a Margulis constant and $W$ a finite type connected surface. For every $\ep>0$ there exists $\delta>0$ such that for any type preserving doubly incompressible pleated surface $g:(W,\sigma)\to N$ mapping geodesically a lamination $\lambda$, if $x,y\in\lambda$ lie in the $\ep_0$-thick part of $(W,\sigma)$, then
\[
d_{{\bf P}(N)}({\bf p}_g(x),{\bf p}_g(y))\le\delta\Longrightarrow d_\sigma(x,y)\le\ep.
\]
Here ${\bf p}_g:\lambda\to {\bf P}(N)$ denotes the map induced by $g$ from the lamination $\lambda$ to the projective unit tangent bundle of $N$.
\end{thm} 

As we want to deal with disconnected surfaces as well, we also need the following mild enhancement.

\begin{restatable}{proposition}{connectedcomp}
\label{prop:connected component}
Fix $\ep_0>0$, a Margulis constant. There exists $\ep_2>0$ such that the following holds: Let $H$ be a handlebody and $\gamma\subset\Sigma=\partial H$ an essential multicurve such that $(H,\gamma)$ is pared acylindrical. We denote by $W:=\Sigma-\gamma$ the complement of $\gamma$. Let $N$ be a hyperbolic structure on $H$ with rank 1 cusps precisely at $\gamma$. For any pleated surface $g:(W,\sigma)\to N$ in the homotopy class of the inclusion $W\subset H$ mapping geodesically a lamination $\lambda$ of $W$, if $x,y\in\lambda$ lie in the $\ep_0$-thick part of $(W,\sigma)$ and $d_{{\bf P}(N)}({\bf p}_g(x),{\bf p}_g(y))\le\ep_2$, then $x,y$ lie in the same connected component of $W$.
\end{restatable}

We remark that Proposition \ref{prop:connected component} should be already covered by Thurston's original works \cite[Theorem 5.7]{Thu86} but we include a proof in Appendix \ref{appendixA}. The reason why we include a discussion and a proof is because we wanted to clarify the assumptions on the connectedness of $W$ which we could not easily extract from the statements in the original sources.

\subsection{Short surgeries of disks}

Here we prove Proposition \ref{prop:short_projection} which is the main result of this section.

\subsubsection{Choices of constants} 
As the arguments in the proof will involve several constants, we briefly describe our choices: 
\begin{enumerate}
    \item{We start by fixing $\ep_0>0$, a (small) Margulis constant chosen so that every lamination of a finite area hyperbolic surface can intersect the $\ep_0'$-cuspidal part in a collection of infinite rays orthogonal to the boundary of the cuspidal part (see Lemma \ref{lem:straight in cusp}).}
    \item{By Lemma \ref{lem:thick to thick}, there is a Margulis constant $\ep_1\in(0,\ep_0)$ (only depending on $\ep_0$ and $\Sigma$) such that for every $\pi_1$-injective, type preserving pleated surface $g:(W,\sigma)\to N$ in a hyperbolic 3-manifold $N$ only the $\ep_0$-thin part of $(W,\sigma)$ can be mapped to the $\ep_1$-thin part of $N$. }
    \item{By Theorem \ref{uniforminj} and Proposition \ref{prop:connected component}, there exists $\ep_3>0$ (only depending on $\ep_0$ and $\Sigma$) such that the following holds: Let $(H,\gamma)$ be a pared acylindrical handlebody with boundary $\Sigma=\partial H$. Denote by $W:=\Sigma-\gamma$ the complement of $\gamma$. Let $N$ be a hyperbolic structure on $H$ with rank 1 cusps precisely at $\gamma$. If $g:(W,\sigma)\to N$ is a pleated surface properly homotopic to the inclusion $W\subset H-\gamma$ and mapping geodesically a lamination $\lambda\subset W$ then if the directions ${\bf p}_g(x),{\bf p}_g(x')\in{\bf P}(N)$ of two leaves $\ell,\ell'$ of $\lambda$ at points $x\in\ell,x'\in\ell'$ contained in the $\ep_0$-thick part of $(W,\sigma)$ are $\ep_3$-close in ${\bf P}(N)$, then $x,x'$ lie in the same connected component of $\Sigma-\gamma$ and $d_\sigma(x,x')\le\ep_1/100$.}
    \item{The last constants that we use come from the following observations. 
    
    {\bf Remark}. There exists $\eta_1>0$ (only depending on $\ep_3$ and $\ep_1$) such that if two geodesic lines $\ell,\ell'$ in a hyperbolic 3-manifold $N$ stay $\eta_1$-close on subsegments $s,s'$ of length at least $1$ passing through the $\ep_1$-thick part, then their directions have distance at most $\ep_3$ in ${\bf P}(N)$ along such subsegments $s,s'$.\qed

    {\bf Remark}. There exists $\eta_2>0$ (only depending on $\eta_1$) such that, if two disjoint bi-infinite geodesics $\ell,\ell'$ of $\mb{H}^2$ come closer than $\eta$ at points $x,x'$ then they stay $\eta_1$ close along segments $s,s'$ of length at least 2 centered around $x,x'$.\qed
    }
\end{enumerate}

\begin{proof}[Proof of Proposition \ref{prop:short_projection}]
We follow the strategy of the proof of the Efficiency of Pleated Surfaces (as can be found in \cite{M00} or \cite{ThurstonII}) with some modifications.

\subsubsection{The geometric step}
\label{sec:geometric step}
Let $\delta\in\mc{D}$ be a disk with minimal intersection number with $\gamma$. After putting $\delta$ and $\gamma$ in minimal position, consider the arcs of $\delta\cap W$ up to isotopy (recall that $W=\Sigma-\gamma$), and let $\lambda\subset W$ consist of one representative for each isotopy class of those arcs. Note that if $Y\subset W$ is a connected component, then each arc in $\delta\cap Y$ represents the subsurface projection $\pi_Y(\delta)$.

As in \cite[Proposition 4.2]{FSVa}, there exists some pleated surface $f:(W,\sigma)\to N$ properly homotopic to the inclusion $W:=\Sigma-\gamma\subset H-\gamma$ that maps $\lambda$ geodesically in $N$ (see for example Theorems I.5.3.6 and I.5.3.9 in~\cite{CEG:notes_on_notes}).

Observe that we can represent $\delta-\gamma$ as a (ideal) concatenation
\[
\delta-\gamma\simeq\ell_1\star\dots\star\ell_n
\]
where $n=i(\delta,\gamma)$ and each $\ell_j$ is a leaf of $\lambda$ (possibly with repetitions, but two consecutive $\ell_j,\ell_{j+1}$, indices modulo $n$, are distinct and on opposite sides of $\gamma$). The geodesic lines $f(\ell_1),\cdots,f(\ell_n)$ form the boundary of an ideal polygon in $N$ that can be filled with an ideal pleated disk $\Delta$. Note that the area of $\Delta$ is $\pi (i(\delta,\gamma)-2)\le \pi i(\delta,\gamma)$. 

Consider the portions $\beta_j^0$ of the edges $\ell_j$ that are contained in the $\ep_0$-non-cuspidal parts $W_0$ of $(W,\sigma)$, that is $\beta_j^0:=\ell_j\cap W_0$. By our choice of constants (1) every $\beta_j^0$ is a non-empty compact connected subsegment of $\ell_j$ as $\ell_j$ intersects the cuspidal parts in infinite rays going straight into the cusps.

We now fix an auxiliary (small) $\eta>0$. We consider the family $\{\beta_j^0\}_{j\in J}$ of those segments $\beta_j^0$ that have distance at least $\eta$ from all other $\beta_i^0$ in the metric of the pleated disk $\Delta$. In particular, the $\eta/2$-normal neighborhoods in $\Delta$ of such segments are embedded and pairwise disjoint. 

Recall that the area of the $\eta$-normal (half-)neighborhood of a segment of length $s$ in $\mb{H}^2$ is $s\cdot\sinh(\eta)$. Therefore, we have
\[
\text{\rm Area}(\Delta)\ge\sum_{j\in J}{\ell(\beta_j^0)\sinh(\eta/2)}.
\]

We show the following. 

\begin{claim}
\label{claim:length2}
    If $\eta<\eta_2$ then $J=\{1,\cdots,n\}$.
\end{claim} 

Here $\eta_2$ is as in our choice of constants (4).

\begin{proof}[Proof of the claim]    
Suppose that there are segments $\beta_j^0,\beta_i^0$, which are $\eta$-close in $\Delta$. By definition of the set $\{\beta_j^0\}_{j\in J}$ and the fact that $j\not\in J$, there are two cases: 
\begin{enumerate}
\item{The segments $\beta_j^0,\beta_i^0$ have points $x_j\in\beta_j^0,x_i\in\beta_i^0$ with $d_\Delta(x_j,x_i)<\eta$ and $x_j$ is in the $\ep_0$-thick part of $(W,\sigma)$.}
\item{The segments $\beta_j^0,\beta_i^0$ have points $x_j\in\beta_j^0,x_i\in\beta_i^0$ with $d_\Delta(x_j,x_i)<\eta$ and $x_j,x_i$ are in the $\ep_0$-Margulis neighborhoods of closed geodesics $\nu_j,\nu_i$ of $(W,\sigma)$.}
\end{enumerate} 

{\bf Case (1)}. Notice that, by our choice of constants (2), only the $\ep_0$-thin part of $(W,\sigma)$ can be mapped to the $\ep_1$-thin part of $N$. Hence, the image of $x_j$ under $f$ lies in the $\ep_1$-thick part of $N$. Let $\xi$ be an arc on $\Delta$ of length at most $\eta$ joining $x_j$ to $x_i$. By our choice of constants (4), if the two geodesics $\ell_j$ and $\ell_i$ stay $\eta_2$-close in $\Delta$ at points $x_j,x_i$ then they stay $\eta_1$ close along subsegments of length at least 2 centered around the points $x_j,x_i$ and the same is true for the lines $f(\ell_j),f(\ell_i)$. By the same choice of constants (4), as $f(\ell_j),f(\ell_i)$ stay $\eta_1$-close along segments of length at least 1 passing through the $\ep_1$-thick part of $N$, their directions in ${\bf P}(N)$ are $\ep_3$-close to each other along the same segments. 

By Proposition \ref{prop:connected component} and the choice of constants (3), this implies that $\ell_j,\ell_i$ lie in the same connected component of $W$.

By Theorem \ref{uniforminj} and our choice of constants (3) applied to the restriction of the pleated surface $f:(W,\sigma)\rightarrow N$ to the common component of $\beta_j^0,\beta_i^0$, the two points $x_j\in\beta_j^0,x_i\in\beta_i^0$ are $\ep_1/100$-close on the surface $(W,\sigma)$. As $\xi$ and the image $f(\zeta)$ of the shortest segment connecting $x_i$ to $x_j$ have both length smaller than $\ep_1$, they are homotopic relative to the endpoints via a homotopy supported in an embedded metric ball of radius $\ep_1$ (recall that the injectivity radius at $f(x_i),f(x_j)$ is at least $\ep_1$).

We use $\zeta$ to surger $\delta$ and obtain two (not necessarily simple) closed curves $\kappa_1=\delta'\zeta$ and $\kappa_2=\zeta^{-1}\delta''$ with $\delta=\delta'\delta''$. Both curves are homotopically trivial in $N$ (they bound disks made of a sub-disk of the pleated disk and a disk bounded by $\xi$ and $\zeta$). Notice that, as $\xi$ separates the vertices of the polygon $\delta$ in two non-empty sets and $\zeta$ is disjoint from $\gamma$, we have that the geometric intersections $i(\kappa_1,\gamma)$ and $i(\kappa_2,\gamma)$ are both strictly smaller than $n=i(\delta,\gamma)$. 

We now show that this leads to a contradiction. Let us consider $\kappa_1$. If it is null-homotopic in $\Sigma$, then we can use it to homotope $\delta$ to $\kappa_2$, a curve that intersects $\gamma$ less than $n=i(\delta,\gamma)$ times. Since this is not possible, $\kappa_1$ is essential in $\Sigma$. By the Loop Theorem, there exists a cycle in $\kappa_1$ (thought as a 4-valent graph in $\Sigma$ whose vertices are the possible transverse self intersections along $\zeta$) that represents an essential disk bounding curve $\delta_1\in\mc{D}$. Such a disk $\delta_1$, however, would have $i(\delta_1,\gamma)\le i(\kappa_1,\gamma)<i(\delta,\gamma)$ which contradicts the minimality of $\delta$. 

This finishes the discussion of Case (1).

{\bf Case (2)}. The strategy is the same as in the previous case. We find an arc $\zeta$ on $W=\Sigma-\gamma$ joining the two segments $\beta_j^0$ and $\beta_i^0$ and use it to surger $\delta$ to produce curves $\kappa_1$ and $\kappa_2$ that are essential in $\Sigma$, null-homotopic in $N$, and intersect $\gamma$ in less than $i(\delta,\gamma)$ points.

As a first step, observe that since $x_j,x_i$ are contained in the $\ep_0$-Margulis neighborhoods of $\nu_j,\nu_i$ in $(W,\sigma)$, we have ${\rm inj}_{x_j}W,{\rm inj}_{x_i}W<\ep_0$. As $f$ is 1-Lipschitz, the same is true in $N$, that is ${\rm inj}_{f(x_j)}N,{\rm inj}_{f(x_i)}N<\ep_0$. As $d_\Delta(x_j,x_i)<\eta$ and $\Delta$ maps 1-Lipschitz into $N$, we have $d_N(f(x_j),f(x_i))<\eta$. Thus, $f(x_j),f(x_i)$ are both contained in the same $\ep_0$-Margulis region $\mb{T}$ of $N$ (recall that distinct Margulis regions have distance at least $1/\ep_0$ and we can assume that $\eta<1/\ep_0$).

Denote by $W_j,W_i$ the connected components of $W$ containing $\beta_j^0,\beta_i^0$. We now exploit the fact that $(H,\gamma)$ is pared acylindrical to show that: 

{\bf Subclaim}. We have $W_j=W_i$ and $\nu_j=\nu_i$.

\begin{proof}[Proof of the subclaim]
By $\pi_1$-injectivity of $W_j,W_i$, the curves $f(\nu_j),f(\nu_i)$ are homotopic to powers of the core curve of the Margulis tube $\mb{T}$. Such power must be $\pm1$, in fact, since $(H,\gamma)$ is pared acylindrical, maximal cyclic subgroups of $\pi_1(W)$ are mapped to maximal cyclic subgroups of $\pi_1(H)$. Up to perhaps changing $\nu_j$ with $\nu_j^{-1}$ we can assume that $\nu_j$ and $\nu_i$ are homotopic. Such homotopy provides us a proper cylinder $(A,\partial A)\to (H,W_j\cup W_i)$. Since, again, $(H,\gamma)$ is pared acylindrical, this cylinder is properly homotopic into the boundary (relative to $\partial A$). Hence, we must have $\nu_j=\nu_i$ and $W_j=W_i$.    
\end{proof}   

Let $\xi$ be an arc, provided by the assumption of Case (2), of length $\le\eta$ on $\Delta$ that joins the point $x_j\in\beta_j^0$ to $x_i\in\beta_i^0$. By the above discussion, both $x_j,x_i$ are contained in the same Margulis neighborhood $A_j$ of the geodesic $\nu_j=\nu_i$. Since $\pi_1(A_j)\rightarrow\pi_1(\mb{T})$ is surjective, we can find an arc $\zeta$ in $A_j$, joining $x_j$ to $x_i$, which is homotopic relative to the endpoints to $\xi$. 

The proof can now proceed as in the previous case: We use $\zeta$ to surger $\delta$ and produce two closed curves $\kappa_1$ and $\kappa_2$ with $\delta=\kappa_1\kappa_2$ and both intersecting $\gamma$ in less than $i(\delta,\gamma)$ points. The curve $\kappa_1$ is essential in $\Sigma$, otherwise we could homotope $\delta$ to $\kappa_2$ and reduce its intersections with $\gamma$. The curve $\kappa_1$ is, instead, null-homotopic in $N$ because $\zeta$ is homotopic to $\xi$. Hence, thinking of $\kappa_1$ as a 4-valent graph in $\Sigma$, by the Loop Theorem we can promote it to a disk bounding curve $\delta_1\subset\Sigma$ represented by a simple cycle on the graph. Such simple cycle has $i(\delta_1,\gamma)\le i(\kappa_1,\gamma)$.   

This finishes the proof of Case (2).
\end{proof}

\subsubsection{The algebraic step}
In this step we choose $\eta=\eta_2/2$.

Using Claim \ref{claim:length2}, we get 

\begin{align*}
\sum_{j\le i(\delta,\gamma)}{\ell(\beta_j^0)} &=\sum_{j\in J}{\ell(\beta_j^0)}\\
&\le\frac{\text{\rm Area}(\Delta)}{\sinh(\eta_2/4)}\\
&=\frac{\pi(i(\delta,\gamma)-2)}{\sinh(\eta_2/4)}.
\end{align*}
Notice that, by construction, no two consecutive $\beta_j^0$ can lie outside the subsurface $Y$, so at least half of them are contained in $Y$. Thus, at least one of them has length at most $4+2\pi/\eta$, otherwise the total length would be larger than $i(\delta,\gamma)(2+\pi/\sinh(\eta_2/2))$. 

Let $\beta:=\beta_j^0$ be one of the segments, properly embedded in $Y_0$, of uniformly bounded length. By classical hyperbolic geometry, the length of a component of $\partial Y_0$ is equal to the area of the cusp it bounds. In particular, the length of each component of $\partial Y_0$ is bounded by the area of the surface $Y$ which is $-2\pi\chi(Y)$. 

{\bf Case (1)}. If $\beta$ has both endpoints on the same connected component $\xi$ of $\partial Y_0$, then consider a very small normal neighborhood $U$ of $\beta\cup\xi$. If the neighborhood is small enough, then the length of the boundary $\partial U$ is bounded by $2\ell(\beta)+\ell(\xi)+1$. Each component of $\partial U$ is an essential simple closed curve representing $\pi_Y(\delta)$.   

{\bf Case (2)}. If $\beta$ has its endpoints on different connected components $\xi_1,\xi_2$ of $\partial Y_0$, then consider a very small normal neighborhood $U$ of $\beta\cup\xi_1\cup\xi_2$. If the neighborhood is small enough, then the length of the boundary $\partial U$ is bounded by $2\ell(\beta)+\ell(\xi)+\ell(\xi_2)+1$. The boundary $\partial U$ is an essential simple closed curve representing $\pi_Y(\delta)$.

This concludes the proof of Proposition \ref{prop:short_projection}.
\end{proof}

{\bf Remark}. As it will play a role in the next section we make the following observation: Notations as in the proof of the proposition. Let us look at the two cases of the algebraic step more closely. 

{\em Case (1)}. Let $W\subset Y$ be an essential non-annular subsurface with $\xi\subset\partial W$. Then we can choose $\alpha$, the surgery of the disk $\delta$, so that it intersects essentially $W$. 

{\em Case (2)}. Let $W\subset Y$ be an essential non-annular subsurface with $\xi_1\subset\partial W$ or $\xi_2\subset\partial W$. Then we can choose $\alpha$, the surgery of the disk $\delta$, so that it intersects essentially $W$.

In fact, in both cases, $\alpha$ can be chosen to be a component of $\partial U$ where $U$ is either a small normal neighborhood of $\beta\cup\xi$ or $\beta\cup\xi_1\cup\xi_2$. Note that $U$ is always a pair of pants which either is contained in $W$ or its boundary intersects essentially the boundary of $W$.

\section{The contribution of non-annular subsurfaces}
\label{sec:upper_bound}

In this section we prove Theorem \ref{thm:non-annular} which we restate for convenience.

\nonannular*

The idea of the proof is that the boundary $T$ of a standard neighborhood of the cusp of $M-\gamma$ contains an embedded vertical annulus with base parallel to the components of $(\Sigma-\gamma)\cap T$ and height at least $S_\gamma$ (see Figure \ref{fig:torus1}). Every simple closed torus1c)urve $\mu\subset T$ crossing essentially $(\Sigma-\gamma)\cap T$, having to cross this annulus from side to side, has length at least $S_\gamma$.

In order to prove this we first study the coverings $N^-,N^+\to M-\gamma$ corresponding to $\pi_1(H^--\gamma),\pi_1(H^+-\gamma)$. There, using Proposition \ref{prop:short_projection}, for every component $Y\subset\Sigma-\gamma$ we find two subsurface projections $\pi_Y(\delta^-),\pi_Y(\delta^+)$ of disks $\delta^-\in\mc{D}^-,\delta^+\in\mc{D}^+$ whose geodesic representatives in $N^-,N^+$ have uniformly bounded length. Then, we use this information to control the geometry of the coverings $Q\to M-\gamma$ corresponding to $\pi_1(Y)$ and deduce from it the desired claim on $T$. 

\subsection{The proof of Theorem \ref{thm:non-annular}}
\label{subsec:proof non annular}
The proof of Theorem \ref{thm:non-annular} occupies the whole Section \ref{subsec:proof non annular}.

\begin{proof}[Proof of Theorem \ref{thm:non-annular}]
We proceed by small steps addressing a problem at time.

\subsubsection{The choice of the component $Y\subset\Sigma-\gamma$}
Consider $A\subset\Sigma$ a tubular neighborhood of $\gamma$ and denote by $\gamma_1\sqcup\gamma_2=\partial A$ its boundary. Observe that 
\[
S_{\gamma_1}+S_{\gamma_2}\ge S_\gamma
\]
where
\[
S_{\gamma_j}:=\sum_{W\subset\Sigma-A\;,\;\gamma_j\subset\partial W}{\{\{d_W(\mc{D}^-,\mc{D}^+)\}\}_K}.
\]
We can assume that $S_{\gamma_1}\ge S_\gamma(\mc{D}^-,\mc{D}^+)/2$. We denote by $Y\subset\Sigma-\gamma$ be the connected component containing the curve $\gamma_1$ and choose on it the end corresponding to $\gamma_1$ (note that we have a choice only when $\gamma$ is non-separating).

As stated above the proof of the theorem relies on the analysis of several coverings of $M-\gamma$. For convenience of the reader, we draw the tower of coverings involved in the argument
\[
\xymatrix{
 &Q\ar[dr]^{\pi_1(Y)}\ar[dl]_{\pi_1(Y)} &\\
N^-\ar[dr]_{\pi_1(H^--\gamma)} & &N^+\ar[dl]^{\pi_1(H^+-\gamma)}\\
 &M-\gamma. &
}
\]

We start by analyzing the handlebody coverings $N^-,N^+\to M-\gamma$ corresponding to the subgroups $\pi_1(H^--\gamma),\pi_1(H^+-\gamma)$.

\subsubsection{The handlebody coverings}
By Lemma \ref{lem:graph_of_groups}, the fundamental group of $H^+-\gamma$ inject in $\pi_1(M-\gamma)$. We first work in the covering $N^+\to M-\gamma$ corresponding to $\pi_1(H^+-\gamma)$ to find in $Y$ a short surgery of a disk in $\mc{D}^+$. The same arguments will apply to the covering $N^-\to M-\gamma$ corresponding to $\pi_1(H^--\gamma)$ and will produce in $Y$ a short surgery of a disk in $\mc{D}^-$.

We first describe the topology and geometry of $N^+$ and check that it satisfies the assumptions of Proposition \ref{prop:short_projection}. In order to do so, we need a couple of observations, namely Claim \ref{claim:parabolics} (to determine the cusps) and Claim \ref{claim:bonahon condition} (to determine the diffeomorphism type). Then, using Proposition \ref{prop:short_projection}, we construct on $Y$ a curve of uniformly bounded length representing $\pi_Y(\mc{D}^+)$ (see Claim \ref{claim:short surgeries}).

First, we identify the cusps of $N^+$.

With an abuse of notation, we will often conflate a curve with the conjugacy class that it represents in the fundamental group of the ambient space, or also a representative within the conjugacy class.

\begin{claim}
\label{claim:parabolics}
All the parabolic elements of $\pi_1(H^+-\gamma)$ are conjugate into $\langle\gamma\rangle$. 
\end{claim}

\begin{proof}[Proof of the claim]
Consider the action of $\pi_1(M-\gamma)$ on the Bass-Serre tree $\mathcal T$ of the splitting corresponding to the decomposition of $M-\gamma$ given by the two handlebodies. We have that the fundamental group $\pi_1(T)$ stabilizes a line on $\mathcal T$, acting without inversions on it. As a consequence, given an element $g$ of $\pi_1(H^+-\gamma)$, if $g$ is parabolic then it fixes a line point-wise. Therefore it suffices to show that any element $g$ of $\pi_1(H^+-\gamma)$ not conjugate into $\gamma$ cannot fix two distinct edges of $\mathcal T$ emanating from the vertex of $\mathcal T$ corresponding to $\pi_1(H^+-\gamma)$. To do so, it suffices to show that the intersection of the stabilizers of the two edges intersects along a conjugate of $\langle\gamma\rangle$. These stabilizers are distinct conjugates in $\pi_1(H^+-\gamma)$ of the fundamental group of a connected component of $\Sigma-\gamma$. Translating to topology, any element in the intersection of these conjugates yields a homotopy in $H^+-\gamma$ between curves on $\Sigma-\gamma$. If the homotopy is deformable into the cusp, then the element we started with is conjugate into $\langle \gamma\rangle$, so this case is fine. But on the other hand if the homotopy is not deformable into the cusp then by pared acylindricity it is deformable into $\Sigma-\gamma$, and this contradicts that we started with distinct conjugates.
\end{proof}

Next, we determine the topology of $N^+$. To this purpose, we need an additional piece of information about the algebraic properties of $\pi_1(H^+-\gamma)$. The following is a direct consequence of ``algebraic separability'' implying ``geometric separability'' in \cite{stallings:handlebody}, see \cite[Theorem 3.2]{stallings:handlebody}.

\begin{claim}
\label{claim:bonahon condition}
Let $\gamma\subset\partial H^+$ be a simple closed curve such that $\partial H-\gamma$ is $\pi_1$-injective. Consider a free product decomposition $\pi_1(H^+-\gamma)=A*B$. Then (any representative in $\pi_1(H^+-\gamma)$ of) $\gamma$ is not conjugate into $A$ or $B$.  \end{claim}

Note that $H^+-\gamma$ lifts homeomorphically into $N^+$ and its inclusion in $N^+$ is a homotopy equivalence. As a consequence of Claim \ref{claim:bonahon condition}, $\pi_1(N^+)$ satisfies Bonahon's Condition (*) as introduced in \cite{Bo86}, in particular, we can apply Bonahon's Tameness Theorem \cite{Bo86} which says that the complement of $H^+-\gamma$ in $N^+$ is a product $(\Sigma-\gamma)\times(0,\infty)$ so that $N^+$ is diffeomorphic to ${\rm int}(H^+-\gamma)$. By Claims \ref{claim:bonahon condition}, \ref{claim:parabolics}, and their consequences, $N^+$ is a hyperbolic structure on $(H,\gamma)$ that satisfies the assumptions of Proposition \ref{prop:short_projection}. Applying it we obtain the following:

\begin{claim}
\label{claim:short surgeries}
There is an essential simple closed curve $\alpha^+\subset Y$ (resp. $\alpha^-\subset Y$) such that $\alpha^+\in\pi_Y(\mc{D}^+)$ (resp. $\alpha^-\in\pi_Y(\mc{D}^-)$) and $\ell_{N^+}(\alpha^+)\le L$ (resp. $\ell_{N^-}(\alpha^-)\le L$) for some uniform constant $L$ only depending on $\Sigma$. In particular $\ell_{M-\gamma}(\alpha^+)\le L$ (resp. $\ell_{M-\gamma}(\alpha^-)\le L$).
\end{claim}

\subsubsection{The surface covering}
We now move to the surface covering $Q\to M-\gamma$ corresponding to $\pi_1(Y)$ where we combine the information we obtained in the handlebody coverings and the model manifold technology by Minsky \cite{M10} and Brock, Canary, and Minsky \cite{BrockCanaryMinsky:ELC2} to understand the shape of the intersection of the cusps of $Q$ corresponding to $\partial Y$ with the convex core of $Q$ (see the discussion below). This intersection is a finite annulus mapping to $T$ under the covering projection. The goal is to relate the size of the annulus to the quantity $S_\gamma(\mc{D}^-,\mc{D}^+)$. In the next step we will push this estimate from $Q$ to $M-\gamma$ using covering theory.

Let $p:Q\to M-\gamma$ be the $\pi_1(Y)$-covering of $M-\gamma$. Recall that $\pi_1(Y)$ injects in $\pi_1(M-\gamma)$ and that peripheral elements of $\pi_1(Y)$ represent parabolic element of $\pi_1(M-\gamma)$ and every parabolic element of $\pi_1(Y)$. Again, by Bonahon's Tameness \cite{Bo86}, this implies that $Q$ is diffeomorphic to $Y\times\mb{R}$.

By Canary-Thurston Covering Theorem \cite{C96}, we have two possibilities: 
\begin{itemize}
    \item{Either $Q$ is a {\em geometrically finite} hyperbolic 3-manifold (see below for an explanation of the terminology) or}
    \item{There exists a finite covering of $M-\gamma$ that fibers over the circle and $Y$ lifts to a fiber of the fibration.}
\end{itemize}  

The second case cannot occur as the fundamental group of $Y$ would be normal in $\pi_1(M-\gamma)$ and we can rule this out by looking at the structure of the fundamental group $\pi_1(M-\gamma)$ as described in Lemma \ref{lem:graph_of_groups}.

Thus, we must be in the first case, that is, $Q$ is geometrically finite. This means that there exists a closed convex subset $C\subset Q$ (containing all the geodesics that connect two points in it) whose inclusion $C\subset Q$ is a homotopy equivalence and such that the 1-neighborhood of $C$ in $Q$ has finite volume. The smallest such subset is the {\em convex core} of $Q$ and we denote it by $\mc{CC}(Q)$. Topologically, there exists a homeomorphsim from $Q$ to $Y\times\mb{R}$ under which $\mc{CC}(Q)$ is identified to $Y\times[-1,1]$. Thurston shows that $\partial\mc{CC}(Q)=\partial^+\mc{CC}(Q)\sqcup\partial^-\mc{CC}(Q)$ has the structure of an embedded convex pleated surface, a fact that will be useful later on. 

By groundbreaking work of Minsky \cite{M10} and Brock, Canary, and Minsky \cite{BrockCanaryMinsky:ELC2}, in order to understand the internal geometry of $\mc{CC}(Q)$ (and the structure of the cusps) it is crucial to analyze the set 
\[
\mc{C}(Q,L):=\{\alpha\in\mc{C}(Y)\left|\,\ell_Q(\alpha)\le L\right.\}
\]
of those simple closed curves $\alpha\in\mc{C}(Y)$ whose geodesic representative in $Q$ have length bounded by a fixed constant $L$.

We already know some curves that belong to $\mc{C}(Q,L)$: 
\begin{itemize}
    \item{By classical results of Bers, there exist pants decompositions $\nu^+,\nu^-$ of $Y$ whose geodesic representatives with respect to the hyperbolic metrics on the two components $\partial^+\mc{CC}(Q),\partial^-\mc{CC}(Q)$ have length bounded by some uniform constant $L_0$ only depending on $\Sigma$.}
    \item{By Claim \ref{claim:short surgeries}, we also have on $Y$ two simple closed curves $\alpha^+,\alpha^-$ representing $\pi_Y(\mc{D}^+),\pi_Y(\mc{D}^-)$ whose geodesic representative in $M-\gamma$ have length $\ell_{M-\gamma}(\alpha^+),\ell_{M-\gamma}(\alpha^-)\le L$ where $L$ is a uniform constant only depending on $\Sigma$. The same is true in the $\pi_1(Y)$-covering $Q$ of $M-\gamma$.}
\end{itemize} 

We use the works \cite{M10} and \cite{BrockCanaryMinsky:ELC2} to get the following. Let $\mb{T}(\gamma)$ be the $\ep_1$-Margulis neighborhood of the cusp of $M-\gamma$ where $\ep_1$ is the constant of Lemma \ref{lem:thick to thick}. Recall that we chose $Y$ to be the connected component of $\Sigma-\gamma$ containing the curve $\gamma_1$ which is a boundary component of a tubular neighborhood $A$ of $\gamma$ in $\Sigma$ such that $S_{\gamma_1}\ge S_\gamma(\mc{D}^-,\mc{D}^+)/2$. Denote by $\mb{T}_Q$ the cusp lift of $\mb{T}(\gamma)$ corresponding to $\gamma_1$ and let $T_Q=\partial\mb{T}_Q$ be its boundary. We have the following:

\begin{claim}
\label{claim:thick_core}
    There exists a constant $c>0$ only depending on $\Sigma$ such that the distance $d$ on $T_Q$ between the two components of $\partial\mc{CC}(Q)\cap T_Q$ is at least $S_\gamma/c$.
\end{claim} 

\begin{proof}
Recall that $\nu^+,\nu^-$ are pants decompositions of uniformly bounded length $\ell(\nu^+),\ell(\nu^-)\le L_0$ (where $L_0$ only depends on $\Sigma$), on $\partial^+\mc{CC}(Q),\partial^-\mc{CC}(Q)$. By work Minsky \cite{M10} and Brock, Canary, and Minsky \cite{BrockCanaryMinsky:ELC2}, it is possible to associate to $\nu^+,\nu^-$ a {\em model manifold} $\mb{Q}(\nu^+,\nu^-)$ (which is a manifold homeomorphic to $Y\times[-1,1]$ endowed with a path metric) and a $K$-bilipschitz homeomorphism $\mb{Q}(\nu^+,\nu^-)\to\mc{CC}(Q)$ (where $K$ only depends on $\Sigma$), called the \emph{model map}. 

The model manifold is made of simple building blocks glued together. We do not need to discuss all of them as we are only interested in the part of the model that is mapped to $T_Q\cap\mc{CC}(Q)$. 

The pre-image of $T_Q$ in $\mb{Q}(\nu^+,\nu^-)$ is a tower of flat annuli all isometric to a fixed standard one. Their number is coarsely equal to the number of 4-edges $e$ in the hierarchy associated with $\nu^+,\nu^-$ whose domain $D(e)\subset Y$ having $\gamma_1$ as a boundary component (see \cite{M10} for the terminology, here we will only need the next formula). It is coarsely bounded from below by 
\[
S_{\gamma_1}(\nu^+,\nu^-)=\sum_{W\subset\Sigma-A\;,\;\gamma_1\subset\partial W}{\{\{d_W(\nu^+,\nu^-)\}\}_K}
\]
(see \cite[Section 9]{M10}, in particular \cite[Equation 9.26]{M10}). As the model map is $K$-bilipschitz with $K$ only depending on $\Sigma$, we deduce that $d$ is coarsely bounded from below by $S_{\gamma_1}(\nu^+,\nu^-)$.

Hence, the claim boils down to the comparison between $S_{\gamma_1}(\nu^+,\nu^-)$ and the term $S_\gamma(\mc{D}^+,\mc{D}^-)$. The relation between the two passes through the structure of the set $\mc{C}(Q,L)$. By \cite[Theorem 1.2]{BBCM}, for every essential $W\subset Y$ with $\gamma_1\subset W$ we have, in $\mathcal C(W)$,
\[
d_{\rm Haus}(\pi_W(\mc{C}(Q,L),{\rm hull}(\pi_W(\nu^+),\pi_W(\nu^-)))\le D_1,
\]
where the hull is the union of all geodesics connecting the given sets.
Recall that $\alpha^-,\alpha^+\in\mc{C}(Q,L)$ and $\alpha^-\in\pi_Y(\mc{D}^-),\alpha^+\in\pi_Y(\mc{D}^+)$. Furthermore, for every essential subsurface $W\subset Y$ with $\gamma\subset\partial W$, we can also choose $\alpha^-,\alpha^+$ to intersect essentially $W$ (see the remark at the end of Proposition \ref{prop:short_projection}). Thus, by property of the subsurface projection on nested surfaces (such as $W\subset Y$) we get $\pi_W(\alpha^-)\in\pi_W(\mc{D}^-)$ and $\pi_W(\alpha^+)\in\pi_W(\mc{D}^+)$.  

In particular, we have
\begin{align*}
d_W(\nu^-,\nu^+) &\ge d_W(\alpha^-,\alpha^+)-2D_1 &\text{by \cite{BBCM}}\\
 &\ge d_W(\mc{D}^-,\mc{D}^+)-2D_2-2D_1 &\text{because ${\rm diam}(\pi_W(\mc{D}^\pm))\le D_2$}.
\end{align*}

In conclusion, $S_{\gamma_1}(\nu^+,\nu^-)$ is coarsely uniformly bounded from below by the quantity $S_{\gamma_1}(\mc{D}^+,\mc{D}^-)$. This finishes the proof of the claim.
\end{proof}

\subsubsection{The length of the meridian}
The next step in the proof of the theorem is to obtain a lower bound on the length of the canonical meridian $\mu\subset T$ associated with $M$. The idea here is that a large portion of $T_Q\cap\mc{CC}(Q)$ embeds in $T$ under the covering projection. In order to prove a statement along these lines we need to identify and locate all the lifts of the cusp $\mb{T}(\gamma)$ of $M-\gamma$ to $Q$. 

To this purpose, we prove the following claim.

\begin{claim}
\label{claim:conjugate}
    Let $\alpha\in\pi_1(Y)$ be a non-peripheral element. Consider $\kappa\in\pi_1(T)-\pi_1(Y)$. Then $\kappa\alpha\kappa^{-1}\not\in\pi_1(Y)$.
\end{claim}

\begin{proof}
The case that $\gamma$ is non-separating is \cite[Lemma 4.10]{FSVa}.
If not, we use the description of the fundamental group of $\pi_1(M-\gamma)$ from Lemma \ref{lem:graph_of_groups}, and in particular we use the action of $\pi_1(M-\gamma)$ on the Bass-Serre $\mathcal T$ tree of the splitting described in the lemma. We have that $\pi_1(Y)$ is the stabilizer of an edge $e$ of $T$. In particular, for $\kappa\in \pi_1(M-\gamma)-\pi_1(Y)$, we have that $\kappa\alpha\kappa^{-1}$ fixes $\kappa e$, which is distinct from $e$, and if it lied in $\pi_1(Y)$ then it would fix both $e$ and $\kappa e$. Therefore, it would fix a geodesic in $\mathcal T$ containing $e$ and $\kappa e$, and therefore it would fix an edge $e$ and another edge $e'$ that shares an endpoint $v$ with $e$. Therefore, we have to show that only peripheral elements of $\pi_1(Y)$ can be contained in either a distinct conjugate of $\pi_1(Y)$ in $\pi_1(H^\pm)$ or $\pi_1(Y')$, where $Y'$ is the other component of $\Sigma-\gamma$. This can be translated into a statement about homotopies in $H^\pm$ between loops in $\Sigma-\gamma$, and the same argument as in Claim \ref{claim:parabolics} allows us to conclude.
\end{proof}

We now construct a loop deep inside the convex core of $Q$. We use it to control the relative position of different lifts of $\mb{T}(\gamma)$.

\begin{claim}
\label{claim:short_loop}
    There exists a closed loop $\alpha\subset Q$ with the following properties
\begin{itemize}
\item{It is based at ${\hat x}\in T_Q$ with $d_{T_Q}({\hat x},\partial\mc{CC}(Q)\cap T)\ge d/2$ where $d$ is the horospherical distance between the two components of $\partial\mc{CC}(Q)\cap T_Q$.}
\item{It has length bounded by $\ell(\alpha)<L$ where $L$ is a uniform constant only depending on $\Sigma$.}
\end{itemize}
\end{claim}

\begin{proof}
By the Filling Theorem \cite{C96}, we have that for every ${\hat x}\in Q$ the length of the shortest non-trivial geodesic loop based at ${\hat x}$ and not homotopic into a cusp is bounded by some uniform constant $L_g$. So, it is enough to choose ${\hat x}\in T_Q\cap\mc{CC}(Q)$ as the midpoint of the shortest horospherical segment joining the two components of $\partial\mc{CC}(Q)\cap T_Q$ and $\alpha$ as the shortest non-trivial geodesic loop based at ${\hat x}$.
\end{proof}

Choose $\alpha$ and ${\hat x}\in T_Q\cap\mc{CC}(Q)$ as provided by the previous claim.

Let $\mu\subset T$ be any simple closed geodesic passing through $x=p({\hat x})$ and intersecting essentially $(\Sigma-\gamma)\cap T$. For example, the geodesic representative on $T$ of the meridian of $\gamma$ in $M$ can be chosen to have this property.

We lift $\mu p(\alpha)$ to an arc $\tau$ on $Q$ based at ${\hat x}$. 

\begin{claim}
\label{claim:not_lift}
    The arc $\tau$ is not closed and has its endpoint different from ${\hat x}$ on a component $\mc{O}$ of the preimage of the cusp under the covering map $p:Q\to M-\gamma$ which is different from the cuspidal components of $Q$.
\end{claim}

\begin{proof}
If $\tau$ has the terminal point on $T$, then so does the lift of $\mu p(\alpha)\mu^{-1}$.
Hence, the element $\mu\alpha\mu^{-1}$ belongs to $\pi_1(Y)$. This contradicts Claim \ref{claim:conjugate}.
\end{proof}

\begin{claim}
\label{claim:horoball_disjoint}
    Provided that the Margulis constant $\ep_1>0$ satisfies the assumptions of Lemma \ref{lem:thick to thick}, then the horoball $\mc{O}$ is disjoint from $\mc{CC}(Q)$.
\end{claim}

\begin{proof}
This follows from the same argument of \cite[Lemma 4.11]{FSVa} which we briefly recall. Consider the component $\partial^+\mc{CC}(Q)$ of $\partial\mc{CC}(Q)$. As we remarked above, it is a pleated surface. The restriction of the covering projection to it $p:\partial^+\mc{CC}(Q)\to M-\gamma$ provides a pleated surface in $M-\gamma$ homotopic to the inclusion $Y\to M-\gamma$ (in particular, the pleated surface is type preserving). 

Suppose that $\partial^+\mc{CC}(Q)$ intersects $\mc{O}$ and choose ${\hat z}$ in the intersection. As $p({\hat z})\in p(\mc{O})=\mb{T}(\gamma)$ and since for a (type preserving) pleated surface $p:\partial^+\mc{CC}(Q)\to M-\gamma$ only the $\ep_0$-thin part can be mapped to the $\ep_1$-thin part (by Lemma \ref{lem:thick to thick}), we must have that ${\hat z}$ is in the $\ep_0$-thin part of $\partial^+\mc{CC}(Q)$. Let $\eta$ be a geodesic loop in the component of the $\ep_0$-thin part of $\partial^+\mc{CC}(Q)$ containing ${\hat z}$ of length $\ell(\eta)<2\ep_0$ (it represents the core of the Margulis region). As $p$ induces an injection at the level of the fundamental group, the curve $p(\eta)$ is essential in $M-\gamma$ and belongs to a $\ep_0$-Margulis region $\mb{T}(p(\eta))$ of $M-\gamma$. Since $p(\eta)$ intersects $\mb{T}(\gamma)$ and distinct Margulis regions are disjoint, we conclude that $\mb{T}(p(\eta))=\mb{T}(\gamma)$. As the only $\eta\subset Y$ which are parabolic are $\gamma_1,\gamma_2$, we deduce that $\eta$ is homotopic to $\gamma_1,\gamma_2$. So ${\hat z}$ is contained in the $\ep_0$-cuspidal part of $Q$. 

In order to conclude, we observe that the connected components of the pre-images in $Q$ of the $\ep_1$- and $\ep_0$-cuspidal regions $\mb{T}_{\ep_1}(\gamma),\mb{T}_{\ep_0}(\gamma)$ of $M-\gamma$ are nested and in bijective correspondence. In particular, as ${\hat z}$ does not belong to the $\ep_1$-cuspidal part, it also does not belong to the $\ep_0$-cuspidal part. This contradicts what we proved above. 

The same argument applies to $\partial^-\mc{CC}(Q)$. Therefore we have that $\partial\mc{CC}(Q)\cap\mc{O}=\emptyset$. As $\mc{O}$ is isometric to a horoball, it cannot be contained in $\mc{CC}(Q)$ (for example, it has infinite volume). The only remaining possibility is $\mc{CC}(Q)\cap\mc{O}=\emptyset$.
\end{proof}

We can now deduce the following estimate on the length of $\mu$. By the above discussion, $\mc{O}$ lies outside $\mc{CC}(Q)$. Hence the arc $\tau$ must intersect $\partial\mc{CC}(Q)$. We conclude that $\ell(\tau)$ is bounded from below by the distance on $T_Q$ between ${\hat x}$ and the two components of $\partial\mc{CC}(Q)\cap T_Q$ which, by our choice of ${\hat x}$, is bounded from below by half of the horospherical distance between the two components. By Claim \ref{claim:thick_core}, the latter is bounded from below by $S_\gamma(\mc{D}^+,\mc{D}^-)/c$, so that $\ell(\tau)\geq S_\gamma(\mc{D}^+,\mc{D}^-)/(2c)$. Lastly, observe that $\ell(\tau)=\ell(\alpha)+\ell(\mu)$ and that $\ell(\alpha)$ is uniformly bounded by Claim \ref{claim:short_loop}. Therefore, if $S_\gamma(\mc{D}^+,\mc{D}^-)$ is sufficiently large then we have
\[
{\rm Length}(\mu)\geq S_\gamma/3c.
\]

On the other hand, if $S_\gamma$ is bounded, then the claim just follows from the fact that ${\rm Length}(\mu)\ge 2\ep_1$ as ${\rm inj}(T)\ge\ep_1$ by the definition of standard $\ep_1$-Margulis neighborhood.

\subsubsection{The normalized length of the meridian}
In the last step of the proof we consider normalized length rather than length. Hence we need to bring into the picture the area of $T$ as we recall that the normalized length is defined to be
\[
L=\frac{{\rm Length}(\mu)}{\sqrt{{\rm Area}(T)}}.
\]

We show that 
\[
L^2\ge{\rm Length}(\mu)/\ep_1.
\]

This boils down to the following observation.

\begin{claim}
\label{claim:bounded_area}
We have ${\rm Area}(T)\le{\rm arcsinh}(2\ep_1)\cdot{\rm Length}(\mu)$.
\end{claim}

\begin{proof}
Recall that ${\rm inj}_x(M-\gamma)=\ep_1$ for every $x\in T$ and that the injectivity radius at $x$ is equal to the length of the shortest geodesic loop based at $x$. 

By the structure of the Margulis neighborhood of cusps, we can assume that the shortest geodesic loop $\alpha$ based at $x\in T$ is contained in $\mb{T}(\gamma)$ and that it is homotopic rel endpoints to a flat simple closed geodesic $\beta\subset T$ passing through $x$. By basic hyperbolic geometry, the length of $\alpha,\beta$ are related by $\sinh(\ell(\beta))=\ell(\alpha)$. Hence $\ell(\beta)={\rm arcsinh}(2\ep_1)$. 

Let $\mu$ intersect $\beta$ essentially (as both $\mu$ and $\beta$ are simple closed geodesics on a flat torus, this case happens as soon as $\ell(\mu)>{\rm arcsinh}(2\ep_1)$ which we can ensure assuming $S_\gamma$ is large enough). By Euclidean geometry, the area of the flat torus $T$ is bounded by ${\rm Area}(T)\le\ell(\mu)\ell(\beta)$.
\end{proof}

This concludes the proof of the theorem.
\end{proof}

\subsection{Lower bound on the area of $T$}
\label{subsec:area bound}

We can also give a lower bound on the area of $T=\partial \mb{T}(\gamma)$ using some of the arguments above and the structure of the covering $p:Q\to M-\gamma$. We have the following.

\begin{proposition}\label{prop:aSgamma<Area}
There exists $a=a(\Sigma)>0$ 
such that the following holds. For any fixed Margulis constant $\epsilon>0$, let $\mb{T}(\gamma)$, denote the cusp neighborhood corresponding to $\gamma$ for the Margulis constant $\epsilon$, and let $T=\partial \mb{T}(\gamma)$. Then
\[
{\rm Area}(T)\geq a\cdot S_\gamma \epsilon^2.
\]
\end{proposition}

\begin{proof}
It suffices to prove the proposition for the Margulis constant given by Claim \ref{claim:horoball_disjoint}.

If $S_\gamma$ is bounded, then the inequality holds since ${\rm Area}(T)$ is bounded from below in terms of the Margulis constant, so we can assume that $S_\gamma$ is large enough that we have the following. Firstly, we can use the claims in the proof of Theorem \ref{thm:annular}. Secondly, using Claim \ref{claim:thick_core}, we have that the horospherical distance $d\geq S_\gamma/c$ between the two components of $\partial\mc{CC}(Q)\cap T_Q$ is large enough that we can find an annulus $A\subset\mc{CC}(Q)\cap T_Q$ with the following properties.
\begin{itemize}
    \item{Its core curve is in the homotopy class of $\gamma$.}
    \item{The boundary of $A$ has length uniformly bounded by a constant $b$ only depending on $\Sigma$ (to arrange this, we arrange that each boundary component of $A$ is the boundary of a cuspidal neighborhood of a pleated surface of given genus, hence fixed area), and $b\leq d/10$.}
    \item{The distance between the boundary components of $A$ lies in $[d/100,d/10]$.}
    \item{The horospherical distance between $A$ and $\partial\mc{CC}(Q)\cap T_Q$ is at least $d/10$.}
\end{itemize}  

We show that the restriction of the covering projection $p:Q\to M-\gamma$ is an embedding on $A$. It is enough to show that $p:A\to T$ is injective. Suppose by contradiction that it is not injective. Let $x,y\in A$ be distinct points such that $p(x)=p(y)$. Denote by $\tau$ a shortest flat geodesic segment between them. By our assumptions, $\ell(\tau)\le\sqrt{b^2+d^2/10^2}$ (the length of the diagonal in a rectangle of sides of length $b$ and $d/10$). Let $\alpha$ be a shortest non-trivial geodesic loop based at $x$ and not homotopic into a cusp. As explained in Claim \ref{claim:short_loop}, $\alpha$ has length bounded by a uniform constant $\ell(\alpha)\le b$ only depending on $\Sigma$. We lift $p(\alpha)$ to $Q$ with basepoint $y$. The lift has its terminal point different from $y$ on the boundary of some preimage $\mc{O}$ of the cusp of $M-\gamma$. By the same argument as in Claim \ref{claim:not_lift} we have that $\mc{O}$ is different from $\mb{T}(\gamma)$ and, hence, by Claim \ref{claim:horoball_disjoint}, it lies outside $\mc{CC}(Q)$. Thus the distance between $x$ and $\partial\mc{CC}(Q)\cap T$ is at most $\ell(\tau)+b\le\sqrt{b^2+d^2/10^2}$. However, by our choices, it is also larger than $d/4$. Thus we obtained a contradiction.
\end{proof}

\section{Move sequences in the pants graph}
\label{sec:move_seq}

Our next goal is to show Proposition \ref{augmenting minsky} below which roughly says that any two pants decompositions on a punctured surface can be efficiently connected by a path of moves where all the moves that need to happen far from a puncture occur in certain specified subpaths. This is a key input in Section \ref{sec:good_homot} where we turn the combinatorially controlled sequence of pants decompositions into a geometrically controlled sequence of homotopies between pleated surfaces using the good homotopy machinery of Minsky \cite{M01}. As explained in Section \ref{sec:strategy} and in Figure \ref{fig:sweepout1}, this is an important step for constructing a model for the standard meridian $\mu\subset T$ associated with $M$ on the boundary of a standard Margulis neighborhood of the cusp of $M-\gamma$.

\begin{proposition}
\label{augmenting minsky}
For any finite-type surface $\Sigma$ (with at least one puncture) the exists $C_0$ such that for all $C\geq C_0$ there exists a constant $K\geq 1$ with the following property. Fix a puncture $p$ of $\Sigma$.
 Let $P$ and $Q$ be pants decompositions of $\Sigma$. Then there exist subsurfaces $Y_1,\dots, Y_k$ not containing $p$, an elementary move sequence $P=P_0\to\ldots\to P_n=Q$, and disjoint intervals $J_j\subseteq \{0,\dots,k\}$ satisfying the following properties. Denote
 $$S_p:=\sum_{Y\in\mathcal Y_p} \{\{d_{Y}(P,Q)\}\}_C,$$
 where $\mathcal Y_p$ is the set of all essential non-annular subsurfaces $Y$ containing $p$.
 
\begin{enumerate}
\item For $i\in J_j$, the elementary move $P_i\to P_{i+1}$ is supported in $Y_j$.
\item{$|\{0,\dots,n\}-\bigcup J_j|\leq K S_p+K$.}
\item $k\leq K S_p+K$.
\end{enumerate}
\end{proposition}

No other result from this section is used elsewhere.

\subsection{The hierarchical structure of the pants graph}

There might be a proof of the proposition by carefully constructing a resolution of a hierarchy (see \cite[Proposition 5.4]{MasurMinsky:II}), but we could not find it. Instead, the proof we give only uses the hierarchical structure of pants graphs as emerged from \cite{MasurMinsky:II} and axiomatized in \cite{HHS1}, and in particular the facts recalled below. All subsurfaces below are essential, non-annular, and regarded up to isotopy. We denote $U\nest V$ if $U$ is contained in $V$ (up to isotopy) and $U\orth V$ if they are disjoint (again up to isotopy). When they are not $\nest$- nor $\orth$-comparable, we say that they overlap and denote this by $U\transverse V$. In other words, we write $U\transverse V$ if the subsurfaces intersect and neither is contained in the other.

The crucial facts we will use are the following, where the constant $E\geq 1$ depends only on the ambient surface.

\begin{itemize}
 \item (Bounded Geodesic Image, \cite[Theorem 3.1]{MasurMinsky:II}) Let $Z\propnest Y$ and let $P,Q$ be multicurves such that $d_Z(P,Q)\geq E$. Then every geodesic in $\mathcal C(Y)$ from $\pi_Y(P)$ to $\pi_Y(Q)$ passes within distance $1$ of $\pi_Y(\partial Z)=\partial Z$.
 \item (Behrstock inequality, \cite[Theorem 4.3]{Beh06}) Let $Z\transverse Y$ and let $P$ be a multicurve such that $\pi_Z(P)$ and $\pi_Y(P)$ are well-defined. If $d_Z(P,\partial Y)>E$ then $d_Y(P,\partial Z)\leq E$.
 \item (Realization theorem, \cite[Theorem 4.3]{BKMM:qi_rigid}, \cite[Theorem 3.1]{HHS2}) For all $A\geq 0$ there exists $B\geq 0$ such that the following holds. Let $Y$ be a subsurface and suppose that for each $U$ either nested into or disjoint from $Y$ we have a fixed subset $b_U\subseteq \mathcal C(U)$ of diameter at most $B$. Suppose that for all $U,V$ either nested into or disjoint from $Y$:
 \begin{enumerate}[(a)]
     \item (Consistency 1) If $U\transverse V$ then $\min\{d_U(b_U,\partial V),d_V(b_V,\partial U)\}\leq A$,
     \item (Consistency 2) Suppose $U\propnest V$. If $d_V(\partial U,b_V)>A$ then $d_U(\pi_U(b_V),b_U)\leq A$.
 \end{enumerate}
 
 Then there exists a pants decomposition $P$ containing $\partial Y$ and such that $d_U(P,b_U)\leq B$ for all $U$ that satisfy $U\nest Y$ or $U\orth Y$.
 \item (Distance formula, \cite[Theorem 6.12]{MasurMinsky:II}) There exists $L_0>0$ such that for all $L\geq L_0$ there exists $K\geq 1$ such that following holds. Let $P$, $Q$ be pants decompositions. Then the minimal number of elementary moves to go from $P$ to $Q$ (that is, their distance in the pants graph) is
 $$\approx_{K,K} \sum_Y d_Y(P,Q),$$
 where the sum is taken over all non-annular subsurfaces.
\end{itemize}

We will also use two similar lemmas (``passing-up lemmas'') that follow from the axiomatic setup, and which we state below in the context we need. They both roughly say, in our context, that given two pants decompositions $P_1,P_2$ and sufficiently many subsurfaces $Y_i$ of some subsurface $Z$ such that $d_{Y_i}(P_1,P_2)$ are far, then there is a subsurface $Y\nest Z$ in which a lot of the $Y_i$ are nested and $d_Z(P_1,P_2)$ is large. The lemmas could be combined, but we prefer to keep them separate for clarity.

\begin{lem}
\label{lem:passing-up}
    Given a finite-type surface $\Sigma$ and constants $C,R\geq 0$ there exists $K\geq 0$ such that the following hold.
    \begin{enumerate}
        \item (\cite[Lemma 2.5]{HHS2}) Let $P, Q$ be pants decompositions and let $\mathcal A$ be a collection of non-annular subsurfaces such that $d_Y(P,Q)\geq C$ for all $Y\in\mathcal A$. If $|\mathcal A|\geq K$, then there exists $Y\in\mathcal A$ a non-annular subsurface $W$ with $Y\propnest W$ such that $d_W(P,Q)\geq C$.
        \item (\cite[Lemma 1.6]{HHS:quasiflats}) Let $P, Q$ be pants decompositions, $Z$ a subsurface, and let $\mathcal A$ be a collection of non-annular subsurfaces such that all $Y\in\mathcal A$ satisfy $Y\nest Z$ and $d_Z(P,Q)\geq C$. If $|\mathcal A|\geq K$ then there exist $Y_1,Y_2\in\mathcal A$ and a non-annular subsurface $W$ with $Y_1,Y_2\propnest W\propnest Z$ and $d_W(\partial Y_1,\partial Y_2)\geq R$.
    \end{enumerate}
\end{lem}

\subsection{Proof of the proposition}

Since the proof is quite technical, we first give an outline to highlight the main idea. The natural candidates for the subsurfaces $Y_i$ are just the maximal subsurfaces not containing the puncture with the property that $P$ and $Q$ project far onto $\mathcal C(Y_i)$. We now think of moving from $P$ to $Q$ with a sequence of elementary moves, and how the projections to the various curve graphs change along the way.

A well-known consequence of the Behrstock inequality says that if $Y_i\transverse Y_j$, there is one subsurface where we need to change the projection first before we can change the projection onto the other one. This results in a partial ordering of the subsurfaces, which we extend to a total order arbitrarily, and we change the indices of the $Y_i$ accordingly. Now, we would like to find the candidate pants decompositions $p_i,q_i$ at the beginning and end of the interval $J_i$. These will contain $\partial Y_i$, and we would like to change all the projections onto any $\mathcal C(Y)$ for $Y\nest Y_i$ that still need to change at that point of the path as we move from $p_i$ to $q_i$. In particular, we would like $\pi_Y(q_i)$ to be (close to) $\pi_Y(Q)$ for all such $Y$. By the realization theorem, there is a pants decomposition of $Y_i$ with these coarse ``coordinates'' in all such $\mathcal C(Y)$. Inductively, it would make sense to set the coordinates of $p_i$ for $Y\orth Y_i$ or $Y\nest Y_i$, as well as the coordinates of $q_i$ for $Y\orth Y_i$, coarsely equal to those of $\pi_Y(q_{i-1})$ (this is to go ``as quickly as possible'' between the pants decomposition $q_{i-1}$ and a pants decomposition $p_i$ containing $\partial Y_i$). This unfortunately does not work, because a given $Y$ can be disjoint from arbitrarily many $Y_i$, and the coordinates in the corresponding $\mathcal C(Y)$ could ``drift'' uncontrollably as we repeat the induction step. Instead, we reverse-engineer what the projections would morally be if we did this procedure and no drifting occurs, and still use the realization theorem to show that pants decompositions with the required projections do exist.

What is left to do is to show that the sum over $i$ of the number of moves needed to go from $q_{i-1}$ to $p_i$ is controlled by $S_p$, using the distance formula. To do so, we check that for any $\mathcal C(Y)$ with $Y\in\mathcal Y_p$ the projections of the intermediate points occur aligned along a geodesic from $\pi_Y(P)$ to $\pi_Y(Q)$, while there is no contribution for $Y\notin\mathcal Y_p$, meaning that $q_{i-1}$ and $p_i$ project close onto any such $Y$.

\begin{proof}[Proof of Proposition~\ref{augmenting minsky}]
In this proof we do not explicitly keep track of constants. Except in one part of the proof that we specify below, when we say that two subsets of a curve graph coarsely coincide we mean that they lie within distance $100B$ of each other, where $B$ is the constant from the realization theorem for any subsurface of $\Sigma$ and for $A=100E$, with $E$ the constant from the hierarchical properties of the pants graph; we can assume $B\geq E$. We denote this by the symbol $\approx$. Sometimes we will deduce from $A_1\approx A_2\approx A_3$ that $A_1\approx A_3$, which is formally not correct, but in those contexts the Hausdorff distances between $A_i$ and $A_{i+1}$ are in fact at most $50B$. When we say that two subsets are far we mean that they lie at least $10^6B$ away from each other. Throughout the proof we assume $C\geq 10^9B$ (so that $C\geq 10^9E$ as well).

Let $\mathcal A$ be the set of all $\nest$-maximal non-annular subsurfaces $Y$ of $\Sigma$ that do not contain $p$ and with the property that $d_{Y}(P,Q)\geq C$.

\par\medskip

\begin{claim}
    \label{claim:passing_up}
    There exists $K=K(\Sigma,C)$ such that $|\mathcal A|\leq K S_p+K$.
\end{claim}

\begin{proof}
Consider first the subset of $\mathcal A$ consisting of all $Y\in\mathcal A$ that are not properly nested into some $Z$ with $d_Z(P,Q)\geq C$ (which necessarily contains the puncture). The number of those is bounded by passing up, Lemma \ref{lem:passing-up}-(1), say by some constant $K$.

For the others, fix a $\nest$-minimal non-annular subsurface $Z=Z(Y)$ that contains $p$ with $d_Z(P,Q)\geq C$. Consider some non-annular subsurface $Z$, and all $Y$ with $Z=Z(Y)$ and $\pi_Z(\partial Y)$ contained in a fixed ball in $\mathcal C(Z)$. The cardinality of the set of all such $Y$ is bounded in terms of the radius of the ball by passing-up, Lemma \ref{lem:passing-up}. This allows us to conclude that there are at most $K\{\{d_{Y}(P,Q)\}\}_C$ many $Y$ with $Z(Y)=Z$, for some constant $K$ bigger or equal than the earlier $K$, since $\pi_Z(\partial Y)$ is contained in a uniform neighborhood of a geodesic from $\pi_Z(P)$ to $\pi_Z(Q)$ by the Bounded Geodesic Image if $d_Z(P,Q)\geq C$.

All in all, we have $|\mathcal A|\leq K S_p+K$ as desired.
\end{proof}

Let us order the elements of $\mathcal A$ as $Y_1,\dots, Y_k$ in a way that if $i<j$ and $Y_i\transverse Y_j$ then
\begin{itemize}
    \item $\pi_{Y_i}(Q)$ coarsely coincides with $\pi_{Y_i}(\partial Y_j)$, and
    \item $\pi_{Y_j}(P)$ coarsely coincides with $\pi_{Y_j}(\partial Y_i)$.
\end{itemize}
This is possible since it is well-known (see e.g. \cite[Proposition 2.8]{HHS2}) that the Behrstock inequality implies the relation $\preceq$ defined by $Y\preceq Z$ if $Y\transverse Z$ and $\pi_{Z}(P)$ coarsely coincides with $\pi_{Z}(\partial Y)$ is a partial order. Moreover, $Y\preceq Z$ is equivalent to $Y\transverse Z$ and $\pi_Y(Q)$ coarsely coincides with $\pi_Y(\partial Z)$. The partial order $\preceq$ can then be embedded into a total order. 

{\bf Intermediate pants decompositions.} The key to this proof is the following construction of various intermediate pants decompositions (obtained ``projecting'' to the $Y_i$).

For notational purposes, we set $P=q_0$ and $p_{k+1}=Q$ and $\partial Y_0=Y_0=P$.

The intermediate pants decompositions are the ones given by the following claims.

\begin{claim}
\label{claim:p}
    For $0\leq i\leq k+1$, there exists a pants decomposition $p_i$ containing $\partial Y_i$ and such that
\begin{itemize}
    \item whenever $Y\orth Y_i$ or $Y\nest Y_i$, let $j<i$ be maximal such that $Y_j$ either cuts $Y$ or $Y\nest Y_j$. In the former case $\pi_Y(p_i)\approx \pi_Y(\partial Y_j)$, while in the latter $\pi_Y(p_i)\approx \pi_Y(Q)$.
\end{itemize}
\end{claim}

\begin{claim}
\label{claim:q}
    For $0\leq i\leq k+1$, there exists a pants decomposition $q_i$ containing $\partial Y_i$ and such that
\begin{itemize}
    \item whenever $Y\orth Y_i$, let $j<i$ be maximal such that $Y_j$ cuts $Y$. Then $\pi_Y(q_i)\approx \pi_Y(\partial Y_j)$,
    \item whenever $Y\nest Y_i$, we have $\pi_Y(q_i)\approx \pi_Y(Q)$.
\end{itemize}
\end{claim} 

\begin{rmk}
    We will only spell out the proof of Claim \ref{claim:p}. In the proof, we will check the consistency conditions, and in order to prove Claim \ref{claim:q} the same argument applies to show consistency whenever $U,V\orth Y_i$, while consistency for $U,V\nest Y_i$ is clear.
\end{rmk}

\begin{proof}[Proof of Claim \ref{claim:p}.]
We will check that the realization theorem applies to yield pants decompositions of both $Y_i$ and its complement, which can then be combined with $\partial Y_i$ to a pants decomposition of $\Sigma$. To be explicit, for all $U$ either nested or disjoint from $Y_i$ we choose $b_U\coloneqq\pi_U(\partial Y_j)$ if $Y_j$ cuts $U$ and $b_U\coloneqq\pi_U(Q)$ if $U\nest Y_j$ and check consistency 1 and consistency 2.
This is the part of the proof where being close means something more restrictive, meaning being within distance $100E$. 

Let $U\transverse V$ with each either disjoint or nested into $Y_i$. Let $j<i$ and $l<i$ be maximal such that $Y_j$ cuts or contains $U$ and $Y_l$ cuts or contains $V$. If they both contain $U$/$V$, then consistency for $Q$ (i.e.~the fact consistency holds when $b_U$ is chosen to be $\pi_U(Q)$ for all $U$ by the Behrstock inequality (for consistency 1) and transitivity of projection (for consistency 2)) implies consistency 1 for $U,V$. Suppose that they both cut.
If $j=l$, then consistency 1,
i.e.~\[\min\{d_U(\pi_U(\partial Y_j),\partial V),d_V(\pi_V(\partial Y_l),\partial V)\}\leq A,\]
is satisfied by the Behrstock inequality, hence we consider the case $j\neq l$.
If we have $j<l$, then $Y_l$ does not cut $U$ by maximality of $j$; in particular, $Y_l$ and $U$ do not overlap, and thus $\pi_V(\partial Y_l)\approx \pi_V(\partial U)$ as desired for consistency.
If $j>l$, then $\pi_U(\partial Y_j)\approx \pi_U(\partial V)$, again implying consistency 1.
Finally, assume, up to swapping $U$ and $V$, $U\nest Y_j$ and that $Y_l$ cuts $V$. We have to show that either $\pi_U(Q)\approx \pi_U(\partial V)$ or $\pi_V(\partial Y_l)\approx \pi_V(\partial U)$. Note that $V$ cuts $Y_j$ since it cuts $U$ and $U\nest Y_j$. By maximality of $l$ we then have $l\geq j$.
If $\pi_V(\partial Y_l)\approx\pi_V(\partial U)$ we are done.
If not, then in particular $\partial U$ and $\partial Y_l$ are not disjoint, hence $Y_l\transverse U$. Since $Y_l$ cuts $U$, we have $l\leq j$ by maximality of $j$. Hence, we have $j=l$. But then $\pi_V(\partial Y_l)=\pi_V(\partial Y_j)\approx \pi_V(\partial U)$ since $U\nest Y_j$, which means this case cannot occur.

Let $U\propnest V$, with each either disjoint or nested into $Y_i$.
If they both contain $U$/$V$, then consistency for $Q$ yields consistency 2 for $U,V$.
Let $j<i$ and $l<i$ be maximal such that $Y_j$ cuts or contains $U$ and $Y_l$ cuts or contains $V$.
Suppose that they both cut.
If $j=l$, then consistency 2 holds, since
$\pi_U(b_V)\coloneqq
\pi_U(\pi_V(\partial Y_j))
=\pi_U(\partial Y_j)\eqqcolon
b_U$;
thus, we consider the case $j\neq l$.
If we have $j<l$, then $Y_l$ does not cut $U$ by maximality of $j$, hence $\pi_V(\partial Y_l)\approx \pi_V(\partial U)$, and consistency 2 holds.
On the other hand, we cannot have $j>l$ since any subsurface that cuts $U$ also cuts $V$.
Suppose now $U\nest Y_j$ and that $Y_l$ cuts $V$.
Then either $V\nest Y_j$ or $Y_j$ cuts $V$,  hence $j\leq l$.
If $\pi_V(\partial Y_l)\approx \pi_V(\partial U)$, then we are done, and if not $Y_l$ must cut $U$, showing $l\leq j$.
Therefore, we have $j=l$.
But then $\pi_V(\partial Y_l)=\pi_V(\partial Y_j)\approx \pi_V(\partial U)$ since $U\nest Y_j$. So actually this case could not occur.
Finally, suppose that $Y_j$ cuts $U$ and that $V\nest Y_l$. Then $U\nest Y_l$, so that $l\leq j$. But since $Y_j$ cuts $U$, it also cuts $V$, so that $j\leq l$. But then $j=l$, which would imply $U\nest Y_j$, so this case cannot occur.
\end{proof}

 The sequence of elementary moves that we will show to have the required properties goes through $P=q_0,p_1,q_1,p_2,\dots, p_{k+1}=Q$, and it is a shortest sequence between $q_i$ and $p_{i+1}$ for all $i$, and any sequence between $p_i$ and $q_i$ supported in $Y_i$. The intervals $J_j$ are the intervals corresponding to sequences between $p_j$ and $q_j$.

Property (1) holds by construction, while (3) holds by Claim \ref{claim:passing_up}.

{\bf Controlling distance formula terms.} We now need two claims to control the projections of the intermediate pants decompositions to all possible subsurfaces, which will then allow us to do suitable ``accounting'' with the distance formula.

\par\medskip

\begin{claim}
\label{claim:monotonic}
    There exists a constant $M$ depending only on the surface such that the following holds. Fix a subsurface $Y\in\mathcal Y_p$ and let $\alpha$ be a geodesic from $\pi_Y(P)$ to $\pi_Y(Q)$ in $\mathcal C(Y)$. Then the projections $\pi_Y(q_0),\pi_Y(p_1),\pi_Y(q_1), \dots$ lie within distance $M$ of points occurring monotonically along $\alpha$.
\end{claim}  

\begin{proof}
For later purposes, we note that below we consider several cases depending on the relation between $Y$ and certain $Y_l$, and all arguments for the various cases are valid regardless of $Y$ containing the puncture or not.

Note that $\pi_Y(q_0)=\pi_Y(P)$ and $\pi_Y(p_{k+1})=\pi_Y(Q)$, so we can consider $p_i$ and $q_j$ with $i\in \{1,\dots,k\}$.

Let us consider $x=\pi_Y(p_i)$ and $y=\pi_Y(q_i)$, aiming to show that $x,y$ lie close to points on $\alpha$, with the point close to $x$ occurring closer to $\pi_Y(P)$ than the one close to $y$. There are two cases (note that $Y\nest Y_i$ cannot occur since $Y$ contains the puncture but $Y_i$ does not), and in all cases we can see that the required conclusion holds.
\begin{itemize}
    \item if $Y\transverse Y_i$ or $Y_i\propnest Y$, then both $x$ and $y$ coarsely coincide with $\pi_Y(\partial Y_i)$, which lies close to $\alpha$. To see the latter, if $Y_i\propnest Y$ apply the Bounded Geodesic Image Theorem, and if $Y\transverse Y_i$ this follows from the Behrstock inequality (as $\pi_{Y_i}(\partial Y)$ is far from at least one of $\pi_{Y_i}(P)$ and $\pi_{Y_i}(Q)$). Therefore, $\pi_Y(\partial Y_i)$ is close to one of $\pi_{Y}(P)$ or $\pi_{Y}(Q)$).
    \item If $Y\orth Y_i$ then $x$ and $y$ coarsely coincide by construction.
\end{itemize}

Note that we already argued in particular that each $\pi_Y(p_i),\pi_Y(q_i)$ occurs close to $\alpha$.

Let now $x=\pi_Y(r_i)$ and $y=\pi_Y(r_j)$ for some $i<j$ and $r\in\{p,q\}$. With the same aim as above, the possible cases are the following.
\begin{itemize}
    \item $Y_i,Y_j$ are both properly nested into $Y$. If $Y_i\transverse Y_j$ then by the ordering of the $Y_l$ we have $\pi_{Y_j}(\partial Y_i)\approx \pi_{Y_j}(P)$. In particular $\pi_{Y_j}(\partial Y_i)$ is far from $\pi_{Y_j}(Q)$, so that by the the Bounded Geodesic Image, $\pi_Y(\partial Y_j)\approx y$ needs to occur close to a geodesic from $\pi_Y(\partial Y_i)\approx x$ to $\pi_Y(Q)$, as required. If $Y_i\orth Y_j$ then $\pi_Y(\partial Y_i)\approx \pi_Y(\partial Y_j)$ (projections of disjoint curves are close) so that $x\approx \pi_Y(\partial Y_i)$ and $y\approx \pi_Y(\partial Y_j)$ coarsely coincide.

    \item $Y_i,Y_j$ both overlap $Y$. In this case $x\approx \pi_Y(\partial Y_i)$ and $y\approx \pi_Y(\partial Y_j)$ and each of $\pi_Y(\partial Y_i),\pi_Y(\partial Y_j)$ coarsely coincides with either $\pi_Y(P)$ or $\pi_Y(Q)$ by the Behrstock inequality (indeed, if this was not the case for, say, $i$, then we would have $\pi_{Y_i}(P)\approx \pi_{Y_i}(\partial Y)\approx \pi_{Y_i}(Q)$ contradicting that there is a large projection onto $\mathcal C (Y_i)$).  The only case to rule out is that one where $\pi_Y(\partial Y_i)\approx \pi_Y(Q)$ and $\pi_Y(\partial Y_j)\approx \pi_Y(P)$. But this would say that $\pi_Y(\partial Y_j)\approx \pi_Y(P)$ are both far from $\pi_Y(\partial Y_i)$, which in particular implies $Y_i\transverse Y_j$. Moreover, the Behrstock inequality implies $\pi_{Y_i}(\partial Y_j)\approx \pi_{Y_i}(\partial Y)\approx \pi_{Y_i}(P)$, which is not the case by the ordering of the $Y_l$.
    \item $Y_i\propnest Y$ and $Y_j\transverse Y$. In this case, once again by the Bounded Geodesic Image, $x\approx \pi_Y(\partial Y_i)$ occurs close to $\alpha$, and $y\approx \pi_Y(\partial Y_j)$ either coarsely coincides with $\pi_Y(P)$ or $\pi_Y(Q)$, so we have to rule out that $\pi_Y(\partial Y_i)$ is far from $\pi_Y(P)$ and $\pi_Y(\partial Y_j)\approx \pi_Y(P)$. In this case $\pi_Y(\partial Y_i)$ is far from  $\pi_Y(\partial Y_j)$ and in particular $Y_i\transverse Y_j$. Moreover, the Bounded Geodesic Image would yield $\pi_{Y_i}(\partial Y_j)\approx \pi_{Y_i}(P)$, which contradicts the ordering of the $Y_l$.
    \item The case where $Y_j\propnest Y$ and $Y_i\transverse Y$ is symmetric to the previous one.

\item $Y\orth Y_i$ or $Y\orth Y_j$. In this case $x\approx \pi_Y(\partial Y_{i'})\approx \pi_Y(p_{i'})$ for some $i'\leq i$ with $Y_{i'}$ that cuts $Y$, and similarly for $j$, yielding $j'$. Also, $i'\leq j'$, so this reduces to one of the previous cases.
    
\end{itemize}
\end{proof}

\begin{claim}
\label{claim:no_proj}
    There exists a constant $M$ depending only on the surface such that the following holds. Fix a subsurface $Y\notin\mathcal Y_p$. Then for all $i$ we have $d_Y(q_i,p_{i+1})\leq M$.
\end{claim}

\begin{proof}
We analyse the cases where $Y\nest Y_i$ or $Y\nest Y_{i+1}$.

Suppose first $Y\nest Y_i$, so that $\pi_Y(q_i)\approx \pi_Y(Q)$. If $Y$ is nested into or orthogonal to $Y_{i+1}$ then the maximal index $j<i+1$ with $Y_j$ cutting or intersecting $Y_{i+1}$ is $j=1$, so by construction $\pi_Y(p_{i+1})\approx \pi_Y(Q)$. If not, then $Y_{i+1}$ cuts $Y$ and $\pi_Y(p_{i+1})\approx \pi_Y(\partial Y_{i+1})$. Since $Y_{i+1}$ cannot be nested into $Y$ as the latter is nested into $Y_i$, we have $Y_{i+1}\transverse Y_i$. By the ordering, $\pi_{Y_{i+1}}(P)\approx \pi_{Y_{i+1}}(\partial Y_i)$, and since $Y\nest Y_i$ we have $\pi_{Y_{i+1}}(\partial Y_i)\approx \pi_{Y_{i+1}}(\partial Y)$. In particular, $\pi_{Y_{i+1}}(\partial Y)$ is far from $\pi_{Y_{i+1}}(Q)$, so that by the Behrstock inequality we have $\pi_Y(Q)\approx \pi_Y(\partial Y_{i+1})$, as required.

Suppose now $Y\nest Y_{i+1}$, and we can assume that $Y$ is not nested into $Y_i$ as we dealt with this case above. If $Y_i$ cuts $Y$, then $\pi_Y(p_{i+1})\approx\pi_Y(\partial Y_i)$ by construction and $\pi_Y(\partial Y_i)\approx \pi_Y(q_i)$, since the pants decomposition $q_i$ contains $\partial Y_i$, so we are done. If $Y\orth Y_{i}$, then both $\pi_Y(p_{i+1})$ and $\pi_Y(q_i)$ come from the same $Y_j$ for $j<i$, so we are also done.

Up to now, we dealt with all cases where either $Y\nest Y_{i+1}$ or $Y\nest Y_i$, which we can rule out from now on. Hence, as noted at the beginning of the proof of Claim \ref{claim:monotonic}, the arguments there apply, in particular the case distinction for $i<j$, with $j=i+1$. We note that if $Y_i\propnest Y$ or $Y_{i+1}\propnest Y$ then $\pi_Y(P)$ and $\pi_Y(Q)$ coarsely coincide by maximality of the $Y_l$. With this observation, an inspection of the cases shows that the arguments yield the required conclusion except for the case $Y_i,Y_j\transverse Y$ which requires more care. In this case, the argument does not rule out $\pi_Y(P)\approx \pi_Y(q_i)\approx \pi_Y(\partial Y_i)$ and $\pi_Y(p_{i+1})\approx \pi_Y(Q)\approx \pi_Y(\partial Y_{i+1})$. If $\pi_Y(P)\approx \pi_Y(Q)$, we are done, and we now do rule out that $\pi_Y(P)$ is far from $\pi_Y(Q)$. In this case, $Y$ needs to be nested in some $Y_j$, which can neither coincide or be orthogonal to either $Y_{i+1}$ or $Y_i$, so we have $Y_j\transverse Y_i,Y_{i+1}$. If $j<i$ then $\pi_{Y_i}(\partial Y)\approx \pi_{Y_i}(\partial Y_j)\approx \pi_{Y_i}(P)$, so that $\pi_{Y_i}(\partial Y)$ is far from $\pi_{Y_i}(Q)$ and by the Behrstock inequality $\pi_Y(\partial Y_i)\approx \pi_Y(Q)$, a contradiction with $\pi_Y(P)\approx \pi_Y(\partial Y_i)$. If $j>i+1$ we get a similar contradiction, concluding this case.
\end{proof}

Let $C_0\geq 1$ be a threshold of the distance formula, fix $C\geq C_0$ and consider $L$ larger than $2M+C$, for $M$ as in Claims \ref{claim:monotonic} and \ref{claim:no_proj}. By the distance formula (in the pants graph), the number of elementary moves needed to go between $q_i$ and $p_{i+1}$ is $\sum_{Y} \{\{d_Y(q_i,p_{i+1})\}\}_L$ up to bounded multiplicative and additive error.

Therefore the total number of elementary moves outside the intervals $J_j$ is $\sum_{Y}\sum_i \{\{d_Y(q_i,p_{i+1})\}\}_L$ up to bounded multiplicative error and additive error a bounded multiple of $k$, whence of $S_p$.

For $Y\in\mathcal Y_p$, by Claim \ref{claim:monotonic} we have that $\sum_i \{\{d_Y(q_i,p_{i+1})\}\}_L$ is bounded by $(2M+1)\{\{d_Y(P,Q)\}\}_{C}$. For $Y\notin \mathcal Y_p$, we have $\sum_i \{\{d_Y(q_i,p_{i+1})\}\}_L=0$ by Claim \ref{claim:no_proj}. Therefore $\sum_{Y}\sum_i \{\{d_Y(q_i,p_{i+1})\}\}_L$ is bounded by $(2M+1)\sum_{Y\in\mathcal Y_p} \{\{d_Y(P,Q)\}\}_{C}$, and the proof is complete.
\end{proof}

\section{Good homotopies and a model for the meridian}
\label{sec:good_homot}

The goal of this section is to make Figure \ref{fig:sweepout1} more accurate. More specifically, we will describe the geometry of the cusp of $M-\gamma$, as well as its interaction with two pleated surfaces that `carry' disks from either side of the Heegaard splitting, and describe meridians of the filling from $M-\gamma$ to $M$. We make this precise in the following proposition.

\begin{proposition}
\label{prop:homotopies}
Fix Margulis constants $\ep_0$ and $\ep_1$ as in Subsection \ref{subsec:margulis}. For every integer $g\geq 2$, there exists a constant $L=L(g,\epsilon_0,\epsilon_1)$ with the following property.

Let $M=H^-\cup_\Sigma H^+$ be a Heegaard splitting of genus~$g$. Let $\gamma\subset\Sigma$ be a simple closed curve such that $(H^-,\gamma),(H^+,\gamma)$ are both pared acylindrical. Fix a finite-volume hyperbolic metric on $M-\gamma$ such that $T=\partial\mb{T}_{\ep_1}(\gamma)$ intersects $\Sigma$ in two simple closed curves and fix two $\delta^\pm\in\calD^\pm$.
Then there exist an annulus $A\subset\Sigma$ and two immersions $f^\pm:\Sigma \to M-\gamma$ with core curve $\gamma$ such that the following holds:
 \begin{enumerate}
  \item $f^\pm|_{\Sigma-{\rm int}(A)}=g^\pm|_{\Sigma-{\rm int}(A)}$ where $g^\pm$ is a pleated surface that geodesically maps the subsurface projection of $\delta^\pm\in\calD^\pm$.
  \item $f^\pm\circ\delta^\pm$ is nullhomotopic in $M-\gamma$.
  \item $A-\gamma$ is contained in the $\ep_0$-cuspidal parts of $(\Sigma-\gamma,\sigma^\pm)$ and contains the $\ep_1$-cuspidal parts of $(\Sigma-\gamma,\sigma^\pm)$. The image $f^\pm(A)$ is contained in $T$. At the level of fundamental groups $f^\pm$ maps a generator of $\pi_1(A)$ to an element of $\pi_1(T)$ corresponding to a component of $\Sigma\cap T$. The boundary $\partial A$ consists of two horocyclic loops both of length 
  at most $2\sinh(\ep_0/2)$.
 \item For every essential arc $t$ on $A$ there are two paths $\gamma_1,\gamma_2$ on $T$ of length at most $LS_\gamma$ such that the concatenation $f^+(\overline{t}),\gamma_1,f^-(t),\gamma_2$ is (well defined and) homotopic to the meridian of $\gamma$ in $M$ (here $\overline{t}$ denotes $t$ with the inverse parametrization). 
  \item The metrics $\sigma^\pm$ on $\Sigma-\gamma$ obtained pulling back the metric of $M-\gamma$ via $g^\pm$ are {\rm $(L,\epsilon_1,LS_\gamma)$-related} outside $A$ (as in Definition \ref{dfn:related_metrics}).
 \end{enumerate}
\end{proposition}

Crucially, we produce a geometrically controlled representative for the standard meridian of $\gamma$ in $M-\gamma$. It consists of two parts. 
\begin{itemize}
    \item{A vertical part which is the track of a sequence of controlled homotopies (as in Definition \ref{dfn:good homotopies}) connecting two pleated surfaces $g^-:(\Sigma-\gamma,\sigma^-)\to M-\gamma$ and $g^+:(\Sigma-\gamma,\sigma^+)\to M-\gamma$ mapping geodesically $\delta^-\cap\Sigma-\gamma$ and $\delta^+\cap\Sigma-\gamma$. Its length is controlled by $S_\gamma(\mc{D}^-,\mc{D}^+)$.}
    \item{A horizontal part whose length we will discuss in the next section. In order to control it we will use the fact that the metrics $\sigma^-,\sigma^+$ on $\Sigma-\gamma$ are related by a sequence of controlled changes (as in Definition \ref{dfn:related_metrics}).}
\end{itemize}

In Section~\ref{sec:annuli} we will use this description to prove Theorems~\ref{thm:annular} and~\ref{thm:annular2}.

\subsection{Choice of Margulis constants}
\label{subsec:margulis}

In this section we fix a Margulis constant $\ep_0$ for dimensions 2 and 3 which satisfies the conclusion of Lemma \ref{lem:straight in cusp}. We also fix $\ep_1$ as in Lemma \ref{lem:thick to thick}.

We now describe the notions of good maps, good homotopies, and related metrics that together with Proposition \ref{augmenting minsky} form the technical heart of the proof of Proposition \ref{prop:homotopies}. After that we start the proof of the proposition.

\subsection{Good maps and good homotopies}

Our approach uses tools developed by Minsky in \cite{M01}, namely good maps and good homotopies (whose definition we recall below). Most importantly, they allow us to translate sequences of elementary moves between pants decompositions into sequences of controlled homotopies between pleated surfaces.  

Let $Q$ be a hyperbolic 3-manifold diffeomorphic to $W\times\mb{R}$ where $W$ is a surface of finite type. Assume that the peripheral curves of $W$ are all parabolics.

\begin{dfn}[Good Maps {\cite[Section 4]{M01}}]
\label{dfn:good maps}
Let $P\subset W$ be a multicurve and $C>0$ a constant. A pleated map $f:(W,\sigma)\to Q$ in the homotopy class of the inclusion $W\times\{0\}\subset Q$ is {\rm $(P,C)$-good} if
\[
\ell_\sigma(\alpha)\le\ell_Q(\alpha)+C
\]
for every component $\alpha\subset P$. The set of good maps is denoted by $\text{{\bf good}}(P,C)$. 
\end{dfn}

As for good homotopies, we first need to recall the following.

\begin{dfn}[Standard Collar]
\label{dfn:standard collar}
If $\alpha\subset W$ is a simple closed geodesic for a complete hyperbolic metric $(W,\sigma)$, then the normal neighborhood of $\alpha$ of radius ${\rm arcsinh}(1/\sinh(\ell_\sigma(\alpha)/2))$ is the {\rm standard collar} of $\alpha$. We denote it by 
\[
{\bf collar}(\alpha,\sigma).
\]
\end{dfn} 

By standard hyperbolic geometry, we have the following properties.
\begin{itemize}
    \item{${\bf collar}(\alpha,\sigma)$ is always an embedded tubular neighborhood of $\alpha$ disjoint from the $\ep$-cuspidal part of $(W,\sigma)$.}
    \item{If $\alpha,\beta\subset W$ are disjoint simple closed geodesics, then ${\bf collar}(\alpha,\sigma)$ and ${\bf collar}(\beta,\sigma)$ are also disjoint.}
\end{itemize}

If $P\subset W$ is a multicurve, we denote by ${\bf collar}(P,\sigma)$ the (disjoint) union of the standard neighborhoods of the geodesic representatives of the components $\alpha\subset P$.

For the next definition we fix a Margulis constant $\ep>0$ (later we will use $\ep=\ep_1$ as chosen in Subsection~\ref{subsec:margulis}).

\begin{dfn}[Good Homotopies {\cite[Section 4]{M01}}]
\label{dfn:good homotopies}
Let $f,g:W\to Q$ be pleated maps. Let $P\subset W$ be a multicurve. We say that $f$ and $g$ admit a $K$-good homotopy with respect to $P$ if there is a homotopy $H:W\times[0,1]\to Q$ with the following properties:
\begin{enumerate}
\item{$H_0,H_1\sim f,g$ up to domain isotopy.}
\item{If $\sigma_j$ denotes the metric induced by $H_j$ for $j=0,1$ then
\[
\text{{\bf collar}}(P,\sigma_0)=\text{{\bf collar}}(P,\sigma_1)
\]
and 
\[
(W,\sigma_1)^{{\rm cusp}}_{\ep}=(W,\sigma_2)^{{\rm cusp}}_{\ep}
\]
where the two terms denote the $\ep$-cuspidal parts.}
\item{The metrics $\sigma_0,\sigma_1$ are locally $K$-bilipschitz outside $\text{{\bf collar}}(P)$ union the $\ep$-cuspidal parts.}
\item{Let $P_0\subset P$ denote the subset of components $\alpha\subset P$ such that $\ell_Q(\alpha)<\ep$. If $p\not\in\text{{\bf collar}}(P_0)$, then the track $H(p\times[0,1])$ has length bounded by $K$.}
\item{For each $\alpha\subset P_0$ the image $H(\text{{\bf collar}}(\alpha)\times[0,1])$ is contained in the $K$-neighbourhood of the Margulis tube $\mb{T}_{\ep}(\alpha)$.}
\end{enumerate}  
\end{dfn}

Minsky proves in \cite{M01} two fundamental properties of good maps. First, the Homotopy Bound Lemma says that good maps for the same pants decomposition are connected by a good homotopy.

\begin{lem}[Homotopy Bound Lemma, see Lemma 4.1 in \cite{M01}]
\label{homotopy bound}
For every $C>0$ and $\ep>0$ there exists $K>0$ only depending on $C$ and $W$ such that for every pants decomposition $P\subset W$ we have the following: If $f,g\in\text{{\bf good}}(P,C)$, then there exists a $K$-good homotopy (with Margulis parameter $\ep$) with respect to $P$ between $f$ and $g$.
\end{lem}

Second, the Halfway Surface Lemma guarantees that if two pants decompositions are related by an elementary move then good maps for the two are connected by the concatenation of two good homotopies.

\begin{lem}[Halfway Surface, see Lemma 4.2 in \cite{M01}]
\label{halfway surface}
There exists $C>0$ only depending on $W$ such that if $P_0\to P_1$ is an elementary move of pants decompositions then
\[
\text{{\bf good}}(P_0,C)\cap\text{{\bf good}}(P_1,C)\neq\emptyset.
\]
\end{lem}

Lemma \ref{homotopy bound} and Lemma \ref{halfway surface}, allow us to keep some control over how the induced hyperbolic metrics change along certain families of pleated surfaces associated with sequences of elementary moves of pants decompositions. 

Having this goal in mind, we introduce a useful definition (see Definition \ref{dfn:related_metrics} below), but first, for convenience of the reader, we recall our setup. Consider $M=H^-\cup_\Sigma H^+$ and a curve $\gamma\subset\Sigma$ satisfying the pared acylindrical assumptions. Endow $M-\gamma$ with a hyperbolic metric of finite volume and let $Q\to M-\gamma$ be the (possibly disconnected if $\gamma$ is separating) covering corresponding to $\pi_1(\Sigma-\gamma)$. It is a hyperbolic 3-manifold diffeomorphic to $(\Sigma-\gamma)\times\mb{R}$ to which we can apply the results of \cite{M01} analyzing the geometry of sequences of pleated surfaces $(\Sigma-\gamma,\sigma)\to Q$. Again, let $\ep>0$ be a Margulis constant (which will be chosen to be $\ep=\ep_1$ as in Subsection~\ref{subsec:margulis} later on).

\begin{dfn}[Related Metrics]
\label{dfn:related_metrics}
Let $A$ be a regular neighborhood of the curve $\gamma$ on the surface $\Sigma$. Two finite-volume hyperbolic metrics $(\Sigma-\gamma, \sigma)$ and $(\Sigma-\gamma,\sigma')$ are $(L,\epsilon)$-\emph{related outside $A$} if $\partial A$ lies outside the $\epsilon$-cuspidal part of both $\sigma$ and $\sigma'$, and moreover either
    \begin{enumerate}[(i)]
        \item{There exists a pants decomposition $P$ of $\Sigma-\gamma$ such that 
        \begin{itemize}
            \item{${\bf collar}(P,\sigma)={\bf collar}(P,\sigma')$ and $(\Sigma-\gamma,\sigma)^{{\rm cusp}}_{\ep}=(\Sigma-\gamma,\sigma')^{{\rm cusp}}_{\ep}$ and}
            \item{$\sigma$ and $\sigma'$ are $L$-bilipschitz outside ${\bf collar}(P,\sigma)$ union the $\ep$-cuspidal parts.}
        \end{itemize} 
        }
        \item{There exists a subsurface $Y\subset\Sigma-\gamma$ not containing $\gamma$ that has totally geodesic boundary for both $\sigma$ and $\sigma'$ and the identity $(Y,\sigma)\to(Y,\sigma')$ is an isometry.}
    \end{enumerate}  
We say that two metrics as above are $(L,\epsilon, N)$-\emph{related outside $A$} if there exists a sequence of metrics $\sigma=\sigma_1,\dots, \sigma_N=\sigma'$ where consecutive pairs are $(L,\epsilon)$-\emph{related outside $A$}.
\end{dfn}

\subsection{The proof of Proposition \ref{prop:homotopies}}
We prove the main result of this section. We enumerate by {\bf Property (1)-(5)} the various properties of the proposition that we want to obtain.
Let $\ep_0,\ep_1$ be the Margulis constants chosen in Subsection~\ref{subsec:margulis}.

\begin{proof}[Proof of Proposition \ref{prop:homotopies}]
Let $\lambda^+,\lambda^-$ be the (finite leaved) laminations of $\Sigma-\gamma$ obtained by putting $\delta^+,\delta^-$ in minimal position with respect to $\gamma$, intersecting them with $\Sigma-\gamma$, and collapsing all collections of parallel arcs to a single one. 

Complete $\lambda^+,\lambda^-$ to ideal triangulations $\tau^+,\tau^-$ of $\Sigma-\gamma$ and let $g^\pm:(\Sigma-\gamma,\sigma^\pm)\to M-\gamma$ be pleated surfaces homotopic to the inclusion $\iota:\Sigma-\gamma\to M-\gamma$ mapping geodesically those maximal laminations. We will set $f^+,f^-:=g^+,g^-|_{\Sigma-{\rm int}(A)}$ for a suitable annulus $A$ which we will choose later on. Observe that with this choice {\bf Property (1)} is satisfied by construction.

\subsubsection{The choice of the annulus}
Consider $(g^\pm)^{-1}(T=\partial\mb{T}_{\ep_1}(\gamma))$. 

Recall that, by our choices of constants (see Subsection \ref{subsec:margulis}) and Lemma \ref{lem:thick to thick}, only the $\ep_0$-thin part of $(\Sigma-\gamma,\sigma^\pm)$ can be mapped to the $\ep_1$-thin part of $M-\gamma$, so $(g^{\pm})^{-1}(T)$ is contained in the $\ep_0$-thin part. Moreover, as $g^\pm$ is 1-Lipschitz and $\pi_1$-injective, the injectivity radius of $(\Sigma-\gamma,\sigma^\pm)$ at points on $(g^{\pm})^{-1}(T)$ is bounded from below by the injectivity radius of $M-\gamma$ at points on $T$ which in turn is bounded from below by $\ep_1$. 

As $g^\pm$ is transverse to $T$ on (the closure of) each ideal triangle of $(\Sigma-\gamma)-\lambda^\pm$, the pre-image $(g^\pm)^{-1}(T)=\alpha_1\cup\cdots\cup\alpha_r$ consists of a collection of simple closed curves. As $g^\pm$ is type preserving, if $\alpha_j$ is essential, then it must be homotopic to $\gamma$ in $\Sigma-\gamma$. In particular, all the essential $\alpha_j$'s are parallel to $\gamma$. We choose $A^\pm$ to be the annulus bounded by the innermost $\alpha_j$'s (the ones that are closer to $\gamma$). Up to pre-composition with a domain isotopy we might as well assume that $A^+=A^-=:A$. 

By our choices
\begin{itemize}
    \item{$A\cup\gamma$ is a tubular neighborhood of $\gamma$ on $\Sigma$,}
    \item{$A$ is contained in the $\ep_0$-cuspidal parts of $(\Sigma-\gamma,\sigma^\pm)$,}
    \item{$A$ contains the $\ep_1$-cuspidal parts of $(\Sigma-\gamma,\sigma^\pm)$.}
\end{itemize} 

In particular, it satisfies the first requirements of {\bf Property (3)}.

Lastly, let us discuss the structure of $\partial A$. Recall that by our choice of Margulis constant (see Subsection \ref{subsec:margulis}) and Lemma \ref{lem:straight in cusp}, every lamination of $\Sigma-\gamma$ can intersect the $\ep_0$-cuspidal part only in infinite rays orthogonal to the boundary of the standard cusp neighborhood and going straight towards the cusp.

Also recall that, since $g^\pm$ is 1-Lipschitz, the image of the $\ep_0$-cuspidal part of the surface, denoted here by $U^\pm$, is contained in the $\ep_0$-Margulis neighborhood of the cusp of $M-\gamma$, denoted here by $V_0^\pm$. By construction, $A$ is contained in the $\ep_0$-cuspidal part of the surface. Therefore, in order to describe it it is enough to study the restriction 
\[
g^\pm:U^\pm\to V_0^\pm.
\]

\begin{claim}
The boundary of $A$ consists of two horocyclic loops orthogonal to $\lambda^\pm$.
\end{claim}

\begin{proof}
Let us work in the upper half-space models of $\mb{H}^2$ and $\mb{H}^3$. The universal cover of the $\ep_0$-cuspidal part $U^\pm$ of the surface is the horoball $\{z\in\mb{C}\left|\,{\rm Im}(z)\ge E\right.\}$ where $\ep_0=2\sinh(1/2E)$. The universal cover of the $\ep_0$-Margulis neighborhoods of the cusp of the 3-manifold is $\{(z,t)\in\mb{C}\times(0,\infty)\left|\,t\ge E_0\right.\}$. The lift of the (smaller) $\ep_1$-Margulis neighborhood $V_1^\pm$ to $\mb{H}^3$ is $\{(z,t)\in\mb{C}\times(0,\infty)\left|\,t\ge E_1\right.\}$ (its boundary is the lift of torus $T$ to the universal cover). 

We lift the restriction of $g^\pm:U^\pm\to V_0^\pm$ to a $\pi_1(A)$-equivariant map 
\[
g^\pm:\{z\in\mb{C}\left|\,{\rm Im}(z)\ge E\right.\}\to \{(z,t)\in\mb{C}\times(0,\infty)\left|\,t\ge E_0\right.\}.
\]
(For simplicity, with a slight abuse of notation, we use the same notation for the map and the lift.)

As described above, the lifts of $\lambda^\pm$ to the universal cover is a $\pi_1(A)$-invariant collection of vertical rays. On each region bounded by two consecutive rays, the map $g^\pm$ is an isometric embedding mapping vertical lines to vertical lines. Thus, the pre-image $(g^\pm)^{-1}(\mb{C}\times\{E_1\})$ (which covers the pre-image of the torus $T$ under $g^\pm:U^\pm\to V_0^\pm$) is a concatenation of a horospherical segments centered at the point at infinity. The claim follows.  
\end{proof}

Note that the length of $\partial A$ is smaller than the length of the boundary of $U^\pm$ which is $E=2\sinh(\ep_0/2)$, checking the last part of {\bf Property (3)}.

\subsubsection{The map on the annulus}
Next, we take care of the definition of $f^+,f^-$ on the annulus $A$.

In order to do so, we push ${\rm int}(A)$ a little bit inside $H^-,H^+$ keeping the boundary $\partial A\subset\Sigma-\gamma$ fixed. We denote by $\iota^+,\iota^-:A\to H^+,H^-$ such inclusion (they agree with $\iota$ on $\partial A$). By the homotopy extension property, we can extend the homotopy between $\iota^+,\iota^-|_{\partial A}$ and $g^+,g^-|_{\partial A}$ to a homotopy of the whole annulus $A$, that is, we find a map 
\[
h^+,h^-:A\times[0,1]\to M-\gamma
\]
that restricts to $\iota^+,\iota^-$ on $A\times\{0\}$ and to a map $g_A^+,g_A^-:A\to M-\gamma$ on $A\times\{1\}$ that satisfies $g_A^+,g_A^-|_{(\partial A)\times\{1\}}=g^+,g^-|_{\partial A}$. 

Now recall that $g^+(\partial A),g^-(\partial A)\subset T=\partial\mb{T}(\gamma)$ and that every finite volume hyperbolic 3-manifold is acylindrical. In our case, this just means that every map of pairs $(A,\partial A)\to (M-\gamma,T)$ is homotopic relative to $\partial A$ to a map in $T$. We use this fact to homotope $g_A^+,g_A^-$ to a map $f_A^+,f_A^-:A\to T$. 

Finally, we define $f^+,f^-$ to be equal to the pleated surface map on $\Sigma-{\rm int}(A)$ and to the map $f_A^+,f_A^-$ on $A$.

By construction, $f^+,f^-$ are homotopic in $M-\gamma$ to the inclusion of $\Sigma$ modified on $A$ by pushing the annulus a little bit inside the two handlebodies. We immediately deduce that for every $\delta_0^-\in\mc{D}^-,\delta_0^+\in\mc{D}^+$ the curves $f^-(\delta_0^-),f^+(\delta_0^+)$ are null-homotopic in $M-\gamma$. This says that {\bf Property (2)} is automatically satisfied.

Again by construction, $f^+(\partial A),f^-(\partial A)$ are homotopic in $M-\gamma$ to $(\Sigma-\gamma)\cap T$. As $M-\gamma$ is acylindrical, $f^+(\partial A),f^-(\partial A)$ are also homotopic to $(\Sigma-\gamma)\cap T$ in $T$. Thus, also {\bf Property (3)} holds.

\subsubsection{Controlled homotopies}
Now we begin to address {\bf Properties (4)} and {\bf (5)}. 

We lift the pleated surfaces $g^-,g^+$ to pleated surfaces in the (possibly disconnected) covering $Q\to M-\gamma$ corresponding to $\pi_1(\Sigma-\gamma)$ (the components of $Q$ correspond to the components of $\Sigma-\gamma$). Using Lemmas \ref{homotopy bound} (with parameter $\ep=\ep_1$) and \ref{halfway surface}, we will connect $g^-,g^+$ in $Q$ by a sequence of good homotopies and then project the sequence down to $M-\gamma$ to connect $g^+,g^-$ in $M-\gamma$. In order to apply Lemmas \ref{homotopy bound} and \ref{halfway surface}, we will associate with $g^-,g^+$ two pants decompositions $P^-,P^+$ with respect to which the pleated surfaces are good maps with uniform constants (in the sense of Definition \ref{dfn:good maps}). Towards this goal, we begin by observing the following.

\begin{claim}
    There exists a constant $B>0$, only depending on $\Sigma-\gamma$, and pants decompositions $P^+,P^-$ of $\Sigma-\gamma$ such that $i(P^+,\lambda^+),i(P^-,\lambda^-)\le B$.
\end{claim}

\begin{proof}
Recall that there are only finitely many ideal triangulations of $\Sigma-\gamma$ up to the action of the mapping class group. The claim follows by choosing for each of those finitely many triangulations a pants decomposition. 
\end{proof}

Consider $P^-,P^+$ as given by the claim. 

The first observation that allows us to connect $g^-,g^+:\Sigma-\gamma\to Q$ by a controlled sequence of good homotopies is that $g^-\in{\bf good}(P^-,L)$ and $g^+\in{\bf good}(P^+,L)$ for some uniform constant $L$ only depending of $\Sigma$. This is a consequence of the following theorem (see Theorem 3.5 in \cite{M00} and Theorem 3.3 in \cite{ThurstonII}).

\begin{thm}[Efficiency of Pleated Surfaces]
\label{efficiency}
For every $\ep>0$ there exists $K>0$ only depending on $\ep$ and $W$ such that the following holds: Let $f:(W,\sigma)\to N$ be a doubly incompressible pleated surface that maps geodesically a maximal lamination $\lambda$ consisting of finitely many arcs and simple closed curves. Suppose that each closed leaf of $\lambda$ has length at least $\ep$. Then, for every multicurve $\gamma$ on $W$ we have 
\[
\ell_N(\gamma)\le\ell_\sigma(\gamma)\le\ell_N(\gamma)+Ka(\gamma,\lambda)
\]
where $a(\gamma,\lambda)$ is the {\rm alternation number} between $\gamma$ and $\lambda$.
\end{thm}

We do not need the definition of the alternation number, as we will only need that in our case it is bounded above by the geometric intersection number, that is $a(\gamma,\lambda)\le i(\gamma,\lambda)$.

We apply Theorem \ref{efficiency} to the doubly incompressible pleated surfaces \[g^\pm:(\Sigma-\gamma,\sigma^\pm)\to Q,\] which map the ideal triangulations $\lambda^\pm$ geodesically. Note that the laminations $\lambda^\pm$ do not have closed leaves; hence, the last assumption of the theorem is trivially satisfied (and we can freely choose the parameter $\ep$). 

By construction $P^-,P^+$ intersect $\lambda^-,\lambda^+$ in at most $B$ points. In particular, the alternation numbers are bounded by $a(P^-,\lambda^-),a(P^+,\lambda^+)\le B$. By Theorem \ref{efficiency}, we conclude that $\ell_{\sigma^-}(P^-)\le\ell_Q(P^-)+KB$ and $\ell_{\sigma^+}(P^+)\le \ell_Q(P^+)+KB$. That is, we can conclude that $g^-\in{\bf good}(P^-,C)$ and $g^+\in{\bf good}(P^+,C)$ where $C=KB$.

We now use Proposition \ref{augmenting minsky} to join $P^-$ to $P^+$ via a sequence of elementary moves on pants decompositions and then turn this sequence of moves into a sequence of geometrically controlled homotopies (as in Definition \ref{dfn:good homotopies}). 

More precisely, Proposition \ref{augmenting minsky} gives us the following.
\begin{itemize}
    \item{Non-annular subsurfaces $Y_1,\dots, Y_k$ that do not contain $\gamma$.}
    \item{An elementary move sequence $P^-=P_0\to\ldots\to P_n=P^+$.}
    \item{Intervals $J_j=[t_i^-,t_i^+]$, with $J_i\le J_{i+1}$, satisfying the following properties. Denote
 $$S=\sum_{Y\in\mathcal Y_\gamma} \{\{d_{Y}(P^-,P^+)\}\}_C,$$
 where $\mathcal Y_\gamma$ is the set of all essential non-annular subsurfaces $Y$ containing $\gamma$. Then, for some constant $K$ depending on $\Sigma$, we have
 
\begin{enumerate}
\item For $i\in J_j$, the elementary move $P_i\to P_{i+1}$ is supported in $Y_j$.
\item{$|\{0,\dots,n\}-\bigcup J_j|\leq K S$.}
\item $k\leq K S$. 
\end{enumerate}
}
\end{itemize}

For every pants decomposition $P_t$ appearing in the sequence, we find a pleated surface $g_t:(\Sigma-\gamma,\sigma_t)\to Q$ mapping geodesically a maximal lamination obtained by adding to $P_t$ finitely many leaves spiraling around the components of $P_t$ so that the complement consists of (ideal) triangles (and with $g_0=g^-$ and $g_n=g^+$). As the elementary moves $P_t\to P_{t+1}$ for $t$ in $J_i$ occur on the subsurface $Y_i$, we can assume that the maximal extensions of $P_t$ and the maps $g_t$ agree outside $Y_i$ for every $t\in J_i$. 

As $g_t\in{\bf good}(P_t,0)$ and $P_t,P_{t+1}$ differ by an elementary move, by Lemma \ref{halfway surface}, we can find a halfway surface $g_{t+1/2}\in{\bf good}(P_t,C_1)\cap {\bf good}(P_{t+1},C_1)$ (where $C_1$ is a uniform constant only depending on $\Sigma-\gamma$). This implies (via Lemma 8.2 of \cite{M01}) the following.
\begin{itemize}
    \item{For every $t\in\bigcup{[t_j^+,t_{j+1}^-]}$, we have that $\sigma_t,\sigma_{t+1}$ are related by property (i) of Definition \ref{dfn:related_metrics} with constant $L^2$ (where $L$ only depends on $\Sigma-\gamma$).} 
    \item{For every $J_j=[t_j^-,t_j^+]$ the metrics $\sigma_{t_j^-},\sigma_{t_j^+}$ are related by property (ii) of Definition \ref{dfn:related_metrics}.}
    \item{The total number of moves is $N=|[0,n]-\bigcup{J_j}|+k$ which, by Property (3) of Proposition \ref{augmenting minsky} has cardinality bounded by $N\le 2KS$.}
\end{itemize}

Recalling that by our choices $A$ contains the $\ep_1$-cuspidal part of both $\sigma^{\pm}$, we just checked {\bf Property (5)}, that is, $\sigma^-=\sigma_0$ and $\sigma^+=\sigma_n$ are $(L^2,\ep_1,N)$-related according to Definition \ref{dfn:related_metrics}. 

We interpolate the pleated surfaces by concatenating consecutive homotopies. We subdivide the construction in two steps.

{\bf Homotopy, first step}. First consider the interval $[t_i^+,t_{i+1}^-]$. For every $t\in[t_i^+,t_{i+1}^-)$ we find a $K^2$-good homotopy between $g_t$ and $g_{t+1}$ by concatenating two $K$-good homotopies provided by Lemma \ref{homotopy bound} between $g_t$ and $g_{t+1/2}$ and between $g_{t+1/2}$ and $g_{t+1}$. 

{\bf Homotopy, second step}. Then consider $J_i=[t_i^-,t_i^+]$. As $g_t,g_{t+1}$ agree outside $Y_i$ for every $t\in J_i$, we can construct a homotopy between them which is constant outside $Y_i$ as we now explain. Consider a homotopy $H:Y_i\times[0,1]\to Q$ between the restrictions $g_t,g_{t+1}|_{Y_i}$. As $g_t,g_{t+1}$ agree on $\partial Y_i$, the restriction of $H$ to $\partial Y_i\times[0,1]$ induces a map $\partial Y_i\times(S^1=[0,1]/0\sim1)\to Q$. As $Q$ is atoroidal and $\partial Y_i$ is not homotopic into the boundary of the Margulis tube, we can homotope the restriction of $H$ to $\partial Y_i\times[0,1]$ to be constant on each $\{\star\}\times[0,1]$. By the homotopy extension property, we can extend this homotopy to the whole $Y_i\times[0,1]$. This provides us a homotopy between the restrictions of $g_t,g_{t+1}$ to $Y_i$ relative to $\partial Y_i$.

Composing the various homotopies that we obtained in the previous steps, we found a geometrically controlled map $H_0:(\Sigma-\gamma)\times[-1,1]\to Q$ joining $g^-$ to $g^+$. It has the following properties.
\begin{itemize}
    \item{$H_0(\bullet,-1)$ and $H_0(\bullet,1)$ coincide with $g^-$ and $g^+$ respectively.}
    \item{The tracks $H_0(\{\star\}\times[-1,1])$ have length bounded by $K|[0,n]-\bigcup{J_i}|\le KLS$ for every $\star\in\partial{\bf collar}(\gamma)$ and are contained in the $K$-neighborhood of the corresponding standard $\ep_1$-neighborhood of the cusp of $Q$.}
\end{itemize}

As $p:Q\to M-\gamma$ is locally isometric, the same is true for the composition $pH_0:(\Sigma-\gamma)\times[-1,1]\to M-\gamma$ which is a homotopy joining $g^-$ to $g^+$.

\subsubsection{The meridian}
Lastly, let us keep track of the meridian of $\gamma$ and check {\bf Property (4)}. Recall that:
\begin{itemize}
    \item{$h^+,h^-:[0,1]\to M-\gamma$ are homotopies between $\iota^+,\iota^-:A\to M-\gamma$ and $g^+,g^-|_A:A\to M-\gamma$.}
    \item{$pH_0:(\Sigma-\gamma)\times[-1,1]\to M-\gamma$ is the geometrically controlled homotopy between $g^+,g^-$ constructed in the previous step.}
\end{itemize}

We show that:

\begin{claim}
For every arc $t\subset A$ joining the points $\star_1,\star_2$ on the two boundary components of $A$, the concatenation of 
\[
f_A^-(t),pH_0(\{\star_1\}\times[-1,1])),f_A^+(\overline{t}),pH_0(\overline{\{\star_2\}\times[-1,1]}))
\]
represents the meridian of $\gamma$ in $M$ (recall that if $s$ is an arc we denote by $\overline{s}$ the arc with the inverse parametrization).
\end{claim}

\begin{proof}
We rewrite the loop
\[
f_A^-(t)*pH_0(\{\star_1\}\times[-1,1])*f_A^+(\overline{t})*pH_0(\overline{\{\star_2\}\times[-1,1]}))
\]
by adding and removing four arcs (see Figure \ref{fig:meridian1}, the arcs are in green)
\begin{align*}
&(h^-(\{\star_1\}\times[0,1])*f_A^-(t)*h^-(\{\star_2\}\times[0,1])^{-1})\\
&*(h^-(\{\star_2\}\times[0,1])*pH_0(\{\star_2\}\times[-1,1])*h^+(\{\star_2\}\times[0,1])^{-1})\\
&*(h^+(\{\star_2\}\times[0,1])*f_A^+(t)*h^+(\{\star_1\}\times[0,1])^{-1})\\
&*(h^+(\{\star_1\}\times[0,1])*pH_0(\{\star_2\}\times[-1,1])*h^-(\{\star_1\}\times[0,1])^{-1}).    
\end{align*}

\begin{figure}[h]
\begin{overpic}[scale=1.8]{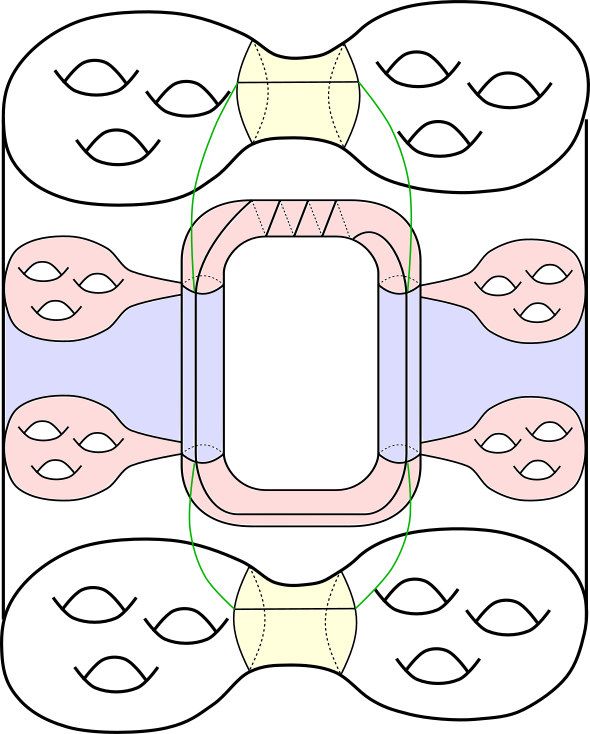}
    
\end{overpic}
\caption{Homotopy class of the meridian.}
\label{fig:meridian1}
\end{figure}

Observe that 
\[
h^-(\{\star_1\}\times[0,1])*f_A^-(t)*h^-(\overline{\{\star_2\}\times[0,1]})
\]
and
\[
h^+(\{\star_1\}\times[0,1])*f_A^+(t)*h^+(\overline{\{\star_2\}\times[0,1]})
\]
are homotopic relative to the endpoints to $\iota^-(t)$ and $\iota^+(t)$.

Next, we consider the concatenations
\[
h^-(\{\star_1\}\times[0,1])*pH_0(\{\star_1\}\times[-1,1])*h^+(\overline{\{\star_1\}\times[0,1]})
\]
and
\[
h^-(\{\star_2\}\times[0,1])*pH_0(\{\star_2\}\times[-1,1])*h^+(\overline{\{\star_2\}\times[0,1]}).
\]
we show that they are both null-homotopic.

Recall that $h^-,h^+:\partial A\times[0,1]\to M-\gamma$ coincide with the restrictions to $(\Sigma-{\rm int}(A))\times[0,1]$ of two homotopies $(\Sigma-{\rm int}(A))\times[0,1]\to M-\gamma$ between the inclusion of $\Sigma-{\rm int}(A)$ and the restrictions $f^-,f^+:\Sigma-{\rm int}(A)\to M-\gamma$. We combine the two homotopies to obtain a third homotopy $h_1:(\Sigma-{\rm int}(A))\times[-1,1]\to M-\gamma$ connecting $f^-,f^+$ and such that $h_1(\{\star_j\}\times[-1,1])=h^-(\{\star_j\}\times[0,1])*h^+(\{\star_j\}\times[0,1])$.

By covering theory, the homotopy $h_1$ lifts to a homotopy $H_1:(\Sigma-{\rm int}(A))\times[-1,1]\to Q$ between the restrictions $g^-,g^+:\Sigma-{\rm int}(A)\to Q$, that is $h_1=pH_1$. 

As the homotopies $H_0,H_1:(\Sigma-{\rm int}(A))\times[-1,1]\to Q$ are homotopic relative to the boundary, we have that the arcs $H_0(\{\star_j\}\times[-1,1])$ and $H_1(\{\star_j\}\times[-1,1])$ are homotopic relative to the boundary. By projecting the homotopy to $M-\gamma$, we conclude that the concatenation of the two arcs (in the correct order) is null-homotopic relative to the endpoints as claimed.
\end{proof}

This finishes the proof of Proposition \ref{prop:homotopies}.
\end{proof}

\section{The contribution of annular subsurfaces}
\label{sec:annuli}

In this section we prove Theorems \ref{thm:annular} and \ref{thm:annular2}, which we recall.

\annular*

\annulartwo*

We refer to Section \ref{sec:strategy} and to Figures \ref{fig:torus2} and \ref{fig:sweepout1} for a friendly outline and useful pictures. We now briefly review the main steps of the arguments.

In order to estimate the normalized length of the standard meridian $\mu\subset T$ associated with $M$ (again $T$ denotes the boundary of the standard Margulis neighborhood of the cusp of $M-\gamma$) we make use the geometric model provided by Proposition \ref{prop:homotopies}. Recall that it consists of two vertical flat segments and two horizontal ones (see Figures \ref{fig:torus2} and \ref{fig:sweepout1}) where the vertical ones have length roughly $S_\gamma$ and the horizontal ones are two parallel copies of a transverse arc $t\subset A$ suitably mapped to $T$ ($A$ is a tubular neighborhood of $\gamma$ in $\Sigma$, see Proposition \ref{prop:homotopies} property (4)).

The key is to find two different horizontal arcs of uniformly bounded length. Concatenating them with the vertical segments gives us another reference meridian with a stronger geometric control. The length bound on the standard meridian will follow from a comparison with the reference one. Additionally, the reference meridian will also give us an area bound of the form ${\rm Area}(T)\lesssim S_\gamma$. 

The first idea is that since the reference meridian and the standard meridian will only differ on the top horizontal segment (say it is $t$ in the standard and $t'$ in the reference) and that in the reference meridian $t'$ has uniformly bounded length, the length of $t$ is controlled essentially by the intersection $i(t,t')$ (minimal number of intersections between arcs in the homotopy classes with fixed endpoints of $t,t'$). 

The second idea is that such intersection is controlled by the annular projection $d_\gamma$ and the difference between the hyperbolic metric on the bottom and top pleated surfaces (as described in Proposition \ref{prop:homotopies}).

In the rest of the section, we first recall how to compute annular projections and then compare annular distances and intersection numbers of arcs in the special setup arising from the above sketch and then we carry out the proof of Theorems \ref{thm:annular} and \ref{thm:annular2}.

\subsection{Annular projections}
We will make crucial use of curve graphs of annuli, which are defined differently than the curve graphs of non-annular surfaces. We briefly recall the relevant definitions and refer the reader to~\cite[Section 2.4]{MasurMinsky:II} for details. 

\begin{dfn}[Annular Distances]
The curve graph $\mc{C}(\gamma)$ of the annulus with core curve $\gamma\subset\Sigma$ has vertices, isotopy classes of properly embedded essential arcs in the natural compactification $\overline{U}$ (homeomorphic to an annulus $S^1\times[-1,1]$) of the $\pi_1(\gamma)$-cover $U\to\Sigma$ of the surface. Two isotopy classes relative to the endpoints of essential arcs of $U$ are connected by an edge if they have representatives that can be realized disjointly. If $\tau,\tau'$ are proper essential arcs of $\overline{U}$ then their distance can be computed as
\[
d_{\mc{C}(\gamma)}(\tau,\tau')=i(\tau,\tau')+1
\]
where $i(\tau,\tau')$ is the geometric intersection number between $\tau,\tau'$ (minimal number of intersections between arcs in the homotopy classes with fixed endpoints of $\tau,\tau'$).
\end{dfn}

\begin{dfn}[Annular Projections]
\label{def:annular coeff}
Let $\gamma\subset\Sigma$ be an essential simple closed curve. Denote again by $U\to\Sigma$ the $\pi_1(\gamma)$-cover of $\Sigma$ and by $\overline{U}$ its natural compactification. Let $\alpha,\beta\subset\Sigma$ be essential simple closed curves intersecting essentially $\gamma$. Consider the pre-images $\tilde{\alpha},\tilde{\beta}\subset U$ of $\alpha,\beta$ under the covering projection. As both curves intersect $\gamma$ essentially, every connected component of $\tilde{\alpha},\tilde{\beta}$ is a proper essential arc of $U$ and can be extended to a properly embedded essential arc in $\overline{U}$ by adding two endpoints in a unique way. We set
\[
d_\gamma(\alpha,\beta):={\rm diam}(\tilde{\alpha}\cup\tilde{\beta}).
\]
\end{dfn}

\subsection{Annular distance vs intersection in annuli}

Suppose that we have an annulus $A\subset\Sigma$ with core curve $\gamma$ on a surface and two curves $\alpha,\beta\subset\Sigma$ that, for simplicity, each intersects the annulus in an essential arc $\alpha\cap A=a,\beta\cap A=b$.

In general, the geometric intersection $i(a,b)$ of the arcs $a,b\subset A$ is not related at all to their annular distance (intuitively, the ``spiraling'' could happen in an isotopic copy of the annulus). 

However, the following lemma says that, under suitable geometric constraints on the curves $\alpha,\beta$ and the annulus $A$, then the annular distance is indeed comparable to the number of intersections $i(a,b)$. More specifically, the lemma is designed specifically to deal with related metrics in the sense of Definition \ref{dfn:related_metrics}.

\begin{lem}\label{lem:annularporjecviaanuularintersecs}
For every $\ep>0,B\geq 1$ there exists $K>0$ such that the following holds. Let $\Sigma$ be a surface containing an annulus $A$ with core curve $\gamma$, and fix finite-volume hyperbolic metrics $\sigma,\sigma'$ on $\Sigma-\gamma$. Denote $Y=\Sigma-int(A)$.

Let $\alpha$ and $\beta$ be simple closed curves on $\Sigma$ respectively homotopic to concatenations 
\[
\alpha=\alpha_1*\tau_1*\dots\,{\rm and}\,\beta=\beta_1*\delta_1*\dots
\]
with the following properties.
\begin{enumerate}
 \item The $\alpha_i$ and $\beta_i$ are essential arcs in $A$,
 \item The $\tau_i$ are subgeodesics of geodesic lines of $\sigma$ that limit to the cusp in both directions,
 \item The $\delta_i$ are subgeodesics of geodesic lines of $\sigma'$ that limit to the cusp in both directions.
 \item{The hyperbolic metric $\sigma'$ is $(B,\ep)$-related outside $A$ to $\sigma$ (as in Definition \ref{dfn:related_metrics})}.
\end{enumerate}

Then 
\[
|d_\gamma(\alpha,\beta) - i(\alpha_i,\beta_j)|\leq K,
\]
for any $i,j$.
\end{lem}

\begin{figure}[h]
\begin{overpic}[scale=1.8]{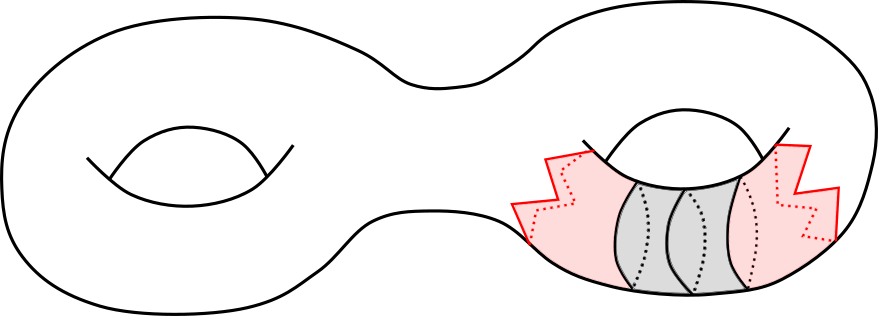}
    
\end{overpic}
\caption{The annulus $A$, containing the $\ep$-cuspidal part of the surface.}
\end{figure}

Before going into the proof let us remark that for our purposes it will be enough to consider the case where $A-\gamma$ {coincides} with the $\ep_1$-cuspidal parts of $(\Sigma-\gamma,\sigma)$ and $(\Sigma-\gamma,\sigma')$. This stronger assumption simplifies some of the steps of the proof, but we decided to keep the weak formulation for potential future applications.

\begin{proof}[Proof of Lemma \ref{lem:annularporjecviaanuularintersecs}]
Note that if two metrics on $\Sigma-\gamma$ are $(B,\ep)$-related outside $A$ then there exists a possibly disconnected subsurface $Z$ of $\Sigma$ whose components are convex for both metrics and such that the complements are $B$-bilipschitz equivalent, and isometric on $\partial Z$. We will only use these properties, together with the condition that $\partial A$ lies outside the $\ep$-cuspidal neighborhoods in both metrics, so we cover both cases of $B$-relatedness simultaneously. We fix $Z$, and note that all relevant arcs are in minimal position with respect to $Z$ by convexity.

In this proof we will need intersection numbers for proper arcs in manifolds with boundary. These are computed minimizing over the proper homotopy classes of the arcs, with any endpoint on the boundary fixed throughout the homotopy, and only counting intersections in the interior of the arcs. We still denote these intersection numbers by $i(\cdot,\cdot)$.

Consider the $\pi_1(A)$-cover $\tilde\Sigma$ of $\Sigma$, which contains a homeomorphic lift $\tilde A_0$ of $A$, itself containing a homeomorphic lift $\tilde\gamma$ of $\gamma$. The annular distance between $\alpha$ and $\beta$ is, up to a constant, the intersection number of any two lifts $\tilde\alpha$ and $\tilde \beta$ that intersect $\tilde\gamma$ essentially. Moreover, $\tilde \Sigma$ contains other lifts of $A$, all homeomorphic to closed infinite strips, as well as covers of the connected components of $Z$ as in the first paragraph of the proof (see Figure \ref{fig:annular cover}). We denote all lifts of both types by $\tilde A_k$.

\begin{figure}[h]
\begin{overpic}[scale=1.8]{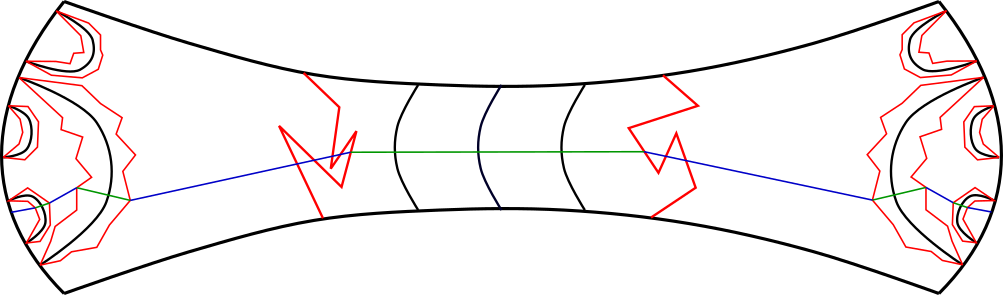}
    
\end{overpic}
\caption{The annular cover relevant to the proof of Lemma \ref{lem:annularporjecviaanuularintersecs}, with the homeomorphic lifts of the $\ep$-cuspidal part contained in the lift of $A$ in the middle. Other lifts of $A$ are also sketched, as is an essential lift of either $\alpha$ or $\beta$.}
\label{fig:annular cover}
\end{figure}

All lifts $\tilde \alpha$ and $\tilde \beta$ split into 3 parts. One part, a lift of some $\alpha_i$ or $\beta_j$, is contained in $\tilde A_0$, and the other two are rays. We denote these rays by $\tilde\tau_i$ and $\tilde\delta_i$, for $i=1,2$ (see Figure \ref{fig:annular cover}).

Denote by $\tilde \alpha_i,\tilde\beta_j$ the lifts of $\alpha_i,\beta_j$ contained in $\tilde \tau\cap \tilde A_0,\tilde\delta\cap\tilde A_0$.

\par\medskip

\begin{claim}
\label{claim:middle_arc}
Up to swapping the indices of the $\tilde\delta_i$, we have
\[
|i(\tilde \alpha,\tilde \beta)-i(\tilde \alpha_i,\tilde\beta_j)|\leq i(\tilde \tau_1,\tilde \delta_1)+i(\tilde \tau_2,\tilde \delta_2)+2.
\]
\end{claim}

\begin{proof}
We argue below that the $\tilde\tau_1$ and $\tilde\tau_2$ are each homotopic relative to the endpoint to arcs contained in a connected component of the closure of $\tilde\Sigma-\tilde A_0$. The same will hold for the $\tilde\delta_i$, and we will then arrange the indices in a way that the components for $\tilde\tau_1$ and $\tilde \delta_1$ coincide.

We have that $\tilde\alpha$ is a concatenation of lifts of the $\tau_k$ and $\alpha_k$, and in order to conclude we need to show that the only lift that intersects the core curve of $\tilde A_0$ essentially is $\tilde\alpha_i$.
Note that each lift of a $\tau_k$ connects distinct lifts of $A$ as $\tau_k$ cannot be homotoped into the annulus. In fact, since $\tau_k$ is disjoint from $\gamma$, the endpoints of any lift of $\tau_k$ connect boundary components of a lift of $A$ which are in the closure of the same connected component of $\tilde \Sigma$ minus all lifts of ${\rm int(}A)$. Lifts of the $\alpha_j$, instead, connect distinct boundary components of a lift of $A$. From these facts we see that each lift of $\gamma$ can be intersected essentially by $\tau$ at most once.

As a consequence of the argument above, up to replacing all relevant subarcs up to homotopy, we have that, for example, $\tilde\tau_1$ and $\tilde\delta_2$ do not intersect.

We can further replace all relevant subarcs of $\tilde \alpha,\tilde \beta$ with representative minimizing the number of intersections, and replace $\tilde \alpha,\tilde \beta$ with the respective concatenations. Each intersection comes with a sign given by the orientation of the surface, and for each pair of arcs all intersections have the same sign since the arcs intersect minimally. The intersection number of $\tilde \alpha,\tilde \beta$ is the absolute value of the sum of the signed intersection numbers of the various pairs of arcs, with an error of at most 2 which can arise in case that $\tilde \tau_i,\tilde \delta_i$ share an endpoint for $i=1$ and/or $i=2$. This concludes the proof.    
\end{proof}

In view of Claim \ref{claim:middle_arc}, we are left to prove the following.

\par\medskip

\begin{claim}
\label{claim:outer_arcs}
There exists a constant $B'=B'(B,\epsilon)$ such that the intersection numbers of $\tilde \tau_1$ and $\tilde \delta_1$ and those of $\tilde \tau_2$ and $\tilde \delta_2$ are bounded by $B'$.
\end{claim}

\begin{proof}
We focus on $\tilde \tau_1$ and $\tilde \delta_1$. We can change the indices of the decompositions of $\alpha$ and $\beta$ to arrange, for convenience, that $\tilde \tau_1$ contains a lift $\tilde\tau_1$ of $\tau_1$ and similarly for $\tilde\delta_1$. We denote by $\tilde \tau'_1$ the minimal initial subarc of $\tilde\tau_1$ with endpoint on some $\tilde A_k\neq \tilde A_0$, and similarly for $\tilde \delta'_1$.

Suppose first that $\tau'_1$ is not homotopic to $\delta'_1$ via paths with endpoints on the lift of $\partial A$. Then the lifts $\tilde\tau_1$ and $\tilde\delta_1$ (start at the same component of $\partial\tilde A_0$ and) end at distinct $\tilde A_k$. In view of the fact shown in Claim \ref{claim:middle_arc} that $\tilde\tau_1$ and $\tilde\delta_1$ intersect essentially each lift of $\gamma$ at most once, we see that there are two disjoint topological disks where the endpoints of $\tilde \tau_1-\tau'_1$ and $\tilde \delta_1-\delta'_1$ lie, respectively.
        
Suppose instead that $\tau'_1$ is homotopic to $\delta'_1$ via paths with endpoints on the lift of $\partial A$, so that the lifts $\tilde\tau'_1$ and $\tilde\delta'_1$ end at the same $\tilde A_k$. In this case $\tilde \tau^1-\tau'_1$ and $\tilde \delta_1-\delta'_1$ have both endpoints in a topological disk.

In either case, the intersection number of $\tilde \tau_1$ and $\tilde \delta_1$ can differ from the intersection number of $\tilde \tau'_1$ and $\tilde \delta'_1$ by at most 1, because of an argument similar to that for Claim \ref{claim:middle_arc}, plus the fact that arcs in a disk or a disjoint union of two disks can have intersection number only at most 1.

In either case, we only need to bound the intersection number rel endpoints of $\tilde\tau'_1$ and $\tilde\delta'_1$. Both arcs are contained in a component $\tilde \Sigma'$ of the complement in $\tilde \Sigma$ of the lift of $\gamma$ which is the core curve of $\tilde A_0$. Hence, we can further lift to the universal cover $\mb{U}$ of $\Sigma-\gamma$, and we wish to count the number of lifts of $\tilde\tau'_1$ that intersect a fixed lift of $\tilde\delta'_1$. We can regard $\mb{U}$ as endowed with the pullback metric of $\sigma$, which makes it isometric to $\mathbb H^2$. 
         
We have that $\mb{U}$ contains a distinguished family of horoballs $\mathcal H$, which are lifts of the $\ep$-cuspidal part of $(\Sigma-\gamma, \sigma)$ (which, by Definition \ref{dfn:related_metrics}, coincides with the $\ep$-cuspidal part of $(\Sigma-\gamma,\sigma')$). The lifts of $\tilde \tau'_1$ are contained in geodesic lines with each endpoint the limit point of a distinguished horoball. Moreover, one endpoint is in common among all lifts, say that it corresponds to the horoball $H_0$, so that all lifts are disjoint from $H_0$. The lifts of $\tilde\delta'_1$ have similar properties, except that they are $B$-bilipschitz paths rather than geodesics. The closest-point projection of any lift of $\tilde\tau'_1$ or $\tilde \delta'_i$ to the horoball $H_0$ has diameter bounded by a universal constant, say $D$, by (the well-known) Lemma \ref{lem:bilip_hor_proj} below for the given $B$. Moreover, projections of distinct lifts of $\tilde\tau'_1$ are at least $\ep$-apart (as the deck transformation mapping one to the other has translation length at least $\ep$ outside of $H_0$), so that a fixed lift of $\tilde\delta'_1$ cannot intersect more than $(D+2D_0)/\ep$ lifts of $\tilde\tau'_1$. This concludes the proof of the claim.
\end{proof}

We can conclude the proof by setting $K=2B'+2$.
\end{proof}

\begin{lem}
\label{lem:bilip_hor_proj}
For all $B$ there exists $D$ such that the following holds. Let $\eta:I\to\mb{H}^2$ be a globally $B$-bilipschitz path in $\mathbb H^2$ which lies outside the interior of a closed horoball $\mathcal H$. Then the endpoints of $\eta$ have closest-point projection to $\mc{H}$ which is at most $D$ apart.
\end{lem}

\begin{proof}
By Gromov-hyperbolicity, $\eta$ lies $C$-Hausdorff close to a geodesic connecting the same endpoints (where $C$ only depends on $B$). For $D$ sufficiently large, any geodesic connecting points that project $D$-away onto $\mathcal H$ contains a point lying deep in the horoball, meaning that the $C$-ball around said point is contained in the horoball. As $\eta$ lies outside $\mc{H}$ and the geodesic with the same endpoints stays $C$-close to it we get the desired upper bound on $D$.
\end{proof}

\subsection{Proof of Theorem~\ref{thm:annular2}}
Before launching the proof we need one last ingredient.
We use the fact that there is a constant such that a path in $\mb{H}^3$ given as a concatenation of arcs that are alternatingly flat geodesics on horospheres and geodesic arcs that are orthogonal that start and end on horospheres can only close up, if at least one of the flat segments has length at most said constant. In fact, in~\cite{FSVDehn}, we gave such a constant explicitly.
\begin{lem}[{see \cite[Lemma~3.1]{FSVDehn}}]\label{lem:2.5}
Let $\eta$ be path in $\mathbb{H}^3$ consisting of a concatenation
$\alpha_0*\beta_0*\dots*\alpha_n*\beta_n$ and horospheres $H_0,\cdots,H_{n+1}$,
where $\alpha_i$ are geodesic arcs on the horosphere $H_i$ of length 2.5, $H_i$ and $H_{i+1}$ are disjoint,
and $\beta_i$ are geodesic arcs that meet $H_i$ and $H_{i+1}$ orthogonally.
Then $\eta$ is not a loop.\qed
\end{lem}

\begin{proof}[Proof of Theorem~\ref{thm:annular2}]
We proceed by small steps.

In the entire proof we aim to estimate the normalized length 
\[
L(\mu)=\frac{{\rm Length}(\mu)}{\sqrt{{\rm Area}(T)}}
\]
of a the simple closed curve $\mu\subset T$ corresponding to the standard meridian determined by $M$ (here $T$ denotes the boundary of a standard Margulis neighborhood of the cusp of $M-\gamma$).

\subsubsection{Setup}
We fix Margulis constants $\ep_0,\ep_1$ as in Subsection \ref{subsec:margulis}, so that we can use Proposition \ref{prop:homotopies}.

We write $\mb{T}_{\ep_1}(\gamma)$ for the standard $\ep_1$-Margulis neighborhood of the cusp of $M-\gamma$ and denote by $T$ its boundary (recall that it is a flat torus).
 
For the entire proof we fix $\delta^\pm\in\mc{D}^\pm$ that realize $d_\gamma(\mathcal{D}^+,\mathcal{D}^-)$, in other words, 
\[
d_\gamma(\mathcal{D}^+,\mathcal{D}^-)=d_\gamma(\delta^+,\delta^-).
\]

Choose an annulus $A$ (tubular neighborhood of $\gamma$ in $\Sigma$), maps $f^{\pm}:\Sigma\to M-\gamma$, and pleated surfaces $g^{\pm}:(\Sigma-\gamma,\sigma^{\pm})\to M-\gamma$ as in Proposition~\ref{prop:homotopies}. 

We can parametrize $\delta^\pm$ as a concatenation of segments 
\[
\delta^\pm=a_1^\pm*b_1^\pm*\cdots* a_{n^{\pm}}^\pm*b_{n^{\pm}}^\pm,
\]
where the $b_i^-$ (resp. $b_i^+$) are as the $\tau_i$ (resp. $\delta_i$) in Lemma \ref{lem:annularporjecviaanuularintersecs} with $\sigma$ (resp. $\sigma'$) chosen as $\sigma^{-}$ (resp. $\sigma^{+}$) and the $a_i^{\pm}$ are essential segments in $A$.

Note that, since $g^{\pm}$ map the lamination $\lambda^\pm$ (made of finitely many bi-infinite geodesics) associated with subsurface projection of $\delta^{\pm}$ geodesically and the $b_i^\pm$ are subgeodesics of geodesic lines of $\lambda^\pm$ that limit to the cusp, we have that $f^\pm(b_i^\pm)$ are geodesic segments that intersect $T$ orthogonally.

We choose a representative of $f^{\pm}(\delta^{\pm})$ in $M-\gamma$ as follows. We set $\beta^{\pm}_i=f^\pm(b_i^\pm)$ and $\alpha^{\pm}_i$ are geodesic segments in $T$ obtained by homotoping $f^\pm(a_i^\pm)$ in $T$ while fixing endpoints (we straighten $f^\pm(a_i^\pm)$ in the flat metric of $T$).

Lifting this representative to $\mathbb{H}^3$, we find a closed path in $\mathbb{H}^3$. By Lemma~\ref{lem:2.5}, there exist $1\leq i\leq n_-$ and $1\leq j\leq n_+$ such that $\alpha^+_i$ and $\alpha^-$ have length at most $2.5$.
Up to cyclic permutation of the parametrizations of $\delta^{\pm}$, we may and do assume $i=j=1$.

\subsubsection{Length of the meridian vs length of  $\beta$} 

We set $t\coloneqq a_1^-$ and consider
the arc $f^+(t)\subset T$ and let $\beta$ be a geodesic representative in $T$ obtained by straightening while fixing the endpoints.

By Proposition~\ref{prop:homotopies} property (4), there are arcs $\gamma_1,\gamma_2$ in $T$ of length at most $LS_\gamma$ such that the concatenation 
\[
\beta*\gamma_1*\alpha^{-}_1*\gamma_2
\]
forms a meridian $\mu$. Hence the length of a geodesic representative in $T$ of a meridian is bounded below by
\begin{equation}\label{eq:meridianboundbyb}
    {\rm Length}(\mu)\ge\ell(\beta)-\ell(\gamma_1)-\ell(\gamma_2)-\ell(\alpha^{-}_1)\geq \ell(\beta)-2LS_\gamma-2.5.
\end{equation}

\subsubsection{Length of $\beta$ vs intersection $i(a_1^-,a_1^+)$.}

We may modify the arc $a_1^-$ slightly (by concatenating with subsegments on $\partial A$) to find a proper arc $a_1$ in $A$ with the same endpoints as $a_1^+$ such that $i(a_1,a_1^+)=i(a_1^-,a_1^+)$. Let $\beta'$ be a geodesic representative in $T$ of $f^+(a_1)$ obtained by pulling tight while fixing the endpoints.
We have
\[
|\ell(\beta')-\ell(\beta)|\leq l( f^+(\partial A))\leq l_{\sigma^+}(\partial A)\leq E,
\]
where $E:=4\sinh(\ep_0/2)$, since $\beta'$ arises from $\beta$ by changing it by arcs in $f^+(\partial A)\subset T$ (which has length bounded by $E$) and then pulling tight relative endpoints.

Consider the concatenation $\eta$ in $T$ of $\beta'$ with $\alpha^+_1$, the straightening on $T$ of $f^+(a_1^+)$. Notice that, by construction, $\eta$ is homotopic to $f^+(a_1^{-1}*a_1^+)$ and $a_1^{-1}*a_1^+$ is homotopic in $A$ to the $i$-th power of the core curve of $A$ for some $i$ with $|i-i(a_1,a_1^+)|\leq 1$. By Proposition \ref{prop:homotopies} property (3), at the level of fundamental groups, $f^+$ maps the homotopy class of the core curve of $A$ to the homotopy class of a component $\theta$ of $\Sigma-\gamma\cap T$. In particular, the length $\ell$ of the geodesic representative of $\eta$ is $i$ times the length $\ell_\theta$ of a geodesic representative of $\theta$. Since $\theta\subset T$ and it is homotopic to $f^+(\partial A)$, we have that $\ell_\theta\leq E$. As ${\rm inj}(T)\ge\ep_1$, we also have $\ell_\theta\ge 2\ep_1$. Moreover, $|\ell-\ell(\eta)|\leq \ell(f^+(a_1^+))\leq 2.5$, and $|\ell(\eta)-\ell(b')|\leq 2.5$. Putting all these together with $|\ell(b')-\ell(b)|\leq E$ we finally get
\begin{equation}\label{eq:l(b)bound}
(i(a^-_1,a_1^+)-1)2\epsilon_1-E-5\leq \ell(\beta)\leq (i(a^-_1,a_1^+)+1)E+E+5.
\end{equation}
In fact, in the rest of the proof, we only use the lower bound from \eqref{eq:l(b)bound}. The upper bound is used in the proof of Theorem~\ref{thm:annular}.

\subsubsection{Intersection $i(a_1^-,a_1^+)$ vs annular projection $d_\gamma(\delta^+,\delta^-)$}

Next we bound the intersection number $i(a_1^-,a_1^+)$ in terms of the annular projection $d_\gamma(\delta^+,\delta^-)$ by employing Proposition~\ref{prop:homotopies} property (5) and Lemma~\ref{lem:annularporjecviaanuularintersecs}.

Let $\sigma_0,\cdots,\sigma_{\lfloor LS_\gamma\rfloor}$ be metrics on $\Sigma-\gamma$ witnessing the fact that $\sigma^{\pm}$ are $(L,\epsilon_1,LS_\gamma)$ related outside $A$. In particular, $\sigma_0=\sigma^-$, $\sigma_{\lfloor LS_\gamma\rfloor}=\sigma^+$, and $\sigma_i$ and $\sigma_{i+1}$ are
$(L,\epsilon_1)$-related outside $A$.

Let $a_1^{i}*b_1^{i}\dots a^{i}_{n^{-}}*b_{n^{-}}^{i}$
be a concatenation of arcs that yields the homotopy class of $\delta^-$,
where the $b_1^{i}$ are proper arcs in $A$ and the $a_i$
are subgeodesics of geodesic lines of $\sigma_i$ that limit to the cusp in both directions. For $i=0$ we set $a_k^{0}=a_k^-$ and $b_k^{0}=b_k^-$. 

Let $K$ be the constant from Lemma~\ref{lem:annularporjecviaanuularintersecs} (applied to $\epsilon=\epsilon_1$ and $B=L$, where $L$ is the constant from Proposition~\ref{prop:homotopies}). By Lemma~\ref{lem:annularporjecviaanuularintersecs},
for all $1\leq i\leq \lfloor LS_\gamma\rfloor$,
we have
$|i(a_{1}^{i-1},a_{1}^{i})|=|d_\gamma(\delta^-,\delta^-)-i(a_{1}^{i-1},a_{1}^{i})|\leq K$.
One further application (applied for the special case of the metrics being the same), yields $|d_{\gamma}(\delta^-,\delta^+)-i(a_1^-,a_1^+)|\leq K$. Hence, we have \begin{equation}\label{eq:d-i}
|d_\gamma(\delta^-,\delta^+)-i(a_1^-,a_1^+)|\leq K(\lfloor LS_\gamma\rfloor+1).\end{equation}

Recalling that we are working under the assumption $d_\gamma(\delta^-,\delta^+)/c\geq S_\gamma$ for a $c>0$ that we may still fix as needed, we choose $c\geq 2KL+1$.
Therefore~\eqref{eq:d-i} yields $i(a_1^-,a_1^+)\geq d_\gamma(\delta^-,\delta^+)/2$. Combining with~\eqref{eq:l(b)bound}, we find \[d_\gamma(\delta^-,\delta^+)/K_1-K_1\leq \ell(\beta),\]
for some uniform $K_1$.
Combined with~\eqref{eq:meridianboundbyb}, and choosing $\epsilon$ smaller if need be,
we find 
\[
\frac{d_\gamma}{K_2}\leq \frac{d_\gamma(\delta^-,\delta^+)}{K_1}-K_1-2LS_\gamma-2.5\leq{\rm Length}(\mu),
\]
for some uniform $K_2$. Dividing by $\sqrt{\mathrm{Area}(T)}$, we find
\begin{equation}\label{eq:d/sqrt(area)<nl}\frac{d_\gamma(\delta^-,\delta^+)}{K_2\sqrt{\mathrm{Area}(T)}}\leq L(\mu).\end{equation}

\subsubsection{Area $\mathrm{Area}(T)$ vs non-annular projections $S_\gamma$}

We claim that there exists a constant $C$ such that
$\mathrm{Area}(T)\leq C S_\gamma$. For this, we consider a curve $\mu'$ that essentially consists of the same pieces as $\mu$, but replacing $f^+(a_1^-)$ by the (short) segment $f^+(a_1^+)$, which has length at most 2.5. 
Pick shortest arcs $s_1$ and $s_2$ on $f^+(\partial A)$ that connect the endpoints of $\gamma_1$ and $\gamma_2$ on $f^+(\partial A)$ with the endpoints of $\alpha_1^+$. Note that the total length of $s_1$ and $s_2$ is bounded by $E$.
With this setup, we take $\mu'$ to be the concatenation $\gamma_1*\alpha_1^-*\gamma_2*s_2*\alpha_1^+*s_1$, where all segments are oriented as needed for the concatenation to be defined.
We have 
\begin{align*}
{\rm Length}(\mu') &\leq \ell(\gamma_1)+\ell(\alpha_1^-)+\ell(\gamma_2)+\ell(\alpha_1^+)+\ell(s_1)+\ell(s_2)\\
&\leq LS_\gamma+2.5+LS_\gamma+2.5+E.
\end{align*}

Note that a geodesic representative of $\mu'$ intersects a geodesic representative of a component $\partial_1$ of $f^+(\partial A)$ once. Indeed this was the case for the meridian $\mu$ and, arguing in $\pi_1(T)$, $[\mu']$ is obtained from $[\mu]$ by composition with an element in $\pi_1(A)\subset\pi_1(T)$. Hence, we have 
\[
\mathrm{Area}(T)\leq {\rm Length}((\Sigma-\gamma)\cap T){\rm Length}(\mu')\leq E{\rm Length}(\mu')
\]
and, therefore, we find $\mathrm{Area}(T)\leq E(2LS_\gamma+5+2E)\leq CS_\gamma$ for $C\coloneqq E(2L+5+E)$.

\subsubsection{Normalized length.}
Combining the last paragraph with~\eqref{eq:d/sqrt(area)<nl}, we have 
\[
\frac{d_\gamma}{K_2\sqrt{CS_\gamma}}\leq L(\mu).
\]
This finishes the proof of Theorem~\ref{thm:annular2}.
\end{proof}

\subsection{Proof of of Theorem~\ref{thm:annular}}
The proof above also contains the last missing ingredients to prove Theorem \ref{thm:annular}.

The setup from the proof of Theorem~\ref{thm:annular2} and similar (but simpler) arguments allow to bound the length of the meridian of $M-\gamma$, which, in combination with Proposition~\ref{prop:aSgamma<Area}, yields an upper bound on the normalized length as needed to prove the lower bound. Here are the details, using the notations introduced in the above proof.
\begin{proof}[Proof of Theorem~\ref{thm:annular}]
We provide an upper bound for ${\rm Length}(m)$ in terms of $S_\gamma$ and $d_\gamma=d_\gamma(\delta^+,\delta^-)$, where $\mu\subset T$ is a meridian of $\gamma$. For this we use the setup from the proof of Theorem~\ref{thm:annular2}, in particular, recall that $f^-(a_1^-)\subseteq T$ has length at most $2.5$, $\gamma_1$ and $\gamma_2$ are arcs of length at most $LS_\gamma$, and $\beta$ is a geodesic representative of $f^+(a_1^-)$ such that their concatenation forms a meridian. In particular, we have
${\rm Length}(\mu)\leq \ell(\beta)+2.5+2LS_\gamma$.

Combined with~\eqref{eq:d-i} and~\eqref{eq:l(b)bound}, we find
\begin{align*}
{\rm Length}(\mu) &\leq (d_\gamma+K(\lfloor LS_\gamma\rfloor+1)+1) (C_1+2.5)+C_1\\
&\leq
C(d_\gamma+S_\gamma),
\end{align*}
for a sufficiently large constant $C$.
Together with Proposition \ref{prop:aSgamma<Area}, we get the following for the normalized length of the meridian.
\[
L(\mu)^2=\frac{{\rm Length}(\mu)^2}{\mathrm{Area(T)}}
\leq \frac{C^2(d_\gamma+S_\gamma)^2}{a\epsilon_1^2S_\gamma}\leq \frac{2C^2}{a\epsilon_1^2}\frac{d_\gamma^2+S_\gamma^2}{S_\gamma}.\qedhere\]

This finishes the proof of Theorem \ref{thm:annular}.
\end{proof}

\appendix

\section{Uniform injectivity and connectedness}
\label{appendixA}

Here we give a proof of Proposition \ref{prop:connected component} which we recall.

\connectedcomp*

\begin{proof}
As in the proof of Theorem \ref{uniforminj}, one proceeds by contradiction.

Consider a sequence of counterexamples, that is:
\begin{itemize}
    \item{A sequence $\gamma_n\subset\Sigma$ of essential multicurves such that $(H,\gamma_n)$ is pared acylindrical. We denote by $W_n:=\Sigma-\gamma_n$ the complementary surfaces.}
    \item{A sequence of hyperbolic structures $N_n$ on $H$ which have rank one cusps precisely at $\gamma_n$.}
    \item{A sequence of pleated surfaces $g_n:(W_n,\sigma_n)\to N_n$ properly homotopic to the inclusions $W_n\subset H-\gamma_n$ and mapping geodesically laminations $\lambda_n$.}
    \item{A sequence of pairs of points $x_n,y_n\in\lambda_n$ lying in the $\ep_0$-thick part of different connected components $X_n,Y_n$ of $W_n$ such that $d_{{\bf P}(N_n)}({\bf p}_{g_n}(x_n),{\bf p}_{g_n}(y_n))\to 0$ as $n\to\infty$.}
\end{itemize}

Observe that as $x_n\in X_n$ (resp. $y_n\in Y_n$) lies in the $\ep_0$-thick part of $(X_n,\sigma_n)$ (resp. $(Y_n,\sigma_n)$) we can find a pair of geodesic loops based at $x_n$ (resp. $y_n$) of length uniformly bounded only in terms of $\Sigma$ and $\ep_0$ that generate a non-abelian free subgroup of $\pi_1(X_n,x_n)$ (resp. $\pi_1(Y_n,y_n)$). Since the inclusions $X_n,Y_n\to H$ are $\pi_1$-injective, the images of those loops generate a non-abelian free subgroup of $\pi_1(N_n,x_n)$ (resp. $\pi_1(N_n,y_n)$), and since $g_n$ is 1-Lipschitz the free subgroups are generated by homotopy classes of loops of bounded length. In particular, by well-known consequences of the structure of the thin part of hyperbolic 3-manifolds, the points $g_n(x_n)$ (resp. $g_n(y_n)$) cannot lie too deep into the thin part, that is, there exists $\eta>0$ (only depending on the length bounds on the generators of the free non-abelian subgroup) such that ${\rm inj}_{g_n(x_n)}(N_n)\ge\eta$ (resp. ${\rm inj}_{g_n(y_n)}(N_n)\ge\eta$). Lastly, before moving on to the next step, also observe that $d_{{\bf P}(N_n)}({\bf p}_{g_n}(x_n),{\bf p}_{g_n}(y_n))\to 0$ implies that $d_{N_n}(g_n(x_n),g_n(y_n))\to 0$.

The injectivity radius bounds discussed above together with the fact that $g_n$ is 1-Lipschitz allow us to take geometric limits (up to passing to subsequences). More precisely:
\begin{itemize}
    \item{The sequences of pointed hyperbolic surfaces $(X_n,x_n),(Y_n,y_n)$ converge in the geometric topology to pointed hyperbolic surfaces $(X,x),(Y,y)$. In particular this means that fixed an exhaustion of $X,Y$ by metric balls of increasing radius $B(x,2n),B(y,2n)$ there are pointed smooth $(1+\xi_n)$-bilipschitz embeddings $\phi_n:(B(x,n),x)\to (X_n,x_n)$ and $\psi_n:(B(y,n),y)\to(Y_n,y_n)$ with $\xi_n\to0$.} 
    \item{The sequence of pointed hyperbolic 3-manifolds $(N_n,g_n(x_n)\simeq g_n(y_n))$ converges in the geometric topology to a pointed hyperbolic 3-manifold $(N,z)$. In particular this means that fixed an exhaustion of $N$ by metric balls $B(z,2n)$ there are pointed smooth $(1+\xi_n)$-bilipschitz embeddings $\tau_n:(B(z,n),z)\to (N_n,g_n(x_n)\simeq g_n(y_n))$ with $\xi_n\to0$.}
    \item{The sequences of (base point preserving) pleated maps $g_n:(X_n,x_n),(Y_n,y_n)\to (N_n,g_n(x_n)\simeq g_n(y_n))$ mapping geodesically $\lambda_n\cap X_n,\lambda_n\cap Y_n$ converge to (base point preserving) pleated maps $g:(X,x),(Y,y)\to (N,z)$ mapping geodesically laminations $\lambda_X,\lambda_Y$ of $X,Y$ respectively containing the basepoints $x\in\lambda_X,y\in\lambda_Y$. These maps have the property that the following diagram commutes up to (proper) homotopies
    \[
    \xymatrix{
    X,Y\ar[r]^g\ar[d]_{\phi_n,\psi_n} &N\ar[d]^{\tau_n}\\
    X_n,Y_n\ar[r]_{g_n} &N_n.
    }
    \]
    }
\end{itemize}

(See \cite[Chapter E]{BP92} for a comprehensive description of the geometric topology and \cite[Chapter I.5]{CEG:notes_on_notes} for the necessary convergence statements used above.)

By \cite[Lemma 5.10]{Thu86}, the limit pleated surfaces $g:X,Y\to N$ being limits of doubly incompressible pleated surfaces are weakly doubly incompressible. In particular, by Theorem \ref{thm:weakly embeds}, the canonical lifts ${\bf p}_g:\lambda_X,\lambda_Y\to{\bf P}(N)$ are embeddings.

By continuity, as $d_{{\bf P}(N_n)}({\bf p}_{g_n}(x_n),{\bf p}_{g_n}(y_n))\to 0$, in the limit we have ${\bf p}_g(x)={\bf p}_g(y)$ in ${\bf P}(N)$, that is, the leaf $\ell_y$ of $\lambda_Y$ containing $y$ and the leaf $\ell_x$ of $\lambda_X$ containing $x$ are mapped by $g$ onto the same complete geodesic $g(\ell_y)=g(\ell_x)$ of $N$.

Let us look at the behavior of $\ell=g(\ell_x)=g(\ell_y)$. There are two possibilities:
\begin{enumerate}
    \item{The geodesic $\ell$ intersects the $\ep_1$-thick part of $N$ infinitely many times (here $\ep_1$ is as in Lemma \ref{lem:thick to thick}).}
    \item{There is an interval $I\subset\ell$ such that $\ell-I$ is contained in the $\ep_1$-cuspidal part of $N$.}
\end{enumerate}

We treat the two cases separately.

{\bf Case (1)}. In this situation, recall that ${\bf p}_g:\lambda_X,\lambda_Y\to{\bf P}(N)$ is an embedding. In particular it is a uniform embedding on $\lambda_X^0=\lambda_X\cap X_0$ and $\lambda_Y^0=\lambda_Y\cap Y_0$ where $X_0,Y_0$ are the $\ep_0$-non-cuspidal parts of $X,Y$. More precisely, this means the following: For every $\delta>0$ there exists $\ep>0$ such that if $a,b\in \lambda_X^0$ (resp. $a,b\in\lambda_Y^0$) and $d_{{\bf P}(N)}({\bf p}_g(a),{\bf p}_g(b))<\delta$ then $d_X(a,b)<\ep$ (resp. $d_Y(a,b)<\ep$). In particular, there exists $\delta_1<\ep_1$ such that $d_{{\bf P}(N)}({\bf p}_g(a),{\bf p}_g(b))<\delta_1$ implies $d_X(a,b)<\ep_1$ or $d_Y(a,b)<\ep_1$.

As $\ell$ comes back to a fixed compact subset of the $\ep_1$-thick part of $N$, we can find an interval $I=[a,b]\subset\ell$ with endpoints in the $\ep_1$-thick part and such that the distance between the directions of $\ell$ at $a,b$ in ${\bf P}(N)$ are smaller than $\delta_0$. Let $I_x=[a_x,b_x]\subset\ell_x$ and $I_y=[a_y,b_y]\subset\ell_y$ be the intervals such that $g(I_x)=g(I_y)=I$. By our choice of $\delta_1$ and the fact that $d_{{\bf P}(N)}({\bf p}_g(a_x),{\bf p}_g(b_x)),d_{{\bf P}(N)}({\bf p}_g(a_y),{\bf p}_g(b_y))<\delta_1$, we have that $d_X(a_x,b_x),d_Y(a_y,b_y)<\ep_1$.

As $a,b$ lie in the $\ep_1$-thick part of $N$ we conclude that the shortest geodesic segment $\kappa$ joining them in $N$ (which has length smaller than $\delta_1<\ep_1$) is contained in an embedded geodesic ball of radius $\ep_1$. Similarly, also the images $g(\kappa_x),g(\kappa_y)$ of the shortest geodesic segments $\kappa_x\subset X,\kappa_y\subset Y$ joining $a_x,b_x$ and $a_y,b_y$ in $X,Y$ (respectively), being shorter than $\ep_1$ are contained in the same embedded geodesic ball of radius $\ep_1$. In particular $g(\kappa_x),g(\kappa_y)$ and $\kappa$ are homotopic relative to the endpoints within the ball. Hence, we found two essential closed curves $\alpha_x=I_x\cup\kappa_x\subset X$ and $\alpha_y=I_y\cup\kappa_y\subset Y$ such that $g(\alpha_x)$ and $g(\alpha_y)$ are homotopic in $N$.

Observe that both $g(\alpha_x),g(\alpha_y)$ are homotopic to $I\cup\kappa$ which is almost a closed geodesic in the sense that if we take the arc-length parametrization $\theta:\mb{R}\to I\cup\kappa$ and we lift it to a map to the universal cover ${\hat \theta}:\mb{R}\to\mb{H}^3$, then the map ${\hat \theta}$ is a local $(A,B)$-quasi-geodesic with constants which are as close as we want to $(1,0)$ provided that $\delta_1$ is sufficiently small. By the classical stability of quasi-geodesics in the hyperbolic space, we can assume that ${\hat \theta}$ is as close as we want to a bi-infinite geodesic in the Hausdorff topology. In particular $I\cup\kappa$ represents a non-trivial hyperbolic element of $\pi_1(N)$.

With reference to the diagram above provided by geometric convergence, we conclude that $g_n(\phi_n(\alpha_x))$ and $g_n(\psi_n(\alpha_y))$ are homotopic in $N_n$ to $\tau_n(I\cup\kappa)$ and, hence, homotopic to each other. However, recall that $g_n:X_n,Y_n\to N_n$ are properly homotopic to the inclusions $\iota_n:X_n,Y_n\to H-\gamma_n$. Therefore the essential closed curves $\phi_n(\alpha_x)\subset X_n$ and $\psi_n(\alpha_y)\subset Y_n$ are homotopic in $H-\gamma_n$ contradicting the pared acylindrical assumption on $(H,\gamma_n)$. This finishes the proof of Case (1).

{\bf Case (2)}. In this case, we can decompose $\ell$ as $\ell=J^-\cup I\cup J^+$ where $J^-,J^+$ are half-lines contained in the $\ep_1$-cuspidal part of $N$ and $I$ is a compact interval. 

As the maps $g:X,Y\to N$ are proper, we can assume that $I$ is the image under $g$ of intervals $I_x\subset\ell_x$ and $I_y\subset\ell_y$ whose endpoints are contained in the $\ep_0$-cuspidal part of $X$ and $Y$ respectively.

Again, referring to the diagram provided by geometric convergence, we conclude that $g_n(\phi_n(I_x))$ and $g_n(\psi_n(I_y))$ are homotopic in $N_n$ to $\tau_n(I)$ via a homotopy supported on a neighborhood of $\tau_n(I)$ which is as small as we want (provided that $n$ is large enough). Note that directly from geometric convergence, we can assume that $\phi_n(I_x),\psi_n(I_y)$ have their endpoints in the $\ep_0$-thin part of $X_n,Y_n$ respectively (but not necessarily in the $\ep_0$-cuspidal part). Let us analyze further the possible configurations. Let us denote by $a,b$ the endpoints of $I$ and by $a_x,b_x$ and $a_y,b_y$ the corresponding endpoints of $I_x$ and $I_y$ respectively.

If $\tau_n(a)$ or $\tau_n(b)$, say $\tau_n(a)$, lie in a $\ep_1/10$-Margulis neighborhood of a closed geodesic of $N_n$ where $\ep_1$ is the constant of Lemma \ref{lem:thick to thick} (only depending on $\Sigma$ and $\ep_0$), then both $g_n(\phi_n(a_x)),g_n(\psi_n(a_y))$ lie in the $\ep_1$-Margulis neighborhood $\mb{T}_n$ of the same geodesic for $n$ large enough. As only the $\ep_0$-thin parts of $X_n,Y_n$ can be mapped to the $\ep_1$-thin part of $N_n$, we conclude that there are two essential simple geodesic loops $\alpha_x,\alpha_y$ based at $\phi_n(a_x),\psi_n(a_y)$ of length smaller than $2\ep_0$ that are mapped to $\mb{T}_n$. Since maximal cyclic subgroups of $\pi_1(X_n),\pi_1(Y_n)$ (such as $\langle\alpha_x\rangle,\langle\alpha_y\rangle$) are mapped to maximal cyclic subgroups of $\pi_1(N_n)$ by the pared acylindrical assumption on $(H,\gamma_n)$ (that implies that the inclusions $\iota_n:X_n,Y_n\to H-\gamma_n$ are doubly incompressible) and the fact that $\pi_1(\mb{T}_n)$ is a maximal cyclic subgroup of $\pi_1(N_n)$ (by the structure of the thin part of a hyperbolic 3-manifold), we deduce that $g_n(\alpha_x)$ and $g_n(\alpha_y)$ are homotopic up to passing to the inverses. As $g_n:X_n,Y_n\to N_n$ is properly homotopic to the inclusions $\iota_n:X_n,Y_n\to H-\gamma_n$, we obtained again a contradiction with the assumption that $(H,\gamma_n)$ is pared acylindrical.  

Hence we can assume that both $\tau_n(a)$ and $\tau_n(b)$ are contained in the $\ep_1/10$-Margulis neighborhoods of the cusps of $N_n$ (possibly the same cusp) which, by assumption, correspond precisely to the components of $\gamma_n$. As only the cusps of $X_n,Y_n$ are mapped to the cusps of $N_n$ by $g_n$ we must have that $\phi_n(a_x)$ and $\psi_n(a_y)$ are contained in an annulus $A_n(a)\subset\Sigma$ around the same component of $\gamma_n$ (the component is the one corresponding to the cusp containing $\tau_n(a)$ and the annulus is the union of the two $\ep_0$-cuspidal neighborhoods of such cusp in $X_n,Y_n$ respectively). Similarly, also $\phi_n(b_x)$ and $\psi_n(b_y)$ are contained in an annulus $A_n(b)$ around the same component of $\gamma_n$. 

Recall that $g_n(\phi_n(I_x))$ and $g_n(\psi_n(I_y))$ are homotopic in $N_n$ via a homotopy supported on a (very small) neighborhood of $\tau_n(I)$. As $g_n:X_n,Y_n\to N_n$ are properly homotopic to the inclusions $\iota_n:X_n,Y_n\to H-\gamma_n$, the arcs $\phi_n(I_x),\psi_n(I_y)\subset X_n,Y_n$ are homotopic in $H-\gamma_n$ via a homotopy that keeps the endpoints in collar neighborhoods of $A_n(a),A_n(b)$ in $H$ respectively. This again contradicts the pared acylindrical assumption on $(H,\gamma_n)$. 

This finishes the discussion of Case (2).

Together, Cases (1) and (2) complete the proof of Proposition \ref{prop:connected component}.
\end{proof}

\bibliographystyle{alpha}
\bibliography{bibliography}

\end{document}